\documentclass[a4paper]{amsart}
\usepackage{amsmath,amssymb,times, amscd, graphicx, xspace,mathrsfs,a4wide,hyperref,url,chngcntr}
\usepackage[title]{appendix}
\hypersetup{
  colorlinks   = true, 
  urlcolor     = blue, 
  linkcolor    = blue, 
  citecolor   = blue 
}

\newcommand{\Z}{\ensuremath{\mathbb{Z}}\xspace}
\newcommand{\Q}{\ensuremath{\mathbb{Q}}\xspace}
\newcommand{\R}{\ensuremath{\mathbb{R}}\xspace}

\newcommand{\A}{\ensuremath{\mathbb{A}}\xspace}
\newcommand{\F}{\ensuremath{\mathbb{F}}\xspace}
\newcommand{\Qp}{\ensuremath{\mathbb{Q}_{p}}\xspace}
\newcommand{\Zp}{\ensuremath{\mathbb{Z}_{p}}\xspace}
\newcommand{\Fp}{\ensuremath{\mathbb{F}_{p}}\xspace}
\newcommand{\Sp}{\mathrm{Sp}\xspace}
\newcommand{\D}{\mathcal{D}}

\newcommand{\oQ}{\ensuremath{\overline{\mathbb{Q}}}\xspace}
\newcommand{\bG}{\ensuremath{\mathbb{G}}\xspace}

\renewcommand{\b}{\ensuremath{\mathfrak{b}}\xspace}

\newcommand{\m}{\ensuremath{\mathfrak{m}}\xspace}

\newcommand{\n}{\ensuremath{\mathfrak{n}}\xspace}

\newcommand{\OO}{\ensuremath{\mathcal{O}}\xspace}

\newcommand{\comment}[1]{}

\DeclareMathOperator{\Aut}{Aut}

\DeclareMathOperator{\Gal}{Gal}
\DeclareMathOperator{\End}{End}
\DeclareMathOperator{\Hom}{Hom}
\DeclareMathOperator{\MCM}{MCM}
\DeclareMathOperator{\RHom}{RHom}
\DeclareMathOperator{\Homi}{\underline{Hom}}
\DeclareMathOperator{\RHomi}{R\underline{Hom}}

\DeclareMathOperator{\Sym}{Sym}
\DeclareMathOperator{\Spec}{Spec}
\DeclareMathOperator{\Spf}{Spf}
\DeclareMathOperator{\Proj}{Proj}

\DeclareMathOperator{\Tor}{Tor}

\DeclareMathOperator{\Ker}{Ker}
\DeclareMathOperator{\Coker}{Coker}

\DeclareMathOperator{\Ext}{Ext}
\DeclareMathOperator{\Exti}{\underline{Ext}}

\DeclareMathOperator{\Rep}{Rep}
\DeclareMathOperator{\Ind}{Ind}
\DeclareMathOperator{\Coh}{Coh}
\DeclareMathOperator{\IndCoh}{IndCoh}
\DeclareMathOperator{\ProCoh}{ProCoh}
\DeclareMathOperator{\QCoh}{QCoh}
\DeclareMathOperator{\RMod}{RMod}
\DeclareMathOperator{\LMod}{LMod}
\DeclareMathOperator{\Perf}{Perf}
\DeclareMathOperator{\Pro}{Pro}
\DeclareMathOperator{\Ch}{Ch}

\newcommand{\T}{\ensuremath{\mathbb{T}}\xspace}

\newcommand{\GL}{\ensuremath{\mathrm{GL}}\xspace}
\newcommand{\PGL}{\ensuremath{\mathrm{PGL}}\xspace}
\newcommand{\SL}{\ensuremath{\mathrm{SL}}\xspace}

\newcommand{\G}{\ensuremath{\mathbb{G}}\xspace}

\newcommand{\fX}{\mathfrak{X}}

\newcommand{\cF}{\mathcal{F}}

\newcommand{\pf}{\ensuremath{\mathfrak{p}}\xspace}

\newcommand{\mbf}{\mathbf}
\newcommand{\mb}{\mathbb}
\newcommand{\mc}{\mathcal}
\newcommand{\ms}{\mathscr}
\newcommand{\mf}{\mathfrak}
\newcommand{\vp}{\varpi}

\newcommand{\sub}{\subseteq}
\newcommand{\oo}{\mathcal{O}}
\newcommand{\ol}{\overline}
\newcommand{\bu}{\bullet}
\newcommand{\ka}{\kappa}

\newcommand{\wh}{\widehat}
\newcommand{\wt}{\widetilde}

\newcommand{\Def}{\mathrm{Def}}

\newcommand{\Mod}{\mathrm{Mod}}
\newcommand{\vc}{\check{\mathbf{V}}}
\newcommand{\V}{\mathbf{V}}
\newcommand{\Fpbar}{\overline{\mathbb{F}}_p}
\newcommand{\ap}{a_1^\prime}

\newcommand{\PsR}{\mathrm{PsR}}
\newcommand{\lb}{[\![}
\newcommand{\rb}{]\!]}
\newcommand{\Ad}{\mathrm{Ad}}
\newcommand{\ad}{\mathrm{ad}}
\newcommand{\isoto}{{\buildrel\sim\over\rightarrow}}
\newcommand{\varep}{\varepsilon}

\newtheorem{theorem}{Theorem}[subsection]
\newtheorem{proposition}[theorem]{Proposition}
\newtheorem{corollary}[theorem]{Corollary}
\newtheorem{lemma}[theorem]{Lemma}

\theoremstyle{definition}
\newtheorem{definition}[theorem]{Definition}

\newtheorem{notation}[theorem]{Notation}
\newtheorem{remark}[theorem]{Remark}
\newtheorem{assumption}[theorem]{Assumption}

\mathchardef\mhyphen="2D
\title{Moduli Stacks of Galois representations and the $p$-adic Local Langlands correspondence for $\GL_2(\Qp)$}
\author{Christian Johansson, James Newton, and Carl Wang-Erickson}

\address{Department of Mathematical Sciences, Chalmers University of Technology and the University of Gothenburg, 412 96 Gothenburg, Sweden}
\email{chrjohv@chalmers.se}
\address{Mathematical Institute, Woodstock Road, Oxford OX2 6GG, UK}
\email{newton@maths.ox.ac.uk}
\address{Department of Mathematics, University of Pittsburgh, Pittsburgh, PA 15260, USA}
\email{carl.wang-erickson@pitt.edu}
\input xy
\xyoption{all}
\usepackage{amscd}

\DeclareMathOperator{\im}{im}

\newcommand{\sm}[4]{\ensuremath{\big(\begin{smallmatrix}#1 & #2 \\ #3 & #4\end{smallmatrix}\big)}}

\begin{document}

\begin{abstract}
We give a categorical formulation of the $p$-adic local Langlands correspondence for $\GL_2(\Qp)$, as an embedding of the derived category of locally admissible representations into the category of Ind-coherent sheaves on the moduli stack of two-dimensional representations of $\Gal(\ol{\Q}_p/\Qp)$. Moreover, we relate our version of the $p$-adic local Langlands correspondence for $\GL_2(\Qp)$ to the cohomology of modular curves through a local-global compatibility formula.
\end{abstract}
\maketitle

\tableofcontents

\counterwithin{equation}{subsection}
\section{Introduction} The main goal of this paper is to give a categorical formulation of the $p$-adic local Langlands correspondence for $\GL_2(\Qp)$, in the spirit of the geometric Langlands program. Moreover, we relate our version of the $p$-adic local Langlands correspondence for $\GL_2(\Qp)$ to the cohomology of modular curves through a `local-global compatibility' formula. Throughout the paper, we let $p$ be a prime number and assume $p \ge 5$. 

\subsection{Local results}
\label{subsec: intro functor}
Before describing our results in detail, let us give some context for the shape of our results. Let $G$ be a connected reductive group over the global function field $F$ of a curve $X$, which we assume to be split for simplicity. Roughly speaking, the geometric Langlands program proposes a link between the quasicoherent sheaf theory on the moduli stack $\mf{X}_{\wh{G}}$ of $\wh{G}$-local systems on $X$ (the stack of Langlands parameters) and the `constructible' sheaf theory of the moduli stack $\mathrm{Bun}_G$ of $G$-torsors on $X$. Replacing $F$ by a nonarchimedean local field (of mixed or equal characteristic), these ideas have been transposed to the setting of the local Langlands correspondence in recent work of Fargues--Scholze \cite{fargues-scholze}, with $\mathrm{Bun}_G$ the stack of $G$-torsors on the Fargues--Fontaine curve\footnote{There is also the work of Zhu \cite{zhu-coherent,zhu-tame}, which instead uses the stack of $G$-isocrystals.}.

A consequence of the main conjecture in \cite{fargues-scholze}, which was conjectured independently by Hellmann \cite{hellmann-derived} and Ben-Zvi--Chen--Helm--Nadler \cite{bzchn} (who also proved it for $G=\GL_n$), is the existence of a fully faithful embedding
\begin{equation}\label{eq: main intro}
\D_{sm}(G) \to \IndCoh(\mf{X}_{\wh{G}}),
\end{equation}
where $\D_{sm}(G)$ is the ($\infty$-categorical) unbounded derived category of smooth $G(F)$-representations and $\IndCoh(\mf{X}_{\wh{G}})$ is the ind-completion of the bounded derived category $\D_{coh}^b(\mf{X}_{\wh{G}})$ of coherent sheaves on the moduli stack $\mf{X}_{\wh{G}}$ of $\wh{G}$-valued Weil--Deligne representations.

The main theorem of this paper is a version of the embedding (\ref{eq: main intro}) in the context of the $p$-adic local Langlands correspondence for $\GL_2(\Qp)$. To state it precisely, we need some more notation. Let $G=\GL_2(\Qp)$. We fix a finite extension $L/\Qp$ (which we think of as large) and let $\oo=\oo_L$ be its ring of integers with residue field $\F$. Furthermore, we fix a smooth character $\zeta : \Qp^\times \to \oo^\times$ and consider the abelian category $\Mod_{G,\zeta}^{lfin}(\oo)$ of smooth and locally finite (or, equivalently, locally admissible) representations of $G$ on $\oo$-modules, with central character $\zeta$. We let $\mf{X}_{\zeta\varep}$ denote the algebraized moduli stack of two-dimensional continuous representations of $\Gamma_{\Qp}:=\mathrm{Gal}(\ol{\Q}_p/\Qp)$ over $\oo$ with fixed determinant $\zeta \varep$ (where $\varep$ is the $p$-adic cyclotomic character); we refer to \S \ref{chap: stacks of Galois reps} for the precise definitions. Our main theorem is the following:

\begin{theorem}\label{intro main embedding}
There exists a fully faithful embedding $\D(\Mod_{G,\zeta}^{lfin}(\oo)) \to \IndCoh(\mf{X}_{\zeta\varep})$.
\end{theorem}

We also prove a `dual' version; we refer to \S \ref{subsec: main thm} for the precise statement of the main theorem and its dual version. Our results are related to conjectures discussed in \cite{egh} (with proofs announced in the case of $\GL_2(\Qp)$ \cite{dotto-emerton-gee}); see \S\ref{ssec:motives} for a discussion about the relation with \cite{egh}.

\subsection{The proof of local results} 
\label{subsec: intro proofs}
We will now give an outline of the proof, which is Morita-theoretic. The category $\Mod_{G,\zeta}^{lfin}(\oo)$ has been computed explicitly by Pa{\v s}k{\=u}nas \cite{paskunas-image}. In particular, it has a block decomposition
\[
\Mod_{G,\zeta}^{lfin}(\oo) = \prod_{\mf{B}} \Mod_{G,\zeta}^{lfin}(\oo)_{\mf{B}}
\]
and the blocks $\mf{B}$ are in bijection with $\Gal(\overline{\F}/\F)$-orbits of two-dimensional semisimple $\Gamma_{\Qp}$-representations over $\overline{\F}$ with determinant $\zeta\varepsilon$; we choose a representative $\rho_{\mf{B}}$ with minimal field of definition. Explicitly, there are four types of blocks containing absolutely irreducible representations\footnote{The remaining blocks can be handled by extending the coefficient field $L$.}:  
\begin{enumerate}
\item $\mf{B} = \{ \pi \}$, where $\pi$ is supersingular;

\item $\mf{B} = \{ \Ind_B^G (\delta_1 \otimes \delta_2 \omega^{-1}), \Ind_B^G (\delta_2 \otimes \delta_1 \omega^{-1}) \}$ with $\delta_2 \delta_1^{-1} \neq \mbf{1}, \omega^{\pm 1}$;

\item $\mf{B} = \{ \Ind_B^G (\delta \otimes \delta \omega^{-1}) \}$;

\item $\mf{B} = \{ \delta \circ \det, \mathrm{St} \otimes (\delta \circ \det), \Ind_B^G (\delta \omega \otimes \delta \omega^{-1}) \}$,
\end{enumerate}
where $\omega$ denotes the modulo $p$ cyclotomic character (and the corresponding character of $\Qp^\times$ under Artin reciprocity). Following \cite{paskunas-image}, we will refer to (1) as the \emph{supersingular} blocks, (2) as the \emph{generic principal series} blocks, and cases (3) and (4) as the \emph{non-generic} blocks, where (3) will be labelled as `non-generic case I' and (4) as `non-generic case II'. The $\rho_\mf{B}$ are in bijection with the connected components of $\mf{X}_{\zeta\varepsilon}$, and hence give a decomposition
\[
\mf{X}_{\zeta\varepsilon} = \bigsqcup_{\mf{B}} \mf{X}_{\mf{B}},
\]
which induces a decomposition $\IndCoh(\mf{X}_{\zeta\varepsilon}) = \prod_{\mf{B}} \IndCoh(\mf{X}_{\mf{B}})$. Thus, we may construct the functor block by block. 

Each $\Mod_{G,\zeta}^{lfin}(\oo)_{\mf{B}}$  has an injective generator $I_{\mf{B}}$ and $E_{\mf{B}} := \End_G(I_{\mf{B}})^{op}$ is a compact ring. The theory of locally finite categories \cite{gabriel} gives an equivalence
\[
\Mod_{G,\zeta}^{lfin}(\oo)_{\mf{B}} \cong \LMod_{disc}(E_{\mf{B}})
\]
between $\Mod_{G,\zeta}^{lfin}(\oo)_{\mf{B}}$ and the category $\LMod_{disc}(E_{\mf{B}})$ of discrete left $E_{\mf{B}}$-modules. The functor in one direction is given by sending a $G$-representation $\sigma$ to the left $E_{\mf{B}}$-module $\Hom_G(\sigma, I_{\mf{B}})^\vee$, where $(-)^\vee$ denotes the Pontryagin dual. The rings $E_{\mf{B}}$ have been computed explicitly by Pa{\v s}k{\=u}nas \cite{paskunas-image}, using Colmez's Montr\'{e}al functor \cite{colmez-functor}. In particular, the center of such $E_{\mf{B}}$ is the universal pseudodeformation ring of $\rho_{\mf{B}}$, and in fact the whole $E_{\mf{B}}$ is often (but not always) isomorphic to the universal Cayley--Hamilton algebra of $\rho_{\mf{B}}$ (see \S \ref{subsec: def theory generalities} for the precise definition). 

Thus, by Morita theory, constructing a fully faithful functor
\[
F_{\mf{B}} : \Mod_{G,\zeta}^{lfin}(\oo)_{\mf{B}} \to \IndCoh(\mf{X}_{\mf{B}})
\]
essentially amounts to exhibiting an object $X_{\mf{B}} \in \D_{coh}^b(\mf{X}_{\mf{B}})$ satisfying 
\[
\RHom(X_{\mf{B}},X_{\mf{B}}) = E_{\mf{B}},
\]
i.e.\ that $\End(X_{\mf{B}})=E_{\mf{B}}$ and $\Ext^i(X_{\mf{B}},X_{\mf{B}})=0$ for $i\geq 1$. The functor is then (essentially) given as the derived tensor product
\begin{equation}\label{intro: formula for functor}
\sigma \mapsto X_{\mf{B}}^\ast \otimes_{E_{\mf{B}}}^L \Hom_G(\sigma, I_{\mf{B}})^\vee,
\end{equation}
where $X_{\mf{B}}^\ast$ denotes the coherent dual of $X_{\mf{B}}$. We note that for generic blocks, the target category $\IndCoh(\mf{X}_{\mf{B}})$ is equivalent to the quasicoherent derived category (Lemma \ref{QCoh = IndCoh}), but this is not the case for non-generic blocks. The source category for $F_{\mathfrak{B}}$ is compactly generated by its full subcategory of finite length objects, so $\IndCoh(\mf{X}_{\mf{B}})$, which is compactly generated by $\D_{coh}^b(\mf{X}_{\mf{B}})$, is the natural target category. The functor $F_{\mathfrak{B}}$ will preserve compact objects.

Finding the objects $X_{\mf{B}}$ and verifying that they satisfy $\RHom(X_{\mf{B}},X_{\mf{B}}) = E_{\mf{B}}$ takes up the bulk of the work in this paper. In particular, we rely on being able to compute the stacks $\mf{X}_{\mf{B}}$ explictly, using the machinery developed in \cite{wang-erickson-algebraic,wang-erickson-A-inf} (building on work of Bella\"{i}che and Chenevier \cite{BC2009,chenevier-determinant}), explicit descriptions of the quotients of $\Gamma_{\Qp}$ relevant to the non-generic cases developed by B\"ockle and Pa{\v s}k{\=u}nas \cite{bockle-demuskin,paskunas-image}, invariant theory, and modular representation theory.

Let us describe the shape of $X_{\mf{B}}$ for the different blocks. We remark that the properties we require of $X_{\mf{B}}$ do not uniquely determine it. Nevertheless, they seem to be natural and we expect that further work on categorical $p$-adic local Langlands will clarify the situation. 

For supersingular blocks, $\rho_{\mf{B}}$ is irreducible and $\mf{X}_{\mf{B}}$ is the stack quotient $[\Spec R /\mu_2]$, where $R$ is a deformation ring of $\rho_{\mf{B}}$. The sheaf $X_{\mf{B}}$ is then the twisted structure sheaf of $\mf{X}_{\mf{B}}$ (i.e.\ $R$, viewed as a $\Z/2$-graded $R$-module in degree $1$), and verifying that this has the correct properties is immediate from the results of \cite{paskunas-image}. 

For the generic principal series blocks and non-generic case I, Pa{\v s}k{\=u}nas has shown that $E_{\mf{B}}$ is the universal Cayley--Hamilton algebra (cf.~Definition \ref{def:univCH}) associated to the universal pseudodeformation of $\rho_{\mf{B}}$. In these cases, we let $X_{\mf{B}}$ be the vector bundle underlying the universal Galois representation on $\mf{X}_{\mf{B}}$. The general theory of the stacks $\mf{X}_{\mf{B}}$ gives a canonical ring homomorphism
\[
E_{\mf{B}} \to \End(X_{\mf{B}}).
\]
In the generic principal series case, it is relatively straightforward to show that this homomorphism is an isomorphism and that $\Ext^i(X_{\mf{B}},X_{\mf{B}})=0$ for $i\geq 1$; this essentially goes back to \cite{BC2009}. We prove this in the non-generic I case as well, but the proof (given in \S \ref{subsec: non-generic I}) is more involved, using tools from modular representation theory and invariant theory together with the explicit nature of $\mf{X}_{\mf{B}}$. This complication is caused by the fact that non-generic case I is the only case in which $\rho_{\mf{B}}$ is not multiplicity free, which means that $\mf{X}_{\mf{B}}$ cannot be written as the quotient of an affine scheme by a \emph{linearly} reductive group. 

The final type of block, non-generic case II, has the most complicated $X_{\mf{B}}$. We construct it as the direct sum of the universal vector bundle and an explicit maximal Cohen--Macaulay (but not locally free!) coherent sheaf, and verifying that $\RHom(X_{\mf{B}},X_{\mf{B}})=E_{\mf{B}}$ is computationally demanding (a short glance at \S \ref{subsec: non-gen II coh sheaves}, where this is done, should convince the reader of this). On the other hand, this gives an explicit `Galois-theoretic' description of $E_{\mf{B}}$ in this case, something which is not done in \cite{paskunas-image} (although a less explicit Galois-theoretic description can be obtained easily from the results of \cite{paskunas-tung}). The non-projective part of $X_{\mf{B}}$ has an endomorphism algebra which matches the (opposite) endomorphism algebra of the injective envelope of an irreducible one-dimensional representation of $G$. This keeps track of information which is lost by applying the Montr\'{e}al functor, whose kernel in $\mf{B}$ is generated by this one-dimensional representation of $G$.

In the supersingular and generic principal series cases, our functors can be directly constructed already at the level of abelian categories, but this is not true for the non-generic cases. In non-generic case I, we show a posteriori that the functor is $t$-exact\footnote{While this means that $H_0(F_{\mf{B}})$ gives a fully faithful embedding at the level of abelian categories, $F_{\mf{B}}$ is \emph{not} simply the derived functor of $H_0(F_{\mf{B}})$ in this case, though it is closely related to it. See Remark \ref{rmk: abelian vs derived for ng1} for more details.}, but in non-generic case II we show that $F_{\mf{B}}$ sends the trivial representation to a complex concentrated in (homological) degree $1$. More generally, we compute $F_{\mf{B}}(\pi)$ explicitly for all blocks and all irreducible representations $\pi$. In particular, we show that $F_{\mf{B}}(\pi)$ is concentrated in homological degree $0$ (resp.\ degree $1$) when $\pi$ is infinite dimensional (resp.\ finite dimensional).

\subsection{The assumption that $p \ge 5$} We have made the running assumption that $p \ge 5$ so that we can appeal to the results of \cite{paskunas-image}. The authors expect (but haven't checked) that the results would extend smoothly to generic blocks for $p = 2,3$, using the results of \cite{paskunas-2}. More recent work of Pa\v{s}k\={u}nas--Tung \cite{paskunas-tung} reproves many of the main results of Pa\v{s}k\={u}nas's earlier work in a way which handles all blocks for all primes. However, they do not compute the ring $E_{\mathfrak{B}}$ (see their \S1.2), which we need in order to explicitly compare with an endomorphism algebra on the Galois side.

\subsection{The Montr\'{e}al functor} Colmez's Montr\'{e}al functor plays an essential role in proving the results of \cite{paskunas-image}. Having used Pa{\v s}k{\=u}nas's results to construct the functor of Theorem \ref{intro main embedding}, a natural question (asked of us by Pa{\v s}k{\=u}nas) is whether we can recover the Montr\'{e}al functor from the embedding of categories? The answer is yes: we show in \S\ref{subsec:montreal} that we can recover the Montr\'{e}al functor from our embedding by tensoring with the universal Galois representation on $\frak{X}_{\zeta\epsilon}$ and taking global sections. This says that the Montr\'{e}al functor is the `Whittaker coefficient' for the universal Galois representation in the sense of the geometric Langlands program (cf.~\cite[\S1.2.3]{faergeman2022nonvanishing}, for example).

\subsection{Local-global compatibility}\label{subsec: intro l-g} As an application, we connect our functors $F_{\mf{B}}$ to the (co)homology of modular curves through a `local-global compatibility' result. For this, we need to enlarge the domain of $F_{\mf{B}}$. Let $\oo \lb G \rb$ be the ring defined by Kohlhaase \cite{kohlhaase} (over a field; see \cite{shotton} for a definition over $\oo$) and let $\oo \lb G \rb_\zeta$ be the largest quotient of $\oo \lb G \rb$ on which the center of $G$ acts as $\zeta$. We show that the defining formula (\ref{intro: formula for functor}) for $F_{\mf{B}}$ can be rewritten as
\[
\sigma \mapsto X_{\mf{B}}^\ast \otimes_{E_{\mf{B}}}^L I_{\mf{B}}^\vee \otimes_{\oo \lb G \rb_\zeta}^L \sigma = (X_{\mf{B}}^\ast \otimes_{E_{\mf{B}}} I_{\mf{B}}^\vee ) \otimes_{\oo \lb G \rb_\zeta}^L \sigma
\]
and use this formula to extend the domain of $F_{\mf{B}}$ to all left $\oo \lb G \rb_\zeta$-modules (here the Pontryagin dual $I_{\mf{B}}^\vee$ of $I_{\mf{B}}$ is flat over $E_{\mf{B}}$). We note that the extended functor is no longer fully faithful.

The setup for our local-global compatibility result is then as follows. For simplicity, we work with $\PGL_{2/\Q}$, and write $G^{ad} := \PGL_2(\Qp)$ (in particular, we look at the trivial central character). Let $\Gamma_\Q := \Gal(\ol{\Q}/\Q)$ with decomposition subgroups $\Gamma_{\Q_\ell}$ for primes $\ell$. Let $r : \Gamma_\Q \to \GL_2(\Fpbar)$ be a continuous representation. We assume that $\det(r) = \omega$ and that
\begin{enumerate}

\item $r |_{\Gamma_{\Qp}}$ is indecomposable, and not a twist of an extension of the form $0 \to \omega  \to r'_p \to \mbf{1} \to 0$;

\smallskip 

\item if $r |_{\Gamma_{\Q_\ell}}$ is ramified for some $\ell \neq p$, then $\ell$ is not a vexing prime in the sense of \cite{diamond-ext};

\smallskip

\item $r |_{\Gamma_{\Q(\zeta_p)}}$ has adequate image, in the sense of \cite[Defn.~2.3]{thorne-adequate}.
\end{enumerate}
We let $N$ be the Artin conductor of $r$, choose a sufficiently large coefficient field $L$, and we let $\mf{B}$ be the block such that $\rho_{\mf{B}}$ is isomorphic to the semisimplification of $r |_{\Gamma_{\Qp}}$. We consider the algebraized moduli stack $\mf{X}_r$ of continuous $\Gamma_{\Q}$-representations with determinant $\varepsilon$, with reduction $r$, and which are minimally ramified at primes $\ell \neq p$. A key role is played by the restriction map
\[
f : \mf{X}_{r} \to \mf{X}_{\mf{B}}.
\]
We set $R_{\Q,N}$ to be the global sections of the structure sheaf of $\mf{X}_{r}$; this is simply the universal deformation ring of $r$ (with conditions as above).

Instead of formulating and proving our results for homology of $\PGL_{2/\Q}$-modular curves, it turns out to be better, both from a conceptual and practical point of view, to work with (adelic) $p$-arithmetic homology, as in e.g.\ \cite{tarrach2022sarithmetic}. Thus, letting $Y_{\infty} = \PGL_2(\R)/\mathrm{PSO}_2(\R)$ and letting $Y_p$ be the Bruhat--Tits tree of $G^{ad}$, we look at the double coset space
\[
Y_N := \PGL_2(\Q) \backslash Y_{\infty} \times Y_p \times \PGL_2(\A^{\infty}) / K_1^p(N) G^{ad},
\]
where $K_1^p(N) \sub \PGL_2(\wh{\Z}^p)$ consists of matrices whose second row is congruent to $(0 \,\, 1)$ modulo $N$ (modulo scalars). Every (abstract) left $\oo[G^{ad}]$-module $\sigma$ (and hence every left $\oo \lb G^{ad} \rb$-module) gives rise to a local system on $Y_N$. If $\sigma$ is the compact induction $\sigma = \mathrm{ind}_{K_p}^{G^{ad}}\tau$ of some $\oo[K_p]$-module $\tau$ for $K_p \sub G^{ad}$ a compact open subgroup, then the homology $H_\ast(Y_N, \sigma)$ is canonically isomorphic to the homology of $\tau$, viewed as a local system on the $\PGL_2$-modular curve of level $K_1^p(N)K_p$. If $M$ is an $\oo \lb G^{ad} \rb$-module, then the Hecke action on the homology $H_\ast(Y_N,\sigma)$ gives it an $R_{\Q,N}$-module structure. Our local-global compatibility theorem is then the following:

\begin{theorem}\label{lg thm intro}
Let $\mc{V}$ be the vector bundle underlying the universal Galois representation on $\mf{X}_r$. Then, if $\sigma$ is a $\oo\lb G^{ad} \rb$-module, we have an isomorphism 
\[
H_\ast(Y_N, \sigma)_r \cong H_\ast(R\Gamma(\mf{X}_r, \mc{V} \otimes f^!(F_{\mf{B}}(\sigma))[-2])
\]
of $R_{\Q,N}$-modules which is functorial in $\sigma$.
\end{theorem}

The act of tensoring with $\mc{V}$ should be seen as `applying a Hecke operator' (on the spectral side) in the sense of \cite{fargues-scholze}. We further note that both sides may be given actions of $\Gamma_{\Q}$, and the isomorphism is equivariant with respect to these actions. For the proof, one reduces to the case $\sigma = \oo \lb G^{ad} \rb$, in which case we prove that $H_\ast(Y_N, \oo \lb G^{ad} \rb)_{r}$ is completed homology for $\PGL_{2/\Q}$ (with tame level $K_1^p(N)$, localized at $r$). The proof then amounts to computing the right hand side and comparing the result to the local-global compatibility results for completed homology from \cite{ceggps2,gee-newton}. A key step of this computation is to show that $f$ is relative complete intersection, which follows from the patching techniques of \emph{op.~cit.}. Along the way we also need to prove a big $R=\T$ theorem, which appears to be new when $r|_{G_{\Qp}}$ is a twist of an extension of $\omega$  by $\mbf{1}$.

Theorem \ref{lg thm intro} has many interesting special cases, concerning coefficient systems well known in the theory of modular forms. In particular, if $K_p \sub \PGL_2(\Zp)$ is a compact open subgroup, then setting $\sigma = \oo \lb G^{ad} \rb \otimes_{\oo \lb K_p \rb} (\Sym^{k-2}A^2)(\det)^{(2-k)/2}$ for $k\geq 2$ even recovers the usual (adelic) arithmetic homology of $\PGL_{2/\Q}$ at level $K^p_1(N)K_p$ with coefficients in $\Sym^{k-2}A^2 \otimes \det^{(2-k)/2}$ (where $A$ can be any $\oo$-algebra), and Poincar\'e duality relates this to cohomology. Another interesting case concerns the $\PGL_{2/\Q}$-eigencurves constructed in \cite{hansen-eigenvarieties,tarrach2022sarithmetic}; see Corollary \ref{local-global eigenvarieties}.

We note here that it should be possible to remove the restriction to trivial central character in our local--global statement, and indeed the restriction to fixed central character in Theorem \ref{intro main embedding}, by using the results of \cite[\S6]{ceggps2}.

\subsection{Motivation and relation to other work}\label{ssec:motives}
This project originated from an attempt to understand Ludwig's non-classical overconvergent eigenforms for $\SL_{2/\Q}$ \cite{ludwig}, the idea being that the structure of a hypothetical $p$-adic local Langlands correspondence in families for $\SL_2(\Qp)$ would explain the existence of such forms and their relation to non-automorphic members of $L$-packets\footnote{A more direct approach to the existence of non-classical overconvergent eigenforms and their relation to non-automorphic member of $L$-packets, still morally using ideas of geometrization of the $p$-adic local Langlands correspondence, was given in \cite{johansson-ludwig}.} (and could be used to show similar phenomena in the completed cohomology of $\SL_{2/\Q}$). However, direct attempts to formulate a $p$-adic local Langlands correspondence in families for $\SL_2(\Qp)$, in the spirit of \cite{kisin-def}, run into issues of dimensions of $\Ext$-groups not matching up. Instead, our calculations of the structure of supersingular blocks for $\SL_2(\Qp)$ (in the sense of \cite{paskunas-image}), together with the first version of \cite{hellmann-derived}, strongly suggested to us the formulation of $p$-adic local Langlands as an embedding of categories. Since our intended strategy for proving results about $\SL_2(\Qp)$ was to deduce them from the case of $\GL_2(\Qp)$, we decided to work those out first. The goal was to show that a categorical formulation of the $p$-adic local Langlands for $\GL_2(\Qp)$ was possible, and might point the way towards the long sought after generalization to other groups.

Since we started to develop these ideas, a lot has happened in the field. In particular, the notes \cite{egh} state a general $p$-adic local Langlands conjecture for $\GL_n$ over $p$-adic fields $F$, as (very roughly speaking) a categorical embedding
\[
\mf{A}=\mf{A}_{\GL_n(F)} : \D(\Mod^{sm}_{\GL_n(F)}(\oo)) \to \IndCoh(\mathrm{EG}_{n,F}),
\]
where $\mathrm{EG}_{n,F}$ denotes the Emerton--Gee stack of \'etale $(\varphi,\Gamma)$-modules of rank $n$ for $F$, with an announcement of a proof for $\GL_2(\Qp)$ \cite{dotto-emerton-gee}. See \cite[Conj. 6.1.14]{egh} for a more precise statement. Moreover, they also conjecture the existence of a similar functor $\mf{A}^{rig}=\mf{A}_{\GL_n(F)}^{rig}$ (perhaps not an embedding) linking the locally analytic representation theory of $\GL_n(F)$ to quasicoherent sheaves on moduli stacks of (not necessarily \'etale) $(\varphi,\Gamma)$-modules of rank $n$ over the Robba ring \cite[Conj. 6.2.4]{egh}. The extended version of the functor $F_\mf{B}$ that we discussed in \S \ref{subsec: intro l-g} should be related to both of these functors. In particular, we expect\footnote{Some evidence for this is given in the proof of \cite[Thm.~7.3.5]{egh}.} that $F_{\mf{B}}$, restricted to $\D(\Mod^{sm}_{\GL_2(\Qp),\zeta}(\oo))$, is equal (or at least very closely related) to the composition of the functor of \cite{dotto-emerton-gee} with pullback along 
\[
\mf{X}_{\mf{B}}^{\wedge} \to \mathrm{EG}_{2,\Qp}
\]
where $\mf{X}_{\mf{B}}^{\wedge}$ is the completion of $\mf{X}_{\mf{B}}$ along the maximal ideal of the universal pseudodeformation ring of $\rho_{\mf{B}}$\footnote{Note that a priori our $F_{\mf{B}}$ gives sheaves on $\mf{X}_{\mf{B}}$, not $\mf{X}_{\mf{B}}^\wedge$. However, pullback along the natural map $\mf{X}_{\mf{B}}^\wedge \to \mf{X}_{\mf{B}}$ induces an equivalence for all $\mf{B}$ except non-generic I, by \cite[Thm.~1.6]{alper-hall-rydh-etale-local}. For a non-generic I block, we expect that the methods of \S \ref{subsec: non-generic I} imply that pullback along $\mf{X}_{\mf{B}}^\wedge \to \mf{X}_{\mf{B}}$ is fully faithful on the essential image of $F_{\mf{B}}$.}. We also expect that $F_{\mf{B}}$, when applied to left modules for the distribution algebra $\ms{D}(\GL_2(\Qp))$, is closely related to the conjectural functor $\mf{A}_{\GL_2(\Qp)}^{rig}$ (or rather its version with a fixed determinant). Indeed, $\mf{A}_{\GL_n(F)}$ and $\mf{A}_{\GL_n(F)}^{rig}$ are expected to be related in general (see \cite[\S 6.2.11]{egh}), and $F_{\mf{B}}$ is related to the coherent sheaf on the $\PGL_2(\Qp)$-eigencurve via Theorem \ref{lg thm intro} in the same way that $\mf{A}^{rig}$ is expected to be \cite[Conj. 9.6.31]{egh}.

The main impact of \cite{egh,dotto-emerton-gee} on this paper is the focus on the functors $F_{\mf{B}}$, as opposed to their dual versions. We originally discovered the dual functors, which arise more naturally in our framework, but shifted our focus after discussions with Toby Gee on the image of irreducible representations under the functors of \cite{dotto-emerton-gee}. Moreover, we refer the reader to \cite{egh} for an excellent and thorough introduction to the $p$-adic Langlands program with a view towards categorification.  

We expect that Theorem \ref{intro main embedding} should have an extension to all $\GL_n(F)$. Part of this expectation is based on the observation that the relation between the Emerton--Gee stack and the moduli stack of Galois representations resembles the relation between the stacks of local systems and their versions with restricted variation in the geometric Langlands program \cite[\S 1]{agkrrv}. Moreover, the geometric Langlands correspondence has a version with restricted variation, which is very closely related to the ``standard'' version \cite[\S 21]{agkrrv}. Nevertheless, we refrain from attempting to formulate a precise conjecture generalizing Theorem \ref{intro main embedding}. The most subtle part appears to be to figure out the source category. Since the Galois stacks decompose according to residual pseudocharacters in full generality, one might expect the source category to have a corresponding block decomposition. A na\"{i}ve guess for such a category is the (ind-completion of the derived category of) smooth representations that are locally both of finite length and finitely presented. However, in general this category seems unlikely to contain irreducible supersingular representations (which are not of finite presentation \cite{Schraen-presentation, Wu-presentation}).

\subsection{Outline of the paper}
Let us briefly outline the contents of this paper. Section \ref{chap: stacks of Galois reps} recalls generalities of deformation and moduli theory of representations of profinite groups, mainly from \cite{BC2009,chenevier-determinant,wang-erickson-algebraic}, and gives our conventions on quasicoherent sheaves on stacks. In Section \ref{chap: stacks for GL2Qp}, we compute explicit presentations of the stacks $\mf{X}_{\mf{B}}$, construct all the $X_{\mf{B}}$, and prove all their relevant properties. Section \ref{chap: GL2preliminaries} then recalls the (absolutely irreducible) blocks for $\GL_2(\Qp)$ and sets up a category-theoretic framework for Theorem \ref{intro main embedding}. Section \ref{chap: geometric interpretation} proves our main results, by comparing our results from \S \ref{chap: stacks for GL2Qp} with those of \cite{paskunas-image}. Finally, Section \ref{chap: local-global} extends the domain of $F_{\mf{B}}$ to all left $\oo \lb G \rb_\zeta$-modules, discusses $p$-arithmetic homology, and proves Theorem \ref{lg thm intro}.

\subsection*{Notation and conventions}\label{subsec: notation} We collect some notation that will we used throughout this paper. We let $p$ be a prime number and assume $p\geq 5$ throughout the paper. If $K$ is a field, then $\Gamma_K$ will denote the absolute Galois group of $K$. Let $\varep$ denote the $p$-adic cyclotomic character of $\Gamma_K$ and let $\omega$ denote its reduction modulo $p$. We normalize Local Class Field Theory so that uniformizers correspond to geometric Frobenii; this is the same convention as in \cite{paskunas-image}. Moreover, for $K/\Qp$ a finite extension and $A$ a profinite ring, we will view any continuous character $\chi : \Gamma_K \to A^\times$ as a continuous character $\chi : K^\times \to A^\times$, by Local Class Field Theory, without changing the notation (and vice versa).

Many of our objects will be defined over $\OO$, the ring of integers in a finite extension $L/\Qp$, with a uniformizer $\varpi$. Its residue field will be denoted by $\F$.

If $A$ is a (not necessarily commutative) ring, then $\LMod(A)$ and $\RMod(A)$ denotes the abelian categories of left and right $A$-modules, respectively. If $A$ is commutative, we simply write $\Mod(A)$. If $A$ is a topological ring, then $\LMod_{disc}(A)$ and $\LMod_{cpt}(A)$ denotes the abelian categories of left discrete and compact $A$-modules, respectively, and we use $\RMod$ with similar decorations for right modules.

Since we will predominantly deal with left exact functors and homology of topological spaces, our conventions in homological algebra will be \emph{homological} (as opposed to cohomological). In particular, our complexes will mostly be \emph{chain} complexes, with $-_\bu$ denoting the index in a chain complex. Our shift convention is that if $C_\bu$ is a chain complex, then $C_\bu[d]$ is the chain complex satisfying $C_\bu[d]_n = C_{n+d}$. In particular, if $C_\bu$ is concentrated in degree $0$, then $C_\bu[d]$ is concentrated in (homological) degree $-d$. We will use the notation
\[
H_\ast(C_\bu)
\]
to denote the homology of $C_\bu$, where we regard $-_\ast$ as a generic index. Alternatively, the reader may interpret $H_\ast(C_\bu)$ as the total homology of $C_\bu$, viewed as a graded abelian group, and morphisms $H_\ast(C_\bu) \to H_\ast(D_\bu)$ as graded morphisms; either interpretation is fine.

Our conventions and notation for derived categories and their $\infty$-categorical enhancements are given mainly in \S \ref{sec: coh sheaves on stacks}, with some additions in \S \ref{sec: categorical constructions}. We do note that, despite using chain complexes throughout, our conventions for bounded below and bounded above follows that used for cochain complexes. Thus, for us $C_\bu$ is bounded above (resp. below) if $C_n=0$ for $n \ll 0$ (resp. $n \gg 0$) and the notation $-^-$ (resp. $-^+$) will be applied to categories of bounded above (resp. below) chain complexes, though we hasten to say that we will mainly work with categories of bounded or unbounded chain complexes.

Throughout the paper, we will write $-^\ast$ for linear duals, and $-^\vee$ for Pontryagin duals. The internal Hom in a monoidal category (if it exists) with be denoted by $\Homi$, and its (total and individual) derived functors will be denoted by $\RHomi$ and $\Exti^i$.

We will need to do many calculations with graded modules; these will either be $\Z$- or $\Z/2$-graded. If $M$ is a graded module then $M_k$ denotes its degree $k$ part. Moreover, $M(n)$ denotes the graded module defined by $M(n)_k = M_{n+k}$. If $R$ is a graded ring then the category of graded $R$-modules is symmetric monoidal under the tensor product (over $R$), and has an internal Hom. If $M$ is finitely generated as an $R$-module and $N$ is arbitrary, the internal Hom is given by $\Homi(M,N) = \Hom_R(M,N)$, with grading $\Homi(M,N)_k = \Hom(M,N(k))$.

\subsection*{Acknowledgments} C.J.\ wishes to thank David Hansen, Arthur-C\'{e}sar Le Bras and Judith Ludwig for conversations on the broader context of this paper. We also thank Toby Gee for discussions in relation to \cite{dotto-emerton-gee}. C.J. wishes to thank the Mathematical Institute of the University of Oxford and Merton College for their hospitality during a visit in April 2022. J.N.\ wishes to thank the Department of Mathematical Sciences at Chalmers University of Technology and the University of Gothenburg for its hospitality during a visit in February 2019. J.N.\ also wishes to thank the Hausdorff Research Institute for Mathematics for its hospitality during the Trimester Program ``The Arithmetic of the Langlands Program'', funded by the Deutsche Forschungsgemeinschaft (DFG, German Research Foundation) under Germany's Excellence Strategy--EXC-2047/1--390685813. The authors thank Toby Gee and Vytas Pa{\v s}k{\=u}nas for helpful discussions about an earlier version of this article. Finally, thanks to the anonymous referee for their careful reading of the paper and helpful comments.

C.J.\ was supported by Vetenskapsr\r{a}det Grant 2020-05016, \textit{Geometric structures in the $p$-adic Langlands program} during part of this project. C.W.E.\ was supported by the Simons Foundation through award TSM-846912 and by the National Science Foundation through award DMS-2401384. J.N.\ was supported by a UKRI Future Leaders Fellowship, grant MR/V021931/1. For the purpose of Open Access, the authors have applied a CC BY public copyright licence to any Author Accepted Manuscript version arising from this submission.

\section{Stacks of representations and coherent sheaves}
\label{chap: stacks of Galois reps}

The goal of this section is to recall generalities on the moduli theory and deformation theory of profinite groups, along with algebraizations of their moduli. We also include discussions of derived categories of coherent sheaves on algebraic stacks. 

\subsection{Deformation theory generalities}\label{subsec: def theory generalities}

Let $\Gamma$ be a profinite group satisfying the $\Phi_p$-finiteness condition of Mazur \cite[\S1.1]{mazur1989}. We recall fundamental facts about $\Spf \Z_p$-formal schemes and stacks of $2$-dimensional representations, following \cite{wang-erickson-algebraic} in part. The reader is presumed to be familiar with the theory of pseudorepresentations and their deformation theory, which is developed in \cite{chenevier-determinant}. Sometimes we take the liberty of discussing a pseudorepresentation as a ``trace function,'' using the theory of pseudocharacters, but these amount to the same thing by \cite[Prop.\ 1.29]{chenevier-determinant}. 

\begin{definition}
\label{defn: moduli}
Let $B$ denote a topologically finitely generated $\Z_p$-algebra. We establish the following moduli functors and stacks in groupoids, over topologically finite type $\Spf \Z_p$-formal schemes with the fppf topology, in terms of their value on $B$. 	
\begin{itemize}
\item Let $\widehat{\Rep}^{\square, \tilde \psi}$ denote the moduli functor of homomorphisms $\Gamma \to \GL_2(B)$. 
\item Let $\widehat{\Rep}^{\tilde \psi}$ denote the moduli groupoid of rank 2 projective $B$-modules $V$ equipped with a homomorphism $\Gamma \to \Aut_B(V)$ and a trivialization of the determinant of $V$, $\wedge^2 V \isoto B$. 
\item Let $\PsR^{\tilde \psi}$ denote the moduli functor of 2-dimensional pseudorepresentations $D : \Gamma \to B$. 
\end{itemize}
These moduli spaces admit natural morphisms $\widehat{\Rep}^{\square,\tilde\psi} \to \widehat{\Rep}^{\tilde\psi} \to \PsR^{\tilde\psi}$, where the first arrow is compatible with a presentation of the stack $\widehat{\Rep}^{\tilde\psi}$ as $[\widehat{\Rep}^{\square,\tilde\psi}/\SL_2]$. Here the action of $\SL_2$ arises from its adjoint action on $\GL_2$. The second arrow arises from associating a pseudorepresentation $D(\rho)$ to the action of $\Gamma$ on the $B$-module $V$ by $\rho$, using the characteristic polynomial coefficients of this action. 
\end{definition}

\begin{remark}
We are adopting the somewhat awkward notation with superscripts $(-)^{\tilde\psi}$ since we will reserve the unadorned notation for those with a fixed determinant $\psi: \Gamma \to \OO^\times$. Thus we are thinking of ``$\tilde\psi$'' as standing implicitly for the universal $p$-adic character of $\Gamma$.
\end{remark}

A 2-dimensional pseudorepresentation $D : \Gamma \to B$ is called \emph{reducible} when it has the form $D(\rho)$ for some $\rho$ of the form $\rho \simeq \nu_1 \oplus \nu_2$ for characters $\nu_i : \Gamma \to B^\times$. Reducibility is a Zariski closed condition on each of these moduli spaces. From now on, we drop ``2-dimensional'' from our terminology for pseudorepresentations.

The moduli functor $\PsR^{\tilde \psi}$ is known to be the disjoint union of formal spectra representing deformation functors of finite field-valued pseudorepresentations $D : \Gamma \to \F$ over their minimal field of definition $\F$, a finite extension of $\F_p$ \cite[Thm.\ F]{chenevier-determinant}. That is, if we write $\Def_{D}^{\tilde \psi} = \Spf R_{D}^{\tilde \psi}$ as the formal spectrum of the complete Noetherian local ring $R_{D}^{\tilde \psi}$ representing the deformation functor for $D$, the decomposition is expressable as 
\[
\PsR^{\tilde \psi} \cong \coprod_{D} \Def_{D}^{\tilde \psi}. 
\]
We write $\widehat{\Rep}^{\square,\tilde \psi}_{D}$ and $\widehat{\Rep}^{\tilde \psi}_{D}$ for the substack/subspace of $\widehat{\Rep}^{\square,\tilde \psi}$ and $\widehat{\Rep}^{\tilde \psi}$ over $\Def_{D}^{\tilde \psi}$. 

Any residual pseudorepresentation is induced by a unique (up to isomorphism) semisimple representation $\rho_D : \Gamma \to \GL_2(\F)$ over the same field of definition $\F$ as $D$. After a possible at most quadratic extension, we may assume that the irreducible summands of this semisimple representation are absolutely irreducible. In what follows, we replace $\F$ with a minimal such an extension. 

A residual pseudorepresentation $D : \Gamma \to \F$ is called \emph{multiplicity free} when the irreducible summands of $\rho_D$ are pairwise distinct. This includes the case that $\rho_D$ is irreducible, in which case we also call $D$ irreducible. 

Next we introduce \emph{Cayley--Hamilton algebras}; see \cite[\S1]{chenevier-determinant} for a reference. We refer to a Cayley--Hamilton algebra over $A$ (or with scalar ring $A$) as an $A$-algebra $E$ equipped with a pseudorepresentation $D_E : E \to A$ satisfying the Cayley--Hamilton property; concisely, this property means that every element of $E$ satisfies the characteristic polynomial determined by $D_E$. 

\begin{definition}
\label{def:univCH}
Let $E_D^{\tilde \psi}$ denote the \emph{universal Cayley--Hamilton algebra over $D$}, which is given by 
\[
E_D^{\tilde \psi} := \frac{R_D^{\tilde \psi}\lb \Gamma\rb}{\mathrm{CH}(D^{u,\tilde \psi})}
\]
where $\mathrm{CH}(D^{u,\tilde \psi})$ denotes the minimal two-sided ideal that factors the universal deformation $D^{u,\tilde \psi} : R_D^{\tilde \psi}\lb\Gamma\rb \to R_D^{\tilde \psi}$ and makes it satisfy the Cayley--Hamilton property. We also write $D_{E_D^{\tilde \psi}}: E_D^{\tilde\psi} \to R_D^{\tilde \psi}$ for the pseudorepresentation that $E_D^{\tilde\psi}$ is equipped with. The Cayley--Hamilton representation
\[
\rho^{u, \tilde\psi} : \Gamma \to E_D^{\tilde\psi}
\]
is universal in the sense that for any Cayley--Hamilton representation $\rho : \Gamma \to E$ with scalar ring $A$, if the induced pseudorepresentation $D_E \circ \rho: \Gamma \to A$ has constant residual pseudorepresentation $D$, then there exists a morphism of Cayley--Hamilton algebras $(f : E_D^{\tilde \psi} \to E$, $R_D^{\tilde \psi} \to A)$ such that $\rho = f \circ \rho^{u,\tilde \psi}$ and the map $R_D^{\tilde \psi} \to A$ equals the map coming from the moduli interpretation of $R_D^{\tilde \psi}$ applied to $D_E \circ \rho$. 
\end{definition}

In this paper, we will almost always want to restrict the determinant of representations and pseudorepresentations. Writing 
\begin{itemize}
\item $\psi : \Gamma \to \OO^\times$ for a character deforming $\det D : \Gamma \to \F^\times$, 
\item $\widehat{\Rep}^\square_D, \widehat{\Rep}_D, \PsR_D$ for moduli functors,
\end{itemize}
along with the following objects representing moduli problems with fixed determinant $\psi$:
\begin{itemize}
\item $E_D$ for the universal Cayley--Hamilton algebra 
\item with scalar ring $R_D$
\item and universal representation $\rho^u : \Gamma \to E_D^\times$ over $D$. 
\end{itemize}
In what follows, we continue with this convention as we introduce new moduli functors and rings.

\subsection{Algebraization of moduli functors and groupoids}

A main result of \cite[\S3]{wang-erickson-algebraic} is that all of the formal moduli spaces or groupoids of representations of $\Gamma$ with residual pseudorepresentation $D$ have a natural $R_D$-algebraic model of finite type. The source of this algebraization is the following finiteness result:

\begin{theorem}[{\cite[Prop.\ 3.6]{wang-erickson-algebraic}}]
$E_D$ is finitely generated as a $R_D$-module. 
\end{theorem} 

Using the universality of $E_D$, one can use the moduli $\Rep^\square(E_D), \Rep(E_D)$ of (non-topological) compatible representations of $E_D$ as a $R_D$-algebraic model for $\widehat{\Rep}_D^\square,\widehat{\Rep}_D$. That is, $\widehat{\Rep^\square(E_D)} \cong \widehat{\Rep}_D^\square$ and $\widehat{\Rep(E_D)} \cong \widehat{\Rep}_D$, completing with respect to the maximal ideal of $R_D$. 

\begin{definition}
Let $(E,B)$ be a Cayley--Hamilton algebra with scalar ring $B$ and pseudorepresentation $D_E : E \to B$. Let $C$ be a commutative $B$-algebra. A $C$-valued \emph{compatible representation} of $E$ is a homomorphism of $B$-algebras $E \to M_2(C)$ such that the following diagram commutes,
\[
\xymatrix{
E \ar[r] \ar[d]^{D_E} & M_2(C) \ar[d]^{\det} \\
B \ar[r] & C
}
\]
\begin{itemize}
\item Let $\Rep^\square(E)$ be the $\Spec B$-functor of compatible representations of $E$.
\item Let $\Rep(E)$ be the $\Spec B$-groupoid which associates to a $B$-algebra $C$ a projective rank 2 $C$-module $V$, an isomorphism $\wedge^2 V \isoto C$, and a compatible representation of $E$ on $V$, that is,
\[
\xymatrix{
E \ar[r] \ar[d]^{D_E} & \End_C(V) \ar[d]^{\det} \\
B \ar[r] & C
}
\]
\end{itemize}
As in Definition \ref{defn: moduli}, $\Rep(E) \cong [\Rep^\square(E)/\SL_2]$ under the adjoint action of $\SL_2$. 
\end{definition}

\begin{proposition}
Assume that $E$ is finitely generated as a $B$-algebra. $\Rep^\square(E)$ is an affine $B$-scheme of finite type. $\Rep(E)$ is a $\Spec B$-algebraic stack of finite type.
\end{proposition}
\begin{proof}
A standard ``generic matrices'' argument shows that $\Rep^\square(E)$ is of finite type over $\Spec B$. See e.g.\ \cite[\S3.1]{BIP2023}. 
\end{proof}

We also record the self-duality of the universal vector bundle on $\Rep(E)$. 

\begin{proposition}
\label{selfduality of universal vb}
Let $\mc{V}$ be the vector bundle underlying the universal representation of $E$. There is a canonical isomorphism $\mc{V} \cong \mc{V}^\ast$. 
\end{proposition}

\begin{proof}
Since we have a trivialization of $\wedge^2\mc{V}$ over $\Rep(E)$, the proposition follows from the standard fact that any rank $2$ vector bundle $\mc{F}$ on any algebraic stack admits a canonical isomorphism $\mc{F}^\vee \otimes \wedge^2 \mc{F} \cong \mc{F}$ (this follows from the fact that the wedge product is a perfect pairing $\mc{F} \times \mc{F} \to \wedge^2\mc{F}$).
\end{proof}

\subsection{GMAs and adapted representations in the multiplicity free reducible case}
\label{subsec: GMAs} 

When $D$ is multiplicity free and reducible, arising from the representation $\chi_1 \oplus \chi_2 : \Gamma \to \GL_2(\F)$, any lift of the two canonical orthogonal ordered idempotents of $\F \times \F$ over $E_D \to \F \times \F$ amounts to a 2-by-2 generalized matrix $R_D$-algebra ($R_D$-GMA) structure on $E_{D}$ \cite[Thm.\ 2.22]{chenevier-determinant}. We will simply use the term ``GMA'' to refer to a 2-by-2 GMA. 

See \cite[\S1.3]{BC2009} for generalities on GMAs. In particular, using coordinates coming from these ordered idempotents, we get an isomorphism
\begin{equation}
\label{eq: E coordinates}
E_D = \begin{pmatrix}
(E_D)_{1,1} & (E_D)_{1,2} \\
(E_D)_{2,1} & (E_D)_{2,2}
\end{pmatrix}
\cong 
\begin{pmatrix}
R_D & B_D \\ 
C_D & R_D
\end{pmatrix},
\end{equation}
where there is an implicit $R_D$-bilinear cross-diagonal multiplication with corresponding map \[B_D \otimes_{R_D} C_D \to R_D\] giving rise to a $R_D$-algebra structure on $E_D$. The pseudorepresentation $E_D \to R_D$ naturally arising from the GMA structure is equal to $D_{E_D}: E_D \to R_D$ \cite[Prop.\ 2.23]{wang-erickson-algebraic} and is Cayley--Hamilton, making any generalized matrix algebra a Cayley--Hamilton algebra. And the \emph{reducibility ideal} of $R_D$, which cuts out the locus of reducible pseudorepresentations in $\Spec R_D$, equals the image of the cross-diagonal multiplication map. 

We recall the following general notions from \cite[\S1.3]{BC2009}, where $B$ is a commutative ring and $E$ is a $B$-GMA. 

\begin{definition}
Let $E$ be a $B$-GMA. An \emph{adapted representation} of $E$ valued in a $B$-algebra $C$ is an $B$-algebra homomorphism $E \to M_2(C)$ that preserves the GMA structure (that is, maps the idempotents defining the GMA structure on $E$ to the standard idempotents in $M_2(C)$). 
\begin{itemize}
\item Let $\Rep^{\square,\Ad}(E)$ denote the $\Spec B$-functor of adapted representations of $E$. 
\item Let $\Rep^\Ad(E)$ denote the $\Spec B$-groupoid whose value on $C$ consists of an ordered pair of rank 1 projective $C$-modules $(V_1, V_2)$ equipped with an isomorphism $V_1 \otimes V_2 \isoto C$ and a homomorphism of $B$-GMAs (so, in particular, they preserve the ordered idempotents) $E \to \End_C(V_1 \oplus V_2)$. 
\end{itemize}
\end{definition}

One may check that $\Rep^\Ad(E) \cong [\Rep^{\square,\Ad}(E)/T]$, where $T$ is the standard diagonal torus in $\SL_2$, acting via the adjoint representation on $M_2$. We fix the isomorphism $T \cong \bG_m$ that makes $\bG_m$ act on the B (upper right) coordinate by $2 \in X^*(\bG_m) \cong \Z$ and the C coordinate by $-2$. 

In the following theorem, we let $A$ denote the $B$-algebra representing $\Rep^\square(E)$ and likewise let $S$ represent $\Rep^{\square,\Ad}(E)$. 

\begin{theorem}[{\cite{wang-erickson-algebraic, BC2009}}]
\label{presentation of moduli stack}
Let $E$ be a $B$-GMA, hence also a Cayley--Hamilton algebra over $B$. 
\begin{enumerate}
\item Adapted representations of $E$ are compatible representations. 
\item The resulting map $\Rep^{\square,\Ad}(E) \hookrightarrow \Rep^\square(E)$ is a closed immersion of affine $B$-schemes; we have the corresponding surjection $A \twoheadrightarrow S$. 
\item The morphism of (2) descends to an isomorphism of algebraic stacks
\[
\Rep^\Ad(E) = [\Rep^{\square, \Ad}(E)/T] \cong [\Rep^\square(E)/\SL_2] = \Rep(E).
\]
\item The GIT quotient scheme $\Rep^{\square,\Ad}(E) /\!/ \bG_m$ is naturally isomorphic to $\Spec B$. Equivalently, in ring-theoretic terms, there are natural isomorphisms
\[
B \cong  A^{\SL_2} \cong S^{\bG_m}. 
\]
\item Moreover, the natural action of $\SL_2$ on $M_2(A)$ (resp.\ $\bG_m$ on $M_2(S)$) and the natural maps $E \to M_2(A) \twoheadrightarrow M_2(S)$ produce isomorphisms
\[
E \cong M_2(A)^{\SL_2} \cong M_2(S)^{\bG_m}. 
\]
\item If $E$ is finitely generated as a $B$-algebra, then all of these schemes and stacks are of finite type over $\Spec B$. 
\end{enumerate}
\end{theorem}

\begin{proof}
See \cite[Prop.\ 2.23]{wang-erickson-algebraic} and the comments after its proof for the proof of part (1). Part (2) is easily checkable. Part (3) is \cite[Prop.\ 2.24]{wang-erickson-algebraic}, but $\SL_2$ replaces $\GL_d$ and $T$ replaces the diagonal torus in $\GL_2$. Since the invariant theory is reduced to a linearly reductive case, and $\GL_2$ and $\SL_2$ are each surjective onto $\PGL_2$ via the adjoint action, the result remains in this case. Part (4) is \cite[Cor.\ 2.25]{wang-erickson-algebraic}, and Part (5) follows quickly from Proposition \ref{prop: GMA description}. (The results above also closely follow after \cite[\S1.3]{BC2009}.) Part (6) follows from the standard construction using generic matrices. 
\end{proof}

We will also use this result of Bella\"iche--Chenevier. 

\begin{proposition}[{\cite[Prop.\ 1.3.13, Rem.\ 1.3.15]{BC2009}}]
\label{prop: GMA description}
Writing $E_{i,j}$ for the $R$-GMA coordinates of $E$ as in \eqref{eq: E coordinates}, $1 \leq i,j \leq 2$, there are canonical isomorphisms of graded $R$-modules
\[
E_{i,j} \simeq S_{2(j-i)}, 
\]
such that the coordinate-wise multiplication maps $E_{i,j} \otimes_R E_{j,k} \to E_{i,k}$ are compatible with the multiplication law of $S$. In particular, $S_0 = R = E_{1,1} = E_{2,2}$, and $S$ is generated as an $R$-algebra by $E_{1,2}$ and $E_{2,1}$. 
\end{proposition}

Next we apply these equivalences to the case of the universal Cayley--Hamilton algebra $E_D$ with scalar ring being the 	universal pseudorepresentation ring $R_D$. First, we set up notation. 

\begin{notation}
\label{notn: stacks}
Let $A_D$ denote the finitely generated $R_D$-algebra representing $\Rep^\square(E_D)$, with $\m_D$-adic completion $\hat A_D$. When $D$ is reducible and multiplicity free, let us write $S_D$ as the ring representing the $R_D$-algebraic moduli functor $\Rep^{\square,\Ad}(E_D)$, and let $\hat S_D$ denote its $\m_D$-adic completion. (We remark that $\hat A_D$ and $\hat S_D$ are not local, in general.) We have the following diagram of rings and moduli functors (and the top row of vertical arrows in the left diagram are pseudorepresentations). 
\[
\xymatrix{
& E_D \ar[r] \ar[d]^{D^u} & M_2(A_D) \ar[d]^{\det} \ar@{->>}[r] & M_2(S_D) \ar[d]^{\det} & & \Rep(E_D) \ar[dl] & \Rep^\Ad(E_D) \ar[l]_\sim \\
\OO \ar[r] & R_D \ar[r] \ar[d]^\wr & A_D \ar[d]^{(-)^\wedge_{\m_D}} \ar@{->>}[r] & S_D \ar[d]^{(-)^\wedge_{\m_D}}
 & \Spec R_D & \Rep^\square(E_D) \ar[l] \ar[u]_{\text{stack }/\SL_2} & \Rep^{\square,\Ad}(E_D) \ar[l] \ar[u]_{\text{stack }/T}\\
& R_D \ar[r] & \hat A_D \ar@{->>}[r] & \hat S_D & \Spf R_D \ar[u]_{(-)^\wedge_{\m_D}} & \widehat{\Rep}^\square_D \ar[l] \ar[u]_{(-)^\wedge_{\m_D}} & \widehat{\Rep}^{\square,\Ad}_D \ar[l] \ar[u]_{(-)^\wedge_{\m_D}}
}
\]

We will use the $T \cong \bG_m$-action on $\Spec S_D$ to consider $S_D$ to be a $\Z$-graded algebra $S_D = \bigoplus_{i \in \Z} S_{D,i}$. In fact, it is a $\Z$-graded $R_D$-algebra because characteristic polynomial functions are adjoint invariants, moreover $R_D \cong S_{D,0}$, by Theorem \ref{presentation of moduli stack}. 
\end{notation}

From now on, our notational convention is that the subscript ``$D$'' is dropped. 

\subsubsection{Irreducible case} 
\label{sssec: irred case}
When $D$ and $\rho = \rho_D$ are irreducible, it is well known that deformation theory of $\rho$ is identical to that of $D$ (see e.g.\ \cite[Thm.\ 2.22]{chenevier-determinant}). That is, the natural homomorphism $R \to R_\rho$, where $R_\rho$ is the universal deformation ring of $R_\rho$, is an isomorphism; also, $E \simeq M_2(R)$. 

In this case, the stacks above can be expressed in terms of the universal deformation ring $R_\rho$ and the universal lifting ring $R_\rho^\square$ as follows. 
\[
\widehat{\Rep} \cong  [\Spf R_\rho/\mu_2], \quad \widehat{\Rep}^\square \cong \Spf R_\rho^\square, \quad 
\Rep(E) \cong [\Spec R_\rho/\mu_2], 
\]
where the implicit adjoint action of $\mu_2 \subset \bG_m \cong T \subset \SL_2$ is trivial. In general, $R^\square_\rho$ is a further completion of $\hat A$ at the maximal ideal of $\hat A$ associated to $\rho$; in the irreducible case, $R^\square_\rho$ and $\hat A$ are isomorphic. 

\begin{remark}
The remaining trivial action of $\mu_2$ reflects the kernel of the adjoint action of $\SL_2$, making $\Rep(E)$ a $\mu_2$-gerbe. It reflects that the collection of rank 2 vector bundles $V$ (over scalar ring $C$) with a fixed isomorphism $\wedge^2 V \isoto C$ such that $\End_C(V) \simeq M_2(C)$, admits a twisting action by the group of isomorphism classes of line bundles whose squares are trivial. 
\end{remark}

\subsection{Coherent sheaves on stacks, and duality}\label{sec: coh sheaves on stacks}

In \S \ref{chap: stacks for GL2Qp} we will describe the stacks $\Rep(E)$ for the four different types of finite field-valued pseudorepresentations of $\Gamma_{\Qp}$ that are of interest to us, as well as the coherent sheaves on $\Rep(E)$ that we will use. For this reason, we record a few general recollections about coherent sheaves on stacks, specialized to the situation we encounter. 

\subsubsection*{Setup for stacks} 
Our basic setup is the following: let $G$ be a reductive group scheme over $\oo$ and let $A$ be a commutative Noetherian $\oo$-algebra with an action of $G$. Set $X^\square = \Spec(A)$ and let $X$ be the quotient stack $[X^\square/G]$. We let $\QCoh(X)$ and $\Coh(X)$ be the abelian categories of quasicoherent and coherent sheaves, respectively. These may be defined in different ways (for example, using the lisse-\'etale site on $X$), but they all coincide with the categories of $G$-equivariant $A$-modules and $G$-equivariant finitely generated $A$-modules, respectively; we will take this as our definition (see \cite[Ex. 2.3]{arinkin-bezrukavnikov}, for example). As a special case, if $G = \G_m$ (or $G=\mu_2$), then the $G$-action on $A$ is (equivalent to) a $\Z$-grading of $A$ as an $\oo$-algebra (or a $\Z/2$-grading), and $G$-equivariant $A$-modules are the same as $\Z$-graded $A$-modules (or $\Z/2$-graded modules). We will use this without further comment.

\subsubsection*{Conventions for derived categories and stable $\infty$-categories} 
We will also need to consider various derived categories, including their $\infty$-categorical enhancements. The stable $\infty$-categories that we will consider can be constructed from differential graded (dg) categories by means of the differential graded nerve construction, cf.\ \cite[\S 1.3.1]{lurie-ha}. All `usual' derived categories will be denoted by the letter $D$, and their $\infty$-categorical enhancements will be denoted by the letter $\D$. 
 
 For an abelian category $\mc{A}$ with enough injectives, the bounded below derived $\infty$-category $\D^+(\mc{A})$ is constructed in \cite[Variant 1.3.2.8]{lurie-ha} as the dg nerve of the dg category $\Ch^+(\mc{A}_{inj})$ of bounded below complexes of injectives in $\mc{A}$. It is also shown in (the dual version of) \cite[Prop.~1.3.4.6]{lurie-ha} that $\D^+(\mc{A})$ may be obtained by taking the dg nerve of the dg category $\Ch^+(\mc{A})$ of all bounded below complexes in $\mc{A}$, and then inverting quasi-isomorphisms. Dually, if $\mc{A}$ has enough projectives, $\D^-(\mc{A})$ can be constructed as the dg nerve of the dg category $\Ch^-(\mc{A})_{proj}$ of all bounded above complexes of projectives \cite[Definition 1.3.2.7]{lurie-ha}, or equivalently by taking the dg nerve of the dg category $\Ch^-(\mc{A})$ of all bounded below complexes and then inverting quasi-isomorphisms \cite[Prop.~1.3.4.6]{lurie-ha}. 
 
 When $\mc{A}$ is a Grothendieck abelian category, the unbounded derived $\infty$-category $\D(\mc{A})$ is constructed in \cite[Definition 1.3.5.8]{lurie-ha}, and it has $\D^+(\mc{A})$ sitting inside it as a full subcategory \cite[Rem.~1.3.5.10]{lurie-ha}. If, in addition, $\mc{A}$ has enough projectives, then $\D^-(\mc{A})$ sits inside $\D(\mc{A})$ fully faithfully as the full subcategory of complexes whose cohomology is bounded above \cite[Prop.~1.3.5.24]{lurie-ha}. We note that the definition of $\D(\mc{A})$ in \cite[Definition 1.3.5.8]{lurie-ha} is in terms of a model structure on the underlying category of the dg category $\Ch(\mc{A})$ of all chain complexes in $\mc{A}$ (see \cite[Prop.~1.3.5.3]{lurie-ha} for the definition of this model structure). Then $\D(\mc{A})$ is by definition the dg nerve of the full dg subcategory $\Ch(\mc{A})_{fib}$ of $\Ch(\mc{A})$ consisting of all fibrant objects. Alternatively, $\D(\mc{A})$ can be described as the underlying $\infty$-category associated with the model structure on the underlying category of $\Ch(\mc{A})$, see \cite[Prop.~1.3.5.15]{lurie-ha}.
 
We finish by making a remark about fully faithful functors of $\infty$-categories. By definition (see \cite[Def.\ 1.2.10.1]{lurie-htt}), a functor $F : \mc{C} \to \mc{D}$ of $\infty$-categories is fully faithful if it induces an equivalence of mapping spaces $\Hom(X,Y) \to \Hom(F(X),F(Y))$ for all objects $X,Y \in \mc{C}$. When $\mc{C}$ and $\mc{D}$ are stable and $F$ is exact, it suffices to check that $\pi_0(\Hom(X,Y)) \to \pi_0(\Hom(F(X),F(Y)))$ is a bijection for all $X,Y \in \mc{C}$, i.e.\ that $F$ induces a fully faithful functor of the underlying triangulated categories. When the stable $\infty$-categories arise as dg nerves (which they will in all cases of interest to us), this is easy to see from the alternative construction of the dg nerve in \cite[Const.\ 1.3.1.16]{lurie-ha}.

\subsubsection*{Categories of sheaves on $X$} 
We now apply this to coherent and quasicoherent sheaves on $X$. Recall that $\QCoh(X)$ is a Grothendieck abelian category; see e.g.\ \cite[\href{https://stacks.math.columbia.edu/tag/0781}{Tag 0781}]{stacks-project} (though one can give a much more direct proof in this special case). We then define $\D^+_{qcoh}(X)$ and $\D_{qcoh}(X)$ as $\D^+(\QCoh(X))$ and $\D(\QCoh(X))$, respectively. We also define $\D^b_{coh}(X)$ as the full subcategory of $\D^+(\QCoh(X))$ of complexes whose cohomology is bounded, and coherent in each degree.

\begin{remark}\label{issues with qc sheaves}
A different, perhaps more standard, definition of the unbounded derived category of quasicoherent sheaves on $X$ is as the unbounded derived category of complexes of lisse-\'etale $\oo_X$-modules with quasicoherent cohomology. Unlike the situation of abelian categories, these different definitions can produce genuinely different categories, as we now recall. Let us use $\D^\prime_{qcoh}(X)$ to denote the category constructed using the lisse-\'etale site; $\D^\prime_{qcoh}(X)$ will not be used anywhere else in this paper. The relationship between $\D_{qcoh}(X)$ and $\D^\prime_{qcoh}(X)$ is as follows: there is a natural functor $\D_{qcoh}(X) \to \D^\prime_{qcoh}(X)$ which identifies $\D^\prime_{qcoh}(X)$ as the left completion of $\D_{qcoh}(X)$; see \cite[Rem.~C.4]{hall-neeman-rydh}. This functor can fail to be an equivalence. Indeed, in the situation of \S \ref{subsec: non-generic I}, where $G=\SL_2$, the functor is not full by \cite[Thm.~1.3]{hall-neeman-rydh}. Moreover, in this situation, the category $\D_{qcoh}(X)$ is not compactly generated (this also follows from the results of \cite{hall-neeman-rydh}). These issues do not arise for the bounded below derived category, cf.\ e.g.\ \cite[Claim 2.7]{arinkin-bezrukavnikov}.
\end{remark}

\begin{remark}\label{remark:duality}
The following remarks about coherent duality will be useful. Assume that $A$ is Gorenstein. This is the case, for example, when $A$ is regular or equal to $B/(f)$ where $B$ is regular and $f$ is a nonzerodivisor \cite[Tags \href{https://stacks.math.columbia.edu/tag/0AWX}{0AWX} and \href{https://stacks.math.columbia.edu/tag/0BJJ}{0BJJ}]{stacks-project}; this covers all cases we will encounter. Assume further that $A$ has a dualizing complex (in all cases we consider, this easily follows from the explicit descriptions of the rings that we will give, together with \cite[\href{https://stacks.math.columbia.edu/tag/0BFR}{Tag 0BFR}]{stacks-project}). Since we assumed that $A$ is Gorenstein, $A$ itself (in degree $0$) is a dualizing complex for $A$ \cite[\href{https://stacks.math.columbia.edu/tag/0DW9}{Tag 0DW9}]{stacks-project}. If $M$ and $N$ are $G$-equivariant $A$-modules, let $\Hom_A(M,N)$ denote the (not necessarily $G$-equivariant) $A$-module homomorphisms from $M$ to $N$, with its induced $G$-action. This is the internal Hom in $\QCoh(X)$ and we will use the notation
\[
\Homi(M,N) := \Hom_A(M,N).
\]
Indeed, $\Homi(-,-)$ will denote the internal Hom in any category where it exists. Moreover, we let $\RHomi(M,N) :=\RHom_A(M,N)$ denote the derived functors of $\Homi$. Then
\[
\RHomi(-,\OO_{X}) = \RHom_A(-,A) : \D^b_{coh}(X) \to \D^b_{coh}(X)
\]
is an exact\footnote{In the sense of stable $\infty$-categories.} involution\footnote{In other words, $\oo_X$ is a dualizing complex for $X$; cf. \cite[Definition 2.16]{arinkin-bezrukavnikov}.}, i.e.\ an antiequivalence whose square is naturally isomorphic to the identity. Here $A$ is viewed as a $G$-equivariant $A$-module. In particular, we obtain an exact\footnote{In the sense of exact categories.} involution 
\[
\Homi(-,\oo_X) = \Hom_A(-,A) : \mathrm{MCM}(X) \to \mathrm{MCM}(X)
\]
where $\mathrm{MCM}(X)\sub \Coh(X)$ is the exact full subcategory of maximal Cohen--Macaulay modules (a $G$-equivariant finitely generated $A$-module is (maximal) Cohen--Macaulay if the underlying $A$-module is (maximal) Cohen--Macaulay). To simplify the notation, we will write
\[
M^\ast := \Homi(M,\oo_X)
\]
for the coherent dual of $M \in \mathrm{MCM}(X)$.
\end{remark}

\section{Stacks of Galois representations for $\GL_2/\Q_p$}\label{chap: stacks for GL2Qp}

The first goal of this section is to specify explicit presentations for the stack of Langlands parameters for $\GL_2/\Q_p$, which we take to be representations of $\Gamma := \Gamma_{\Q_p} = \Gal(\oQ_p/\Q_p)$. We put emphasis on the comparison of the moduli stack of representations to the moduli space of pseudorepresentations, which equals the coarse moduli space of representations in the sense of geometric invariant theory. Then, we compute all of these objects. We break up into cases according to block type. 

As far as notation, we make the following adjustments from Notation \ref{notn: stacks}, with the main adjustment being that we drop the subscript $D$ that denoted the choice of residual pseudorepresentation. 
\begin{notation}
Denote the algebraized stack of Galois representations with residual pseudorepresentation $D$ and constant determinant $\zeta\varep$ as $\fX := \Rep(E)$, where $E = E_D$ is the Cayley--Hamilton algebra associated to the semi-simple representation $\rho = \rho_D$. Likewise, $\hat \fX := \widehat{\Rep(E)}$, where the completions refer to $\m = \m_D$-adic completion; here $(R,\m) = (R_D,\m_D)$ is the (local) pseudodeformation ring of $D$, which is the scalar ring of $E$. We also drop the subscript $D$ from each object mentioned in Notation \ref{notn: stacks}. 
\end{notation}

The goal is then to explicitly describe $\fX$ for each of four types of $\rho$, which match the four types of block for $G = \GL_2(\Q_p)$ (enumerated in \S\ref{subsec: intro proofs}), 
\begin{enumerate}
\item $\rho$ is irreducible: the supersingular case
\item\label{type2} $\rho$ is reducible such that $\rho \simeq \chi_1 \oplus \chi_2$ with $\chi_1\chi_2^{-1} \not\simeq 1, \omega^{\pm 1}$: the generic principal series case
\item $\rho$ is a scalar representation: non-generic principal series case I (``non-generic I'') 
\item\label{type4} $\rho$ is a twist of $\omega \oplus 1$: non-generic principal series case II (``non-generic II''). 
\end{enumerate}

Using these explicit descriptions, for $\rho$ of type (\ref{type2})--(\ref{type4}) (no further work being required in the supersingular case), we describe certain coherent sheaves $X$ on $\fX$ and compute the $\Ext$ groups $\Ext^i(X,X)$ (in particular, showing they vanish for $i \ne 0$). These coherent sheaves will be used to define functors as sketched in \S\ref{subsec: intro proofs}. In order to compute these functors, we will use projective resolutions of the simple $\End(X,X)$-modules which we write down in this section.

We now proceed to consider each case (1)--(4) in turn.

\subsection{Supersingular case}\label{subsec: supersingular case}

In the supersingular case, we know from the discussion of \S\ref{sssec: irred case} that
\begin{itemize}
\item $R \cong R_{\rho} \hookrightarrow S \hookrightarrow \hat S$  
\item $A \hookrightarrow \hat A \hookrightarrow R^\square_\rho$
\item $E \simeq M_2(R)$ 
\item $\fX \cong [\Spec R/\mu_2]$, where $\mu_2$ acts trivially. 
\end{itemize}

\begin{theorem}
There are isomorphisms 
\[
R \cong R_{\rho} \simeq \OO \lb X_1, X_2, X_3\rb, 
\]
and a choice of isomorphism $E \cong M_2(R)$ gives rise to an isomorphism between the maps $R \to A \twoheadrightarrow S$ and 
\[
\OO \lb X_1, X_2, X_3\rb  \to \OO \lb X_1, X_2, X_3\rb[\PGL_2]  \twoheadrightarrow \OO \lb X_1, X_2, X_3\rb[\bG_m].
\]
where the implicit closed immersion $\bG_m \hookrightarrow \PGL_2$ is the standard torus. 
\end{theorem}

\begin{proof}
The first claimed isomorphism follows from the fact that $H^2(\Q_p, \Ad^0 \rho) = 0$ according to standard deformation-theoretic arguments, as this $H^2$ is the obstruction space \cite[\S1.6]{mazur1989}. Upon a choice of identification $E \isoto M_2(R)$, the isomorphism $A \isoto R[\PGL_2]$ arises from interpreting $\PGL_2$ as the group scheme of ring scheme automorphisms of $M_2$. Then the interpretation of $A \twoheadrightarrow S$ amounts to observation that the standard torus in $\PGL_2$ is cut out by the condition that an automorphism of $M_2$ fixes its two standard idempotents. 
\end{proof}

\begin{remark}
The proof relies on the running assumption that $p \geq 5$. For instance, when $p=3$ and $\rho$ is induced from a character of $\Gal(\oQ_3/\Q_3(\zeta_3))$, $R_{\rho}$ can be obstructed. 
\end{remark}

\subsection{Generic principal series}
\label{subsec: stacks generic ps GL2}

In this case, $\rho \simeq \chi_1 \oplus \chi_2$ where $\chi_1\chi_2^{-1} \neq 1, \omega^{\pm1}$. Because $\chi_1\neq \chi_2$, we may and do choose the additional structures discussed in \S\ref{subsec: GMAs}, such as a GMA structure on $E$ and the resulting adapted moduli functor represented by $S$. In particular, we use a $\Z$-grading of $S$ to represent the $T = \bG_m$-action on the moduli scheme $\Spec S = \Rep^{\Ad,\square}(E)$ of adapted representations. 
\begin{theorem}
\label{computation of rings generic ps}
There is an isomorphism of $\Z$-graded rings
\[
S \cong \oo \lb a_0, a_1, bc\rb [b, c],
\]
where $b$ has graded degree $2$, $c$ has graded degree $-2$, and the remaining generators have degree $0$. The pseudodeformation ring $R$ is the degree 0 subring $R = S_0 \subset S$, 
\[
R \cong \OO\lb a_0, a_1, bc\rb 
\]
and its reducibility ideal is generated by $bc$. The universal Cayley--Hamilton algebra admits $R$-GMA form
\[
E \cong 
\begin{pmatrix}
R & Rb \\
Rc & R
\end{pmatrix}
\]
where the cross-diagonal multiplication is given by 
\[
Rb \times Rc \to Rbc \subset R, \quad (xb,yc) \mapsto xybc. 
\]
\end{theorem}

\begin{remark}
\label{rem: hatS}
For clarity in the proofs of Theorems \ref{computation of rings generic ps} and \ref{presentation of moduli stack non-generic 2}, we point out that $\hat S$ denotes the $\m$-adic completion of $S$ where $\m \subset R$ is the maximal ideal of the pseudodeformation ring of $D$. In terms of the usual variable names $a_i, b_j, c$ we use to denote generators of $R$, a set of generators of $\m$ comprises terms of the form $a_i$ and $b_j c$. 
\end{remark}

\begin{proof}
We know from Theorem \ref{presentation of moduli stack} that $R = S_0$. Straightforward calculations in local Galois cohomology yield that 
\[
H^2(\Q_p, \ad \rho) = 
\begin{pmatrix}
H^2(\Q_p, \F) & H^2(\Q_p, \chi_1\chi_2^{-1}) \\ 
H^2(\Q_p, \chi_2\chi_1^{-1}) & H^2(\Q_p, \F)
\end{pmatrix} = 0,
\]
by the genericity assumption $\chi_1\chi_2^{-1} \neq 1, \omega^{\pm 1}$. Therefore the deformation theory of $\rho$ is unobstructed. Likewise, the tangent space (mod $p$) of $\fX$ at $\rho$ decomposes as 
\[
H^1(\Q_p, \ad^0 \rho) \cong H^1(\Q_p, \chi_1\chi_2^{-1}) \oplus H^1(\Q_p, \F) \oplus H^1(\Q_p, \chi_2\chi_1^{-1}),
\]
whose summands have $\F$-dimensions $1 \oplus 2 \oplus 1$. Then the theorem \cite[Thm.\ 11.3.3]{wang-erickson-A-inf} (see Proposition \ref{prop: abstract presentation}) establishes an explicit presentation of $\hat S/p \hat S$ as 
\[
\hat S/p \hat S \cong \F[b,c] \lb \bar a_0, \bar a_1, \bar b \bar c\rb, 
\]
where $\{\bar b\}$, $\{\bar a_0, \bar a_1\}$, and $\{\bar c\}$ are dual bases to the summands of $H^1(\Q_p, \ad^0 \rho)$, respectively, and where these bases have graded degrees $2$, $0$, and $-2$ respectively. In this way, $\hat S/\varpi\hat S$ is the coordinate ring of a formal algebraic stack in which each element of its defining colimit is a quotient stack of $\hat S/(\varpi,(a_0, a_1, bc)^n)$ by the action of $\bG_m$. 

By Proposition \ref{prop: GMA description} and Nakayama's lemma, $E$ has off-diagonal summands generated by lifts $b$ of $\bar b$ and $c$ of $\bar c$  respectively, of identical graded degrees. Likewise, denote by $a_i \in R$ lifts of $\bar a_i \in \bar R$ for $i=0,1$. Thus (again using Proposition \ref{prop: GMA description}) we arrive at a presentation of graded rings $\oo\lb a_0, a_1, bc\rb [b,c] \twoheadrightarrow S$ that we wish to show is an isomorphism. 

The vanishing of $H^2$ implies that the completion $A^\wedge_\rho$ of $A$ at its maximal ideal corresponding to $\rho$ is formally smooth over $\Spf \Z_p$, since $A^\wedge_\rho$ is the framed deformation ring at $\rho$. Therefore $\fX^\wedge_\rho$ and $\Spec S^\wedge_\rho$ (the respective completions at the point corresponding to $\rho$) are formally smooth at $\rho$ as well, since these spaces are connected by smooth presentations $\Spec A \to \fX \leftarrow \Spec S$ as quotient stacks by $\SL_2$ and $\bG_m$, respectively. Because $S^\wedge_\rho$ is formally smooth, then so is $S$, because $\fX$ is coherently complete at $\rho$ (by \cite[Thm.\ 1.6]{alper-hall-rydh-etale-local}). Indeed, the equivalence of categories of coherent sheaves of \textit{loc.\ cit.}, i.e.\ a ``formal GAGA'' result, implies in an equivalence of ideal sheaves under completion and therefore an equivalence of closed subschemes; this implication is just like the classical case for formal schemes, in which an equivalence of closed subschemes \cite[Cor.\ 5.1.8]{ega3-1} is deduced from a formal GAGA result \cite[Cor.\ 5.1.3]{ega3-1}. Therefore the presentation map is an isomorphism, as desired. 
\end{proof}

We wish to discuss some line bundles on $\fX = \Rep(E)$, which we present, by Theorem \ref{presentation of moduli stack}, as $\fX \cong [\Spec S/\bG_m]$. 
In particular, coherent sheaves (resp.\ vector bundles) on $\fX$ are equivalent to finitely generated $\Z$-graded $S$-modules (resp.\ finitely generated $\Z$-graded $S$-modules which are projective as $S$-modules), where we regard $S$ as a graded ring as in Theorem \ref{computation of rings generic ps}. We refer back to \S \ref{subsec: notation} for our notational conventions regarding graded rings and modules. For $m\in \Z$, we define the graded $S$-module $L_m$ as $S(m)$, i.e.\ 
\[
(L_m)_k = S_{m+k}.
\]
This is a line bundle on $\fX$. If $\mc{V}$ is the vector bundle on $\fX$ underlying the universal representation, then we observe that its corresponding graded $S$-module is $L_1 \oplus L_{-1}$. From this, we get the following theorem.

\begin{theorem}\label{Ext groups of universal rep generic ps}
We have $\End(\mc{V}) = E$ as rings. Moreover, any locally free object of $\QCoh(\mf{X})$ is projective in $\QCoh(\mf{X})$. In particular, $\mc{V}$ is projective.
\end{theorem}

\begin{proof}
We equate $\QCoh(\mf{X})$ with the category of graded $S$-modules. That $M$ is locally free means that $M$ is projective as an $S$-module, i.e.\ that $\Exti^i(M,N)=0$ for all $i\geq 1$ and all $N$. The global sections functor is then $M \mapsto M_0$, and is hence exact, so we see that $\Ext^i(M,N)=\Exti^i(M,N)_0=0$ for locally free $M$ (with $N$ arbitrary and $i\geq 1$), proving the assertions about projectivity. 

The claim about $\End(\mc{V})$ amounts to the conclusion of Theorem \ref{presentation of moduli stack}. To make this clear, we compute 
\[
\Hom_S(L_m, L_n)_0 = \Hom_S(S,L_{n-m})_0 = S_{n-m}
\]
for all $m,n\in \Z$ to see that
\[
\End_S(L_1 \oplus L_{-1})_0 = \begin{pmatrix} L_0 & L_2 \\ L_{-2} & L_0 \end{pmatrix}_0 = \begin{pmatrix} R & Rb \\ Rc & R \end{pmatrix}
\]
and one easily checks that the multiplication matches.
\end{proof}

\begin{remark}
The isomorphism $\End(\mc{V}) =E$ when the automorphism group is linearly reductive, another version of which is Theorem \ref{presentation of moduli stack}(5), goes back at least to Procesi \cite[Thm.\ 2.6]{procesi}. 
\end{remark}

\subsection{Non-generic case I}\label{subsec: non-generic I}

In this case, the underlying pseudorepresentations are deformations of the trivial pseudorepresentation, and the determinant is trivial, after twisting. The pseudodeformation ring $R$ and the Cayley--Hamilton algebra $E$ were studied by Pa{\v s}k{\=u}nas \cite[Appendix A and \S 9]{paskunas-image}, where it was shown that they are equal to the corresponding object for the maximal pro-$p$ quotient $\mc{G}$ of $\Gamma$ \cite[Cor.\ A.3, A.4]{paskunas-image}. It is well known that $\mc{G}$ is a free pro-$p$ group on two generators, which greatly helps in the study of $R$ and $E$. Continue denoting by $\mc{V}$ the vector bundle of the universal representation on $\fX$. We let $\Spec A$ be the affine scheme representing $\Rep^\square(E)$. In this section, we will prove the following result.

\begin{theorem}
\label{thm: good invariants}
The natural map $R \to A^{\SL_2}$ is an isomorphism, and $\Ext^i(\mc{V},\mc{V})=0$ for all $i\geq 1$. Moreover, the natural map
\[
E \to M_2(A)^{\SL_2} = \End(\mc{V})
\]
is an isomorphism.
\end{theorem}

\begin{remark}
Bella\"iche--Chenevier highlighted the question of whether $E$ is always isomorphic to the adjoint invariants of $M_2(A)$ in general (an ``embedding problem'' \cite[\S1.3.4]{BC2009}). Recently, Jinyue Luo constructed an example in characteristic 2 where $\Gamma$ is a finite 2-group, $\rho$ is the trivial 2-dimensional representation, and $E \to M_2(A)^{\SL_2}$ has a non-zero kernel, showing that the answer is ``no'' \cite{luo-embedding}. More specifically, in contrast to the residually multiplicity free case described in Theorem \ref{presentation of moduli stack}, Luo identifies non-zero elements of the kernel of the map of scalar subrings $R \to A^{\SL_2}$, showing that the pseudodeformation ring is sometimes not isomorphic to the adjoint invariant subring of $A$ in non-multiplicity-free cases. In particular, we emphasize that the validity of Theorem \ref{thm: good invariants} does not follow from some general theory that applies to all groups $\Gamma$ and all residual representations $\rho$. 
\end{remark}

We will prove the statements in Theorem \ref{thm: good invariants} in the order they are mentioned. For the first and second part, we will make use of the notion of a \emph{good filtration} on algebraic representations of reductive groups over $\Fpbar$, which is summarized briefly in \cite[\S VIII.5.1]{fargues-scholze}. For details on standard constructions in the representation theory of algebraic groups we refer to \cite{jantzen}. 

Let $H/\Fpbar$ be a connected reductive group and let $T \sub B \sub H$ be a maximal torus and a Borel subgroup of $H$, respectively. For a dominant weight $\lambda$, let $\oo(\lambda)$ denote the corresponding standard line bundle on $H/B$ and set
\[
\nabla_\lambda := H^0(H/B,\oo(\lambda)).
\]
A descending filtration $(V_i)$ ($i \in \Z$) of $H$-subrepresentations of an $H$-representation $V$ is said to be \emph{good} if the successive quotients $V_i/V_{i-1}$ are isomorphic to direct sums of $\nabla_\lambda$s. Given a total ordering $0=\lambda_0,\lambda_1,\dots$ of the dominant weights, compatible with the dominance ordering, then we can choose $V_i$ to be the maximal subrepresentation of $V$ with weights $\lambda_j$ for $j\leq i$, and $V$ has a good filtration if and only $V_i/V_{i-1}$ is isomorphic to a direct sum of copies of $\nabla_{\lambda_i}$. An $H$-representation $V$ has a good filtration if and only if $H^i(H,V\otimes \nabla_\lambda) = 0$ for all $i\geq 1$ and all $\lambda$ \cite{donkin-filtration}. In particular, $H^i(H,V) = 0$ for all $i\geq 1$ if $V$ has a good filtration.

To prove that $R = A^{\SL_2}$, we begin by recalling the following result of Donkin \cite[\S 3.1]{donkin-invariants}. For simplicity, we specialize to $\GL_2$, which is the case we need. For any $r \in \Z_{\geq 1}$ and any function $\sigma : \{1,\dots,r \} \to \{1,2\}$, define a function 
\[
t_{r,\sigma}(g_1,g_2) := \mathrm{trace}(g_{\sigma(1)} \dots g_{\sigma(r)})
\]
on $\GL_2^2$. Moreover, set $d_i(g_1,g_2)= \det(g_i)$ for $i=1,2$.

\begin{theorem}[{Donkin \cite{donkin-invariants}}] 
\label{donkin}
Let $\GL_2$ act on $\GL_2^2$ by diagonal conjugation, and let $\oo[\GL_2^2]$ be the ring of functions of the group scheme $\GL_2^2$ over $\oo$. Then the ring of invariants $\oo[\GL_2^2]^{\GL_2}$ is generated by the functions $t_{r,\sigma}$ together with $d_1^{\pm 1}$ and $d_2^{\pm 1}$.
\end{theorem}

From this, we deduce the following corollary.

\begin{corollary}
Let $\SL_2$ act on $\SL_2^2$ by diagonal conjugation, and let $\oo[\SL_2^2]$ be the ring of functions of the group scheme $\SL_2^2$ over $\oo$. Then the ring of invariants $\oo[\SL_2^2]^{\SL_2}$ is generated by the functions $t_{r,\sigma}$.
\end{corollary}

\begin{proof}
We may regard $\oo[\SL_2^2]$ as a $\GL_2$-representation, acting by diagonal conjugation; clearly $\oo[\SL_2^2]^{\GL_2}=\oo[\SL_2^2]^{\SL_2}$. The restriction map $\oo[\GL_2^2]\to \oo[\SL_2^2]$ is surjective, and is the first part of a Koszul resolution
\[
0 \to \oo[\GL_2^2] \to \oo[\GL_2^2]^2 \to \oo[\GL_2^2] \to \oo[\SL_2^2] \to 0,
\]
since $\SL_2^2$ is a complete intersection in $\GL_2^2$ cut out by the equations $d_1=d_2=1$. If $\oo[\GL_2^2]$ has vanishing higher cohomology, the Koszul resolution together with elementary considerations of long exact sequences in cohomology shows that $\oo[\GL_2^2]^{\GL_2} \to \oo[\SL_2^2]^{\GL_2}$ is surjective, and the result then follows from Theorem \ref{donkin}. Therefore it remains to show that $\oo[\GL_2^2]$ has vanishing higher cohomology. By \cite[Thm.~10.5]{vanderkallen}, each $H^i(\GL_2,\oo[\GL_2^2])$ is a finitely generated module over the finitely generated $\oo$-algebra $\oo[GL_2^2]^{\GL_2}$, so it suffices to show that $H^i(\GL_2,\oo[\GL_2^2])\otimes_{\oo}\ol{L} = 0$ and $H^i(\GL_2,\oo[\GL_2^2])\otimes_{\oo}\ol{\F} = 0$ for $i\geq 1$. We have $H^i(\GL_2,\oo[\GL_2^2])\otimes_{\oo}\ol{L} = H^i(\GL_2,\ol{L}[\GL_2^2]) = 0$ for $i\geq 1$, where the first equality comes from \cite[I.4.18, Prop.]{jantzen} and the second comes from $\GL_2$ being reductive and $\ol{L}$ having characteristic $0$. Finally, we also have $H^i(\GL_2,\oo[\GL_2^2])\otimes_{\oo}\ol{\F} \hookrightarrow H^i(\GL_2,\ol{\F}[\GL_2^2]) = 0$ for $i\geq 1$, where the injection comes from \cite[I.4.18, Prop.]{jantzen} and the equality $H^i(\GL_2,\ol{\F}[\GL_2^2]) = 0$ holds because $\ol{\F}[\GL_2^2]$ has a good filtration, by \cite[Cor.~VIII.5.7]{fargues-scholze}.
\end{proof}

Let $F$ be the free group on two generators; its pro-$p$ completion is $\mc{G}$. Attached to $F$, we have its $\SL_2$-representation variety, which is isomorphic to $\SL_2^2 = \Spec A_F$, its character variety $\Spec A_F^{\SL_2}$ (that is, the GIT quotient $\SL_2^2 /\!/ \SL_2$), and its moduli variety of pseudorepresentations $\Spec R_F$, all taken over the base $\oo$. There is a canonical map $R_F \to A_F^{\SL_2}$, which is an adequate homeomorphism by \cite[Thm.~6.0.5(iv)]{emerson} (cf.\ the $\GL_d$ case in \cite[Thm.\ 2.20]{wang-erickson-algebraic}). By Theorem \ref{donkin}, it is also surjective. To show that $R_F \to A_F^{\SL_2}$ is an isomorphism, it therefore suffices to prove that $R_F$ is reduced. In fact, we may compute $R_F$. 

\begin{proposition}\label{pseudorep ring is GIT quotient}
$R_F$ is isomorphic to a polynomial ring over $\oo$ in three variables. In particular, $R_F$ is reduced and the map $R_F \to A_F^{\SL_2}$ is an isomorphism.
\end{proposition}

\begin{proof}
The second part follows from the first part and the discussion above, so it remains to prove the first part. We obtain a map $\phi : \oo[s_1,s_2,s_3] \to R_F$ by sending, at the level of $B$-points for $B$ an arbitrary $\oo$-algebra, a pseudorepresentation $T : F \to B$ to the tuple $(T(\gamma),T(\delta),T(\gamma\delta))$, where $\gamma$ and $\delta$ are generators of $F$. Since a two-dimensional pseudorepresentation $T$ over $\oo$ with trivial determinant satisfies the identity
\begin{equation}\label{eq: pseudorep identity}
T(g^{-1}h) - T(g)T(h) + T(gh) = 0
\end{equation}
(see \cite[Lem.~1.9]{chenevier-determinant}\footnote{This identity follows from the pseudorepresentation identity applied to the three elements $g$, $g$ and $g^{-1}h$.}) for any $g,h\in F$, \cite[Lem.~9.10]{paskunas-image} implies that $\phi$ induces an injection $f \mapsto f \circ \phi$ at the level of functors of points (cf.~\cite[Cor.~9.11]{paskunas-image}). Proving that $f \mapsto f \circ \phi$ is surjective on $B$-points for any $\oo$-algebra $B$ (and hence that $\phi$ is an isomorphism) is then equivalent to showing that there is a map $\psi : R_F \to \oo[s_1,s_2,s_3]$ such that $\psi \circ \phi$ is the identity. In terms of the moduli problem, this means that we need to construct a pseudorepresentation $T^{\mathrm{univ}} : F \to \oo[s_1,s_2,s_3]$ with $T(\gamma) = s_1$, $T(\delta) = s_2$ and $T(\gamma\delta) = s_3$.
 
For the sake of brevity\footnote{One can also construct $T^{\mathrm{univ}}$ directly, as the trace of an `algebraic' version of the representation $\rho$ from \cite[Prop.~9.8]{paskunas-image}.}, we construct $T^{\mathrm{univ}}$ from the representation $\rho : \mc{G} \to \SL_2(C)$ constructed in \cite[Prop.~9.8]{paskunas-image}. Here $C$ is a ring that is finite over $\oo \lb t_1, t_2, t_3 \rb$, and the trace $T_\rho = \mathrm{tr}(\rho)$ satisfies $T_\rho(\gamma) = 2 + 2t_1$, $T_\rho (\delta) = 2+2t_2$ and $T_\rho(\gamma\delta) = 2+2t_3$. Let $T^\prime$ denote the restriction of $T_\rho$ from $\mc{G}$ to $F$. By equation (\ref{eq: pseudorep identity}) and \cite[Lem.~9.10]{paskunas-image}, $T^\prime$ takes values in $\oo[t_1,t_2,t_3]$. A simple change of variables then gives the desired pseudorepresentation $T^{\mathrm{univ}}$.
\end{proof}

\begin{lemma}\label{automatic continuity}
\begin{enumerate}
\item The completed local ring of $R_F$ at the trivial pseudorepresentation is isomorphic to $R$. In particular, the natural map $R_F = A_F^{\SL_2} \to R$ is flat.

\smallskip

\item Let $B$ be an $R$-algebra and let $\rho_B : R[F] \to M_2(B)$ be a representation whose pseudorepresentation is equal to the universal pseudorepresentation of $F$ composed with the composition $R_F \to R \to B$. Then $\rho_B$ factors through the natural map $R[F] \to R\lb \mc{G} \rb$.
\end{enumerate}

\end{lemma}

\begin{proof}
The first part follows from the proof of Proposition \ref{pseudorep ring is GIT quotient} and Pa{\v s}k{\=u}nas's analogous result for $R$ \cite[Cor.~9.13]{paskunas-image}.

For the second part, we first note that $\rho_B$ factors through the Cayley--Hamilton quotient $E_R$ of $R[F]$ with respect to the specialization of the universal pseudorepresentation to $R$. Since $R[F]$ is finitely generated over $R$, $E_R$ is a finite $R$-module \cite[Prop.~2.13]{wang-erickson-algebraic}, and is in particular $\m_R$-adically complete. For each $n \ge 1$, the quotient $E_R \otimes_R R/\m_R^n$ is a finite length $R$-module, so the map $R[F] \to E_R \otimes_R R/\m_R^n$ factors through  $R[F/H]$ for a finite quotient $F/H$ of $F$ (consider the induced map $F \to \left(E_R \otimes_R R/\m_R^n \right)^\times$). The proof of \cite[Lem.\ 3.8]{chenevier-determinant} now shows that we can take $F/H$ to be a $p$-group. Indeed, we have a Cayley--Hamilton pseudorepresentation $D_n: E_R \otimes_R R/\m_R^n \to R/\m_R^n$ and \cite[Lem.\ 2.10, Thm.\ 2.16]{chenevier-determinant} implies that the radical $\mathcal{R}$ of $E_R \otimes_R R/\m_R^n$ satisfies $(E_R \otimes_R R/\m_R^n)/\mathcal{R}\cong M_2(k)$, with the induced representation of $F$ equal to the trivial representation. In particular, the image of $F/H$ in $E_R \otimes_R R/\m_R^n$ lies in $1+\mathcal{R}$ which is a $p$-group. Taking the limit over $n$ shows that the map $R[F] \to E_R$ factors through $R\lb \mc{G} \rb$, and we are done.
\end{proof}

\begin{corollary}\label{base change for A}
We have $A = A_F \otimes_{R_F} R$.
\end{corollary}

\begin{proof}
From the definitions, $A_F$ is the representation ring for the Cayley--Hamilton quotient $E_F$ of $R_F[F]$ with respect to the universal pseudorepresentation $R_F[F] \to R_F$, and $A$ is the representation ring for the Cayley--Hamilton quotient $E$ of $R \lb \mc{G} \rb \to R$. By compatibility of Cayley--Hamilton quotients with base change \cite[\S 1.17]{chenevier-determinant}, $A_F\otimes_{R_F} R$ is the representation ring for the Cayley--Hamilton quotient $E_F \otimes_{R_F}R$ of $R[F]$ with respect to the pseudorepresentation $R[F] \to R$. In particular, to show that $A = A_F \otimes_{R_F} R$ it suffices to show that the natural $R$-linear map $E_F \otimes_{R_F} R \to E$ induces bijections
\[
\iota_B : \Rep^\Box(E)(B) \to \Rep^\Box(E_F \otimes_{R_F}R)(B)
\]
for all $R$-algebras $B$. Since $F$ is dense in $\mc{G}$, the map $E_F \otimes_{R_F} R \to E$ has dense image, which implies that it is surjective since both sides are finite $R$-modules. This gives injectivity of $\iota_B$, and surjectivity then follows from Lemma \ref{automatic continuity}(2). 
\end{proof}

Next we record some properties of cohomology that we will need.

\begin{lemma}\label{finite generation and flat base change}
Let $V$ be a finitely generated $A_F$-module on which $\SL_2$ acts compatibly. Then $H^i(\SL_2,V\otimes_{R_F}R) = H^i(\SL_2,V)\otimes_{R_F}R$ for all $i \geq 0$. In particular, $R = A^{\SL_2}$. Moreover, $H^i(\SL_2,V)$ is a finitely generated $R_F$-module and $H^i(\SL_2,V\otimes_{R_F}R)$ is a finitely generated $R$-module. 
\end{lemma}

\begin{proof}
Since $R$ is flat over $R_F$, we may write $R = \varinjlim_j R_F^{m_j}$ as a direct limit of finitely generated free modules by Lazard's theorem. Since cohomology commutes with direct limits \cite[I.4.17, Lem.]{jantzen}, we see that
\begin{gather*}
H^i(\SL_2,V \otimes_{R_F}R) = H^i(\SL_2, \varinjlim_j V \otimes_{R_F} R_F^{m_j}) = \varinjlim H^i(\SL_2, V \otimes_{R_F} R_F^{m_j}) \\
= \varinjlim H^i(\SL_2, V) \otimes_{R_F} R_F^{m_j} = H^i(\SL_2,V) \otimes_{R_F}R
\end{gather*}
as desired. That $R = A^{\SL_2}$ then follows by setting $V=A_F$, $i=0$, and using Proposition \ref{pseudorep ring is GIT quotient} and Corollary \ref{base change for A}. Finally, by \cite[Thm.~10.5]{vanderkallen}, $H^i(\SL_2,V)$ is a finitely generated $R_F$-module, and hence $H^i(\SL_2,V\otimes_{R_F}R) = H^i(\SL_2,V) \otimes_{R_F}R$ is a finitely generated $R$-module.
\end{proof}

This proves that $R = A^{SL_2}$ as desired. We can now prove that $\Ext^i(\mc{V},\mc{V}) = 0$ for $i \geq 1$. 
\begin{proposition}\label{vanishing of exts non-generic I}
We have $\Ext^i(\mc{V},\mc{V}) = 0$ for $i \geq 1$. 
\end{proposition}

\begin{proof}
Let $\mathrm{ad}$ denote the adjoint representation of $\GL_2$, restricted to $\SL_2$, which is a direct sum of induced representations. Then we have 
\[
\Ext^i(\mc{V},\mc{V}) = H^i(\SL_2, A \otimes \mathrm{ad}) =  H^i(\SL_2, A_F \otimes \mathrm{ad}) \otimes_{R_F} R,
\]
where the last isomorphism follows from Lemma \ref{finite generation and flat base change}, and it is a finitely generated $R$-module. Since $R$ is local it suffices to prove that $H^i(\SL_2,A_F \otimes \mathrm{ad}) \otimes_{\oo}\ol{\F} = 0$ for $i \geq 1$. This cohomology group injects into $H^i(\SL_2,(A_F \otimes \mathrm{ad})\otimes_{\oo}\ol{\F})$ by \cite[I.4.18, Prop.]{jantzen}, and $H^i(\SL_2,(A_F \otimes \mathrm{ad})\otimes_{\oo}\ol{\F}) = 0$ since $A_F \otimes_{\oo}\ol{\F}$ has a good filtration by \cite[Cor.\ VIII.5.7]{fargues-scholze}.
\end{proof}

It remains to prove that $E \to \End(\mc{V})$ is an isomorphism. As above,
\[
\End(\mc{V}) = (A \otimes \mathrm{ad})^{\SL_2} = M_2(A)^{\SL_2}.
\]
We start by looking at the problem after inverting $\vp$. Then $A[1/\vp]$ is the representation ring for the Cayley--Hamilton algebra $E[1/\vp]$ and hence, by \cite[Thm.\ 2.6]{procesi}, the natural map
\[
E[1/\vp] \to M_2(A[1/\vp])^{\SL_2}
\]
is an isomorphism. Since $M_2(A[1/\vp])^{\SL_2} = M_2(A)^{\SL_2}[1/\vp]$, we see that $E \to M_2(A)^{\SL_2}$ is an isomorphism after inverting $\vp$. To prove that it is an isomorphism on the nose, we will need to study the map more explicitly.

We begin this by recalling the structure of $R$ and $E$ from \cite[\S 9.2]{paskunas-image}. Let $\gamma$ and $\delta$ be two generators of $F$. By \cite[Prop.\ 9.12, Cor.\ 9.13]{paskunas-image} we have $R = \oo \lb t_1, t_2, t_3 \rb$, where $2+2t_1$ is the trace of $\gamma$, $2+2t_2$ is the trace of $\delta$ and $2+2t_3$ is the trace of $\gamma \delta$. The ring $E$ is a free $R$-module of rank $4$ by \cite[Cor.~9.25]{paskunas-image}, with a basis given by elements $1$, $u$, $v$ and $uv-vu$, where
\[
u = \gamma -1 - t_1 , \,\,\,\,\,\,\, v = \delta - 1 - t_2;
\] 
recall that $E$ is a quotient of $R \lb \mc{G} \rb$. The ring $A$ may be described as a quotient of $R[a_i,b_i,c_i,d_i \mid i=1,2 \,]$, where the universal representation $E \to M_2(A)$ sends
\[
\gamma \mapsto \begin{pmatrix} 1+a_1 & c_1 \\ c_2 & 1+a_2 \end{pmatrix} , \,\,\,\,\,\,\, \delta \mapsto \begin{pmatrix} 1+b_1 & d_1 \\ d_2 & 1+b_2 \end{pmatrix}.
\]
We have five relations. The first four come from the trace and determinant of the image of $\gamma$ and $\delta$, and amount to 
\[
a_1 + a_2 = 2t_1, \,\,\, b_1 + b_2 = 2t_2, \,\,\, a_1 + a_2 + a_1 a_2 - c_1 c_2 = 0, \,\,\, b_1 + b_2 + b_1 b_2 - d_1 d_2 = 0.
\]
The fifth comes from the trace of $\gamma \delta$, and is
\[
a_1 + a_2 + b_1 + b_2 + a_1 b_1 + a_2 b_2 + c_1 d_2 + c_2 d_1 = 2 t_3.
\]

Let us now write the map $E \to M_2(A)$ explicitly as an $R$-module map, using the basis $1$, $\gamma-1$, $\delta-1$  and $uv-vu=\gamma \delta -\delta \gamma$. Clearly $1$ gets sent to the identity matrix, and from the descriptions above we see that
\[
\gamma -1  \mapsto \begin{pmatrix} a_1 & c_1 \\ c_2 & a_2 \end{pmatrix} , \,\,\,\,\,\,\, \delta -1  \mapsto \begin{pmatrix} b_1 & d_1 \\ d_2 & b_2 \end{pmatrix},
\]
and hence
\[
uv - vu \mapsto \begin{pmatrix} c_1 d_2 - c_2 d_1 & a_1 d_1 + b_2 c_1 - a_2 d_1 - b_1 c_1 \\ a_2 d_2 + b_1 c_2 - a_1 d_2 - b_2 c_2 & c_2 d_1 - c_1 d_2 \end{pmatrix}.
\]
Now consider a general element $X = \lambda_1 + \lambda_2 (\gamma -1) + \lambda_3 (\delta -1) + \lambda_4 (\gamma \delta - \delta \gamma) \in E[1/\vp]$. It gets sent to
\[
\begin{pmatrix} \lambda_1 + \lambda_2a_1 + \lambda_3 b_1  + \lambda_4 (c_1 d_2 - c_2 d_1) & \lambda_2 c_1 + \lambda_3 d_1 + \lambda_4 (a_1 d_1 + b_2 c_1 - a_2 d_1 - b_1 c_1) \\ \lambda_2 c_2 + \lambda_3 d_2 + \lambda_4 ( a_2 d_2 + b_1 c_2 - a_1 d_2 - b_2 c_2) & \lambda_1 + \lambda_2 a_2 + \lambda_3 b_2 + \lambda_4 (c_2 d_1 - c_1 d_2) \end{pmatrix}.
\]
These expressions are somewhat unwieldy to analyse. We will instead consider the quotient $C$ of $A$, introduced in \cite[Def.\ 9.7]{paskunas-image}, which is given by setting
\[
c_1 = 1, \,\,\, c_2 = 0, \,\,\, d_1 = 0, \,\,\, d_2 = 2t_3 - 2t_1 - 2t_2 - a_1 b_1 - a_2 b_2.
\]
With this, one gets the presentation 
\[
C = \frac{R[a_1,a_2,b_1,b_2]}{(a_1 + a_2 -2t_1, a_1 a_2 + 2t_1, b_1 + b_2 -2t_2, b_1 b_2 + 2t_2) },
\]
(we have changed some signs compared to \emph{loc.~cit.}, correcting apparent typos), which may be further simplified to
\[
C = \frac{R[a_1,b_1]}{(a_1^2 - 2t_1 a_1 - 2t_1, b_1^2 - 2t_2 b_2 - 2t_2)}.
\]
In particular, we see that $C$ is a biquadratic extension of $R$ and we get the following:

\begin{lemma}\label{freeness of C}
$C$ is a free $R$-module of rank $4$, with basis $1$, $a_1$, $b_1$, $a_1 b_1$.
\end{lemma}
\begin{proof}
	We can view $C$ as a quotient of the flat local $R$-algebra $R\lb a_1,b_1\rb$. Since $a_1^2,b_1^2$ is a regular sequence in $k\lb a_1,b_1\rb$, \cite[\href{https://stacks.math.columbia.edu/tag/00MG}{Tag 00MG}]{stacks-project} shows that $C$ is flat over $R$. $C$ is also clearly finite over $R$, so it is free and we can check for a basis modulo the maximal ideal of $R$.
\end{proof}

The composition of $E[1/\vp] \to M_2(A[1/\vp])$ with $M_2(A[1/\vp]) \to M_2(C[1/\vp])$ is then given by sending the general element $X = \lambda_1 + \lambda_2 (\gamma -1) + \lambda_3 (\delta -1) + \lambda_4 (\gamma \delta - \delta \gamma) \in E[1/\vp]$ to
\[
\begin{pmatrix} \lambda_1 + \lambda_2a_1 + \lambda_3 b_1  + \lambda_4 d_2  & \lambda_2 + \lambda_4 (b_2 - b_1) \\  \lambda_3 d_2 + \lambda_4 ( a_2 d_2 - a_1 d_2) & \lambda_1 + \lambda_2 a_2 + \lambda_3 b_2 - \lambda_4 d_2 \end{pmatrix}.
\]
With these preparations, we now prove the main theorem of this subsection.

\begin{theorem}\label{endomorphisms in non-generic 1}
The map $j : E \to M_2(A)^{\SL_2}$ is an isomorphism.
\end{theorem}

\begin{proof}
We know that $E[1/\vp ] \to M_2(A)^{\SL_2}[1/\vp ]$ is an isomorphism and $E$ is $\vp$-torsionfree, since it is free over $R$. In particular $j$ is injective, so it remains to prove surjectivity. Note that $A$ is $\vp$-torsionfree as well, so by surjectivity of $j$ after inverting $\vp$, it suffices to show that if an element $X = \lambda_1 + \lambda_2 u + \lambda_3 v + \lambda_4 (uv-vu) \in E[1/\vp]$ as above has image $j(X) \in M_2(A)$, then we must have $\lambda_i \in R$ for $i=1,\dots,4$. If $j(X) \in M_2(A)$, then its image in $M_2(C[1/\vp])$ lies in $M_2(C)$, i.e.
\[
\begin{pmatrix} \lambda_1 + \lambda_2a_1 + \lambda_3 b_1  + \lambda_4 d_2  & \lambda_2 + \lambda_4 (b_2 - b_1) \\  \lambda_3 d_2 + \lambda_4 ( a_2 d_2 - a_1 d_2) & \lambda_1 + \lambda_2 a_2 + \lambda_3 b_2 - \lambda_4 d_2 \end{pmatrix} \in M_2(C).
\]
Looking at the top right corner, we see that 
\[
\lambda_2 + \lambda_4 (b_2 - b_1) = (\lambda_2 + 2t_2 \lambda_4) - 2 \lambda_4 b_1 \in C.
\]
By Lemma \ref{freeness of C}, we deduce first that $\lambda_4 \in R$ and then that $\lambda_2 \in R$. Applying this to the top left corner, we see that $\lambda_1 + \lambda_3 b_1 \in C$ and hence by Lemma \ref{freeness of C} again, we see that $\lambda_1,\lambda_3 \in R$. This finishes the proof.
\end{proof}

We finish this section by describing a free resolution of the left $E$-module $\OO_{\mathbf{1}}$ given by the quotient (of $\OO$-algebras) $E \xrightarrow{f} \OO$ with $f(g - 1) = 0$ for all $g \in \mathcal{G}$ and $f(t_i) = 0$ for $1 \le i \le 3$.  

We have already recalled the $R$-basis of $E$ given by $1, u, v, uv-vu$. We set $w := uv-vu$. The squares $u^2, v^2$ lie in $R$, the center of $E$.

\begin{proposition}
	The following gives a free resolution of the left $E$-module $\OO_{\mathbf{1}}$:
	
	\begin{equation}\label{eq:ng1-resolution}
	0 \to E \xrightarrow{\left(\begin{smallmatrix}
			v & u
	\end{smallmatrix}\right)} E^{\oplus 2} \xrightarrow{\left(\begin{smallmatrix}
		vu & -u^2\\ -v^2 & uv
	\end{smallmatrix}\right)} E^{\oplus 2} \xrightarrow{\left(\begin{smallmatrix}
	u \\ v
\end{smallmatrix}\right)}  E \xrightarrow{f} \OO_{\mathbf{1}} \to 0
	\end{equation}
with the matrices acting from the right on row vectors.
\end{proposition}
\begin{proof}
	First we need to check that the left ideal generated by $u, v$ coincides with kernel of $f$. Since this left ideal contains $Ru, Rv$ and $Rw$, it suffices to show that it also contains the prime ideal $ (t_1,t_2,t_3)$. In fact, we have $(u^2,v^2,uv+vu) = (t_1,t_2,t_3)$, which be useful later. This follows from the identities \begin{align*}u^2 &= 2t_1 - t_1^2\\ v^2 &= 2t_2 - t_2^2\\ uv+vu &= 2(t_3-t_1-t_2-t_1t_2).\end{align*} 
	
	The first two of these identities are \cite[Equation (159)]{paskunas-image}.  The third can be checked by rewriting $uv + vu - 2t_3$ using the identities $u = \frac{\gamma-\gamma^{-1}}{2}, v = \frac{\delta-\delta^{-1}}{2}$ and $2t_3 + 2 = \frac{T_{\rho}(\gamma\delta) + T_{\rho}(\delta\gamma) }{2}$.
	
	Next we need to show that the kernel of ${\left(\begin{smallmatrix}
			u \\ v
		\end{smallmatrix}\right)}$ is contained in the image of ${\left(\begin{smallmatrix}
		vu & -u^2\\ -v^2 & uv
	\end{smallmatrix}\right)}$. Suppose $(\lambda_1 + \lambda_2 u + \lambda_3 v + \lambda_4 w, \mu_1 + \mu_2 u + \mu_3 v + \mu_4 w)  \in E^2$ is in the kernel. Applying the map ${\left(\begin{smallmatrix}
	u \\ v
\end{smallmatrix}\right)}$ and comparing coefficients tells us that this boils down to the following equalities in $R$:
\begin{align}
\lambda_1 &= -2v^2\mu_4 - \lambda_4(uv+vu)\\
\mu_1 &= 2u^2\lambda_4 + \mu_4(uv+vu)\\
\lambda_3 &= \mu _2 \\
0 &= \lambda_2u^2+\mu_3v^2 + \lambda_3(uv+vu)\label{eqn:koszul}.
\end{align}

Translating by $(-2\lambda_4, 2\mu_4 ){\left(\begin{smallmatrix}
		vu & -u^2\\ -v^2 & uv
	\end{smallmatrix}\right)} = (\lambda_1+\lambda_4 w, \mu_1+\mu_4 w)$,  we may assume that $\lambda_1 = \mu_1 = \lambda_4 = \mu_4 = 0$. Now we consider equation (\ref{eqn:koszul}). Since $u^2, v^2, (uv+vu)$ form a regular sequence in $R$, we can use the Koszul complex to write
\[(\lambda_2,\mu_3,\lambda_3) = (x,y,z)\begin{pmatrix}
	0 & -(uv+vu) & v^2\\ -(uv+vu) & 0 & u^2\\ -v^2& u^2 & 0
\end{pmatrix} \] for some $x,y,z \in R$. Then, noticing that $vuv = (uv+vu)v- v^2 u$ and $uvu = (uv+vu)u - u^2v$, the reader can check that we have \[(-yu, zu-xv)\begin{pmatrix}
vu & -u^2\\ -v^2 & uv
\end{pmatrix} = (\lambda_2 u + \lambda_3 v, \lambda_3u + \mu_3 v).\] 

To check exactness at the next step of the sequence we consider the condition that  $(\lambda_1 + \lambda_2 u + \lambda_3 v + \lambda_4 w, \mu_1 + \mu_2 u + \mu_3 v + \mu_4 w)  \in E^2$ is in the kernel of ${\left(\begin{smallmatrix}
		vu \\ -v^2
	\end{smallmatrix}\right)}$. Again, comparing coefficients gives some equalities in $R$. One of them is \[-\lambda_2 u^2 = \mu_3 v^2,\] which tells us that there is an $x \in R$ with $\mu_3 = xu^2$ and $\lambda_2 = -xv^2$. Translating by \[(-xuv){\begin{pmatrix}
	vu \\ -v^2
\end{pmatrix}} = (\lambda_2u, \mu_3 v - x(uv+vu)u),\] we may assume that $\lambda_2 = \mu_3 = 0$. Now the condition that $(\lambda_1 + \lambda_3 v + \lambda_4 w, \mu_1 + \mu_2 u + \mu_4 w)$ is in the kernel of ${\left(\begin{smallmatrix}
vu \\ -v^2
\end{smallmatrix}\right)}$ boils down to the equalities
\begin{align*}
	\lambda_1 &= -2v^2\mu_4 + \lambda_4(uv+vu)\\
	\mu_1 &= 2u^2\lambda_4 - \mu_4(uv+vu)\\
	\lambda_3 &= \mu _2,
\end{align*}
which means we have
\[\lambda_1 + \lambda_3 v + \lambda_4 w = (2\lambda_4u-2\mu_4v+\lambda_3)v \] and \[\mu_1+\mu_2u +\mu_4w = (2\lambda_4u-2\mu_4v+\lambda_3)u.\] This shows that we do have something in the image of $\left(\begin{smallmatrix}
	v & u
\end{smallmatrix}\right)$. 
Finally, the map $E \xrightarrow{\left(\begin{smallmatrix}
		v & u
	\end{smallmatrix}\right)} E^{\oplus 2}$ is injective because $v^2$ is a non-zero divisor.
\end{proof}

\subsection{Non-generic case II}
\label{subsec: non-gen II}

In this case, $\rho \simeq \chi_1 \oplus \chi_2 \cong \chi \otimes (\omega \oplus 1)$ for some character $\chi : \Gamma \to \F^\times$. Unlike all other cases, the moduli of representations $\fX$ is not smooth. Pa{\v s}k{\=u}nas has computed some deformation rings of representations with semi-simplification isomorphic to $\rho$ \cite[\S B]{paskunas-image}, relying on a presentation due to B\"ockle \cite{bockle-demuskin}. We will adapt these results to describe the entire moduli space $\fX$. 

\begin{theorem}
\label{presentation of moduli stack non-generic 2}
There is an isomorphism of graded rings
\[
S \cong S' := \frac{\OO\lb a_0, a_1, b_0 c, b_1 c\rb [b_0, b_1, c]}{(pb_0 + a_1 b_0 + a_0 b_1)}
\]
where $b_i$ have degree 2 for $i=0,1$, $c$ has degree $-2$, and $a_i$ has degree $0$ for $i=0,1$. The isomorphism $S \cong S'$ induces an isomorphism of subrings of degree 0, $R = S_0 \cong R' = S'_0$ of degree 0, 
\[
R \cong R' := \frac{\OO\lb a_0, a_1, b_0 c, b_1 c\rb}{(pb_0c + a_1 b_0 c + a_0 b_1 c)} \cong \frac{\OO\lb a_0, a_1, Y_0, Y_1\rb}{(pY_0 + a_1 Y_0 + a_0Y_1)}.
\]
The universal Cayley--Hamilton algebra $E$ has $R$-GMA form
\[
\begin{pmatrix}
R & \frac{Rb_0 \oplus Rb_1}{\langle (p+a_1)b_0 +  a_0 b_1\rangle}\\
Rc & R
\end{pmatrix}
\]
with cross-diagonal multiplication given by
\[
((x_0 b_0, x_1 b_1), yc) \mapsto x_0 yb_0c + x_1yb_1c.
\]
\end{theorem}

The claim that $S'$ is a model for $S$ is the main new statement and is developed in Theorem \ref{thm: univ to explicit isoms} below. For the moment, we deduce Theorem \ref{presentation of moduli stack non-generic 2} from Theorem \ref{thm: univ to explicit isoms} using facts about the residually multiplicity free case summarized in \S\ref{subsec: GMAs}.

\begin{proof}[{Proof of Theorem \ref{presentation of moduli stack non-generic 2} given Theorem \ref{thm: univ to explicit isoms}}] 
We know that $R = S_0$ and $E = \End(\mathcal{V}) = M_2(S)^{\bG_m}$ from Theorem \ref{presentation of moduli stack}. What remains is to deduce the claimed presentations of $R$ by $R'$ and of $E$ as above. This follows directly from Proposition \ref{prop: GMA description}, which will imply that
\[
E_{1,2} \cong S_2, \quad E_{2,1} \cong S_{-2}
\]
and that the cross-diagonal multiplication map is compatible with the multiplication map $S_2 \times S_{-2} \to S_0 = R$. Then the form of $E$ given in Theorem \ref{presentation of moduli stack non-generic 2} follows from straightforward calculations of $S_{\pm 2}$ given the isomorphism $S \cong S'$ proved in Theorem \ref{thm: univ to explicit isoms}. 
\end{proof}

The proof that $S \cong S'$ is what remains. Without loss of generality, we will write this proof in the case that $\chi$ and $\psi$ are trivial and $\oo = \Z_p$; the general case follows by twisting. We begin this with Pa{\v s}k{\=u}nas's description in \cite[\S B]{paskunas-image} of a certain quotient group of $\Gamma$. It requires the following data and notation. 
\begin{itemize}
\item Let $\cF$ denote a free pro-$p$ group on $p+1$ generators $x_0, \dotsc, x_p$.
\item Given a profinite group $H$, let $H(p)$ denote its maximal pro-$p$ quotient. 
\item Given a pro-$p$ group $H$, there is a $p$-lower central series filtration defined inductively as 
\[
H_1 = H, \quad H_{i+1} = H^p_i[H_i,H] \text{ for } i \in \Z_{\geq 1}
\]
\item Because $\Gamma_{\Q_p(\zeta_p)}(p)$ is a Demu\v{s}kin group with invariants $n = p+1$ and $q = p$ (for a reference, see e.g.\ \cite[\S3.9]{NSW2008}), there exists a surjection $\varphi : \cF \twoheadrightarrow \Gamma_{\Q_p(\zeta_p)}(p)$ with kernel generated by a single element $r$. 
\end{itemize}

We quote this lemma from \cite[App.\ B]{paskunas-image}. 
\begin{lemma}[{B\"ockle, Pa{\v s}k{\=u}nas \cite[Lem.\ B.1]{paskunas-image}}]
\label{lem: Gal action on cF}
There exists an action of $\Gal(\Q_p(\zeta_p)/\Q_p)$ on $\cF$ and a choice of $\varphi$ such that $\varphi$ is equivariant for the natural actions of $\Gal(\Q_p(\zeta_p)/\Q_p)$, 
\begin{enumerate}
\item $g x_i g^{-1} = x_i^{\tilde \omega(g)^i}$ for all $g \in \Gal(\Q_p(\zeta_p)/\Q_p)$ and $0 \leq i \leq p$, and
\item the image of $r$ in $\mathrm{gr}_2 \cF$ is equal to the image of
\[
r' = x_1^p[x_1, x_{p-1}][x_2, x_{p-2}] \dotsm [x_{\frac{p-1}{2}}, x_{\frac{p+1}{2}}][x_p,x_0].
\]
\end{enumerate}
\end{lemma}

Next we will produce a representation of $\cF$ with coefficients in the ring $S'$ of Theorem \ref{presentation of moduli stack non-generic 2}. Afterward we will show that it factors through $\varphi$ and is universal, producing the isomorphism $S \isoto S'$. This is a straightforward adaptation of the construction in \cite[pg.\ 180]{paskunas-image} from a deformation ring to the whole moduli stack of representations. 

\begin{definition}
\label{def: alpha map}
Denote by $\alpha: \cF \rtimes \Gal(\Q_p(\zeta_p)/\Q_p) \to \GL_2(S')$ the homomorphism determined by 
\begin{gather*}
\Gal(\Q_p(\zeta_p)/\Q_p)  \ni g \mapsto \sm{\tilde\omega(g)}{0}{0}{1} \\
\text{ for }i = 2,3,\dotsc,p-3 , \quad x_i \mapsto 1  \\
x_{p-2} \mapsto \displaystyle \begin{pmatrix} 1 & 0 \\ c & 1 \end{pmatrix} \\
\text{for }j=0,1, \quad 
x_{1+j(p-1)} \mapsto \begin{pmatrix} 1 & b_j \\ 0 & 1 \end{pmatrix} \\
\text{for }j=0,1, \quad 
x_{j(p-1)} \mapsto \begin{pmatrix} (1+a_j)^{-\frac{1}{2}} & 0 \\ 0 & (1+a_j)^{\frac{1}{2}}\end{pmatrix}. 
\end{gather*}
where the semi-direct product structure is as in Lemma \ref{lem: Gal action on cF}. The fact that these images of generators defines a homomorphism can be read off from the semi-direct product structure. 
\end{definition}

Let $\Gamma'$ be the Galois group over $\Q_p$ of the maximal pro-$p$ extension of $\Q_p(\zeta_p)$. Let $\Gamma'_{\Q_p(\zeta_p)} \subset \Gamma'$ denote the subgroup fixing $\Q_p(\zeta_p)$. Thus we naturally have a quotient map $\pi : \Gamma \twoheadrightarrow \Gamma'$, and the universal adapted representation $\rho_S : \Gamma \to \GL_2(S)$ factors through $\Gamma'$. 
\begin{proposition}[{Following \cite[Prop.\ B.2]{paskunas-image}}]
\label{prop: NG2 factor}
There exists a continuous group homomorphism
\[
\varphi' : \cF \rtimes \Gal(\Q_p(\zeta_p)/\Q_p) \twoheadrightarrow \Gamma'
\]
such that $\varphi' \equiv \varphi \pmod{(\Gamma'_{\Q_p(\zeta_p)})_3}$ and there exists a factor $\tilde \rho$ of $\alpha$ producing a commuting diagram
\[
\xymatrix{
\cF \rtimes \Gal(\Q_p(\zeta_p)/\Q_p) \ar[rd]_{\varphi'} \ar[r]^(.65)\alpha & \GL_2(S') \\
 & \Gamma' \ar[u]_{\tilde \rho}
}
\]
In addition, there exists $r_1 \in \cF$ such that $\Gal(\Q_p(\zeta_p)/\Q_p)$ acts on $r_1$ by $\tilde\omega$ and $\ker \varphi'$ equals the closed normal subgroup of $\cF$ generated by $r_1$. 
\end{proposition}

\begin{proof}
First we observe that $r' \in \ker \varphi'$, where $r'$ was defined in Lemma \ref{lem: Gal action on cF}. Indeed, for $j=0,1$, 
\[
[\varphi'(x_{1+(p-1)j}), \varphi'(x_{(p-1)(1-j)})] = \begin{pmatrix} 1 & a_{1-j}b_j \\ 0 & 1 
\end{pmatrix}
\]
while $[\varphi'(x_i), \varphi'(x_{p-i})] = 1$ for $i \not\equiv 1,0 \pmod{p-1}$, and therefore
\[
\varphi'(r) = \varphi'(x_1)^p \cdot \prod_{i=1}^{\frac{p-1}{2}} [\varphi'(x_i),\varphi'(x_{p-i})] \cdot [\varphi'(x_p),\varphi'(x_0)] = 
\begin{pmatrix}
1 & pb_0 \\ 
0 & 1 
\end{pmatrix}
\begin{pmatrix}
1 & a_1 b_0 \\ 
0 & 1 
\end{pmatrix}
\begin{pmatrix}
1 & a_0 b_1 \\ 
0 & 1 
\end{pmatrix}
\]
which vanishes in $\GL_2(S')$ due to the presence of the relation $pb_0 + a_1 b_0 + a_0 b_1$. 

By Lemma \ref{lem: Gal action on cF}(2), $r \equiv r' \pmod{\cF_3}$, and therefore $\alpha(r) \in \alpha(\cF_3)$. The rest of the proof follows exactly as in \cite[Proof of Prop.\ B.2]{paskunas-image}, producing $r_1 \in \cF$, equivalent to $r$ and $r'$ $\pmod{\cF_3}$, such that $r_1 \in \ker \alpha$ and the conjugation action of $\Gal(\Q_p(\zeta_p)/\Q_p)$ on $\cF$ acts on $r_1$ by the character $\tilde\omega$. 
\end{proof}

Because $\tilde \rho \circ \pi : \Gamma \to \GL_2(S')$ has residual pseudorepresentation $\psi(\omega \oplus 1)$, the universal property of the universal Cayley--Hamilton algebra $(E,R,D_E : E \to R)$ (see Definition \ref{def:univCH}) produces 
\begin{itemize}
\item a ring homomorphism $R \to S'_0 = R' \subset S'$ 
\item an $R$-algebra homomorphism $\eta : E \to M_2(S')$ such that 
\begin{itemize}
\item $(E,R,D_E) \to (M_2(S'),S',\det : M_2(S') \to S')$ is a morphism of Cayley--Hamilton algebras
\item $\tilde \rho = \eta \circ \rho^u$. 
\end{itemize}
\end{itemize}
We impose the $R$-GMA structure on $E$ arising from the idempotents arising from pullback over $\eta$,
\[
(\eta^{-1}(\sm{1}{0}{0}{0}), \ \eta^{-1}(\sm{0}{0}{0}{1}).
\]
(These idempotents lie in $\eta(E)$ because they are $\Z_p$-linear combinations of the image of $\Gal(\Q_p(\zeta_p)/\Q_p)$ specified in Definition \ref{def: alpha map}.) Now that $E$ has been endowed with an $R$-GMA structure, we write $S$ for the graded $R$-algebra representing its adapted representation moduli functor. Thus its universal property along with $\eta$ induce a graded $R$-algebra homomorphism
\[
\phi : S \to S', \text{ with } 0\text{-degree part } R=S_0 \to R'=S'_0
\]
where $R$ is the pseudodeformation ring. 

\begin{theorem}
\label{thm: univ to explicit isoms}
The homomorphisms $\phi : S \to S'$ and $\phi_0 : R \to R'$ are isomorphisms. 
\end{theorem}

To prove the theorem, we import a description of $\hat S/p \hat S$ from \cite{wang-erickson-A-inf}. We remind the reader that the completion $\hat S$ of $S$ is not local: see Remark \ref{rem: hatS}. Because $\rho \simeq \omega \oplus 1$, we can apply the decomposition $\Ad^0 \rho \cong \omega \oplus 1 \oplus \omega^{-1}$. 
 
\begin{proposition}[{\cite{wang-erickson-A-inf}}]
\label{prop: abstract presentation}
There exists a presentation of $\hat S/p\hat S$ of the form 
\[
\left[\frac{(\Sym^*_{\F_p} H^1(\Gamma, \Ad^0 \rho)^*)}{(m^*(H^2(\Gamma, \Ad^0 \rho)^*)}\right]^\wedge \isoto \hat S/p\hat S
\]
where 
\begin{enumerate}
\item the completion denoted $[ \dotsm]^\wedge$ is at the ideal generated by $H^1(\Gamma,1)^*$, $H^1(\Gamma, \omega)^* \otimes_{\F_p} H^1(\Gamma, \omega^{-1})^*$. 
\item the presentation is $\bG_m$-equivariant, as expressed by a $\Z$-grading (of each constituent of the limit that is the completion) where the degrees of the modules of generators and relations are given by 
\begin{itemize}
\item $\deg H^i(\Gamma, \omega)^* = 2$ for $i = 1,2$
\item $\deg H^1(\Gamma, 1)^* = 0$
\item $\deg H^1(\Gamma, \omega^{-1})^* = -2$
\end{itemize}
\item $m^*$ is $\bG_m$-equivariant. 
\item the image of $m^*$ lies in the ideal $\Sym^{\geq 2}$, and the quadratic term (modulo $\Sym^{\geq 3}$) $m_2^*$ is the $\F_p$-linear dual of the composite of the cup product and Lie bracket map
\[
H^1(\Gamma, \Ad^0 \rho) \otimes_{\F_p} H^1(\Gamma, \Ad^0 \rho) \to H^2(\Gamma, \Ad^0 \rho \otimes_{\F_p} \Ad^0 \rho) \buildrel{[\cdot,\cdot]}\over\longrightarrow H^2(\Gamma, \Ad^0 \rho). 
\]
\item the universal representation $\rho_S : \Gamma \to M_2(S)$ has the following form modulo $(p, \Sym^{\geq 2} H^1(-)^*)$,
\[
\begin{pmatrix}
\omega(1+ \tilde A)& \tilde B\\
\omega \tilde C & 1 - \tilde A
\end{pmatrix}
\]
where 
\begin{gather*}
\tilde B \in Z^1(\Gamma, \omega) \otimes H^1(\Gamma, \omega)^*\\
\tilde A \in Z^1(\Gamma, \F_p) \otimes H^1(\Gamma, \F_p)^* \\
\tilde C \in Z^1(\Gamma, \omega^{-1}) \otimes H^1(\Gamma, \omega^{-1})^*
\end{gather*}
are choices of lift of $\mathrm{id} \in \End_{\F_p}(H^1(-)) \cong H^1(-) \otimes_{\F_p} H^1(-)^*$ under the natural projection
\[
Z^1(-) \otimes H^1(-)^* \twoheadrightarrow H^1(-) \otimes H^1(-)^*.
\]
\end{enumerate}
\end{proposition}

\begin{proof}
This is mostly a description of the objects set up to state the main theorem \cite[Thm.\ 11.3.3]{wang-erickson-A-inf}. The following references in this proof refer to \cite{wang-erickson-A-inf}. Part (1) appears in Definition 11.3.1. The $\bG_m$-equivariance of parts (2) and (3) corresponds to the $\F_p \times \F_p$-algebra structure of the presentation: letting $e_1 = \sm{1}{0}{0}{0}$, $e_2 = \sm{0}{0}{0}{1}$, the $\bG_m$-action (i.e.\ the adjoint action of conjugation of the torus of $\SL_2$) is by the character $2 \in X^*(\G_m)$ on $e_1 S e_2$, the character $0$ on $e_i S e_i$ (for $i = 1,2$), and $-2$ on $e_2 S e_1$. Part (4) follows from a description of the quadratic term $m_2^*$ of $m^*$ appearing in Corollary 5.2.6, where here we use the Lie algebra version produced by the associative version there. (Indeed, $\Ad^0 \rho \subset \End \rho$ and the Lie bracket is the commutator.)  Part (5) appears in the construction of $\rho_S$ appearing in Corollary 7.4.5. In particular, the section of $Z^1(-) \to H^1(-)$ appearing in part (5) is denoted by $f_1$ in Corollary 7.4.5. 
\end{proof}

In order to work explicitly with this presentation of $S/pS$, we will use the following choice of basis of $H^1(\Gamma, \Ad^0 \rho)$. To specify this basis, we use generators $x_i \in \Gamma'$, $0 \leq i \leq p$, produced by Definition \ref{def: alpha map} and Proposition \ref{prop: NG2 factor}. (This is a slight abuse of notation, since this generators are actually in $\cF$ and we use $x_i$ to refer to $\varphi'(x_i) \in \Gamma'$.) The basis is labeled so that it matches the deformations to $\F_p[\varepsilon]/(\varepsilon^2)$ of $\rho$ that arise from using each non-identity matrix listed in Definition \ref{def: alpha map}, as follows. 

\begin{lemma}
\label{lem: cohom bases}
There exists a set of choices of bases of the $\F_p$-vector spaces 
\begin{itemize}
\item $\{\bar b_0^*, \bar b_1^*\} \subset H^1(\Gamma, \omega)$
\item $\{\bar a_0^*, \bar a_1^*\} \subset H^1(\Gamma, \F_p)$
\item $\{\bar c^*\} \subset H^1(\Gamma, \omega^{-1})$
\end{itemize}
characterized by the property that for each $y \in Y = \{b_0, b_1, a_0, a_1, c\} \subset S'$, the lift $\rho_y : \Gamma \to \GL_2(\F_p[\epsilon]/(\epsilon^2))$ of $\rho$ given by specializing the coefficients of $\tilde \rho \circ \pi: \Gamma \to \GL_2(S')$ along the map $\nu_y : S \to \F_p[\epsilon]/(\epsilon^2)$ given by
\[
y \mapsto \epsilon, \quad z \mapsto 0 \ \text{ for all } z \in Y \smallsetminus \{y\}
\]
realizes the cohomology class $\bar y^*$ under the standard bijection between lifts of $\rho$ to $\F_p[\epsilon]/(\epsilon^2)$ and $Z^1(\Gamma, \Ad^0 \rho)$. 
\end{lemma}

\begin{proof}
As is well known, lifts of $\rho$ to $\F_p[\epsilon]/(\epsilon^2)$ with fixed determinant biject with $Z^1(\Gamma,\Ad^0 \rho)$, and they have non-trivial projection to $H^1(\Gamma, \Ad^0 \rho)$ if and only if they are not conjugate by $1 + \epsilon \cdot M_2(\F_p)$ to the trivial lift. By Proposition \ref{prop: NG2 factor}, and in particular by applying $\varphi'$, the specified lifts $\rho_y$ of $\rho$ produce the three subsets $\{\bar b_0^*, \bar b_1^*\}$, $\{\bar a_0^*, \bar a_1^*\}$, and $\{\bar c^*\}$ of $Z^1(\Gamma, \Ad^0\rho)$. Viewing Definition \ref{def: alpha map}, we observe that they are 
\begin{itemize}
\item concentrated in the summand of $Z^1(\Ad^0 \rho)$ named in the lemma (e.g.\ $\rho_{b_0} \in Z^1(\Gamma, \omega)$) under the standard decomposition $\Ad^0 \rho \simeq \omega \oplus 1 \oplus \omega^{-1}$
\item linearly independent after projection to $H^1(\cF \rtimes \Gal(\Q_p(\zeta_p)/\Q_p), -)$, and therefore also linearly independent subsets of the cohomology groups $H^1(\Gamma,-)$ named in the lemma. 
\end{itemize}
Finally, by standard Tate local duality and Euler characteristic formulas using the assumption $p \geq 5$, the dimension of these $H^1(\Gamma,-)$ equals the cardinality of each linearly independent subset named in the lemma. 
\end{proof}

In the following, ``$\mathrm{Kum}$'' refers to a Kummer class (under the standard bijection of Kummer theory between first cohomology valued in a cyclotomic character and the unit group of $\Q_p$), and $\Q_{p^p}/\Q_p$ denotes the unique unramified degree $p$ extension of $\Q_p$. 
\begin{remark}
\label{rem: identify generators}
It is possible, but not necessary for the proof, to directly prove the following equalities up to $\F_p^\times$-scalar.
\begin{itemize}
\item $\bar b_0^* = \mathrm{Kum}(1+p) \in H^1(\Q_p, \omega)$
\item $\bar b_1^* = \mathrm{Kum}(p) \in H^1(\Q_p, \omega)$
\item $\bar a_0^* \in \Hom(\Gal(\Q_{p^p}(\zeta_p)/\Q(\zeta_p)), \F_p) \subset H^1(\Q_p(\zeta_p), \F_p)^{\omega^0} \cong H^1(\Gamma, \F_p)$
\item $\bar a_1^* \in \Hom(\Gal(\Q_p(\zeta_{p^2})/\Q_p(\zeta_p)), \F_p)  \subset H^1(\Q_p(\zeta_p), \F_p)^{\omega^0} \cong H^1(\Gamma, \F_p)$
\end{itemize}
In particular, the perfect Tate duality pairing is realized by the standard cup product $H^1(\Gamma, \F_p) \times H^1(\Gamma,\omega) \to H^2(\Gamma, \omega) \cong \F_p$ and satisfies $\langle a^*_i, b^*_{1-j}\rangle = \delta_{ij}$ for $i,j \in \{0,1\}$, which explains the form ``$a_0b_1 + a_1b_0$'' of the relation (mod $p$) presenting $S'$ in Theorem \ref{presentation of moduli stack non-generic 2}: it arises by evaluating the $m^*$ of Proposition \ref{prop: abstract presentation}. 
\end{remark}

We will only need the following weaker implication of Remark \ref{rem: identify generators}. Let 
\[
\{\bar a_0, \bar a_1\} \subset H^1(\Gamma, \F_p)^*, \quad \{\bar b_0, \bar b_1\} \subset H^1(\Gamma, \omega)^*,  \quad \{\bar c\} \subset H^1(\Gamma, \omega^{-1})^*
\]
denote dual bases to the bases listed in Lemma \ref{lem: cohom bases}. 

\begin{corollary}
\label{cor: S-pres}
There is an isomorphism of limits of graded rings (where the graded degree of $H^1(\Gamma, \omega)^*$  is 2, the graded degree of $H^1(\Gamma,\F_p)$ is $0$, and the graded degree of $H^1(\Gamma, \omega^{-1})$ is $-2$) 
\[
\frac{\F_p[\bar b_0, \bar b_1, c] \lb \bar a_0, \bar a_1, \bar b_0 \bar c, \bar b_1 \bar c\rb}{\left(F+ \displaystyle{\sum_{0 \leq i,j \leq 1}} \alpha_{i,j}\bar a_i \bar b_j\right)} \ \isoto \ \hat S/p \hat S,
\]
where $F \in \left(\F_p[\bar b_0, \bar b_1, c] \lb \bar a_0, \bar a_1, \bar b_0 \bar c, \bar b_1 \bar c\rb\right)_2 \  \cap \  (\bar a_0, \bar a_1, \bar b_0, \bar b_1 ,\bar c)^3$  and $(\alpha_{i,j}) \in \GL_2(\F_p)$. 
\end{corollary}
In words, the condition on $F$ means that it is of graded degree 2 and is a power series in monomials of degree at least 3 in the variables $a_i, b_i, c$. 

\begin{proof}
This is a particular application of our knowledge of the dimensions of the Galois cohomology groups arising in Proposition \ref{prop: abstract presentation}, along with the appearance of the Lie bracket and cup product in Proposition \ref{prop: abstract presentation}(4). As mentioned in Remark \ref{rem: identify generators}, the only non-trivial summand of this cup product is non-degenerate as a bilinear form; the dual of its factorization through the tensor product is $m_2^*:  H^2(\Gamma, \omega)^* \to H^1(\Gamma, \F_p)^* \otimes_{\F_p} H^1(\Gamma, \omega)^*$. This non-degeneracy is reflected, equivalently, in the conclusion that $\det (\alpha_{i,j}) \neq 0$. 
\end{proof}

Now we can prove Theorem \ref{thm: univ to explicit isoms}. 
\begin{proof}[{Proof of Theorem \ref{thm: univ to explicit isoms}}]
We begin with some reduction steps. Because $S'$ is $p$-torsion free, it will suffice to prove that $\phi/p : S/pS \to S'/pS'$ is an isomorphism. We fix some presentation of $S/pS$ as in Corollary \ref{cor: S-pres}. Like the end of the proof of Theorem \ref{computation of rings generic ps}, we appeal to \cite[Thm.\ 1.6]{alper-hall-rydh-etale-local} to say that it will suffice prove that the local homomorphism $\hat \phi_\rho:  (S/pS)^\wedge_\rho \to (S'/pS')^\wedge_\rho$, defined to be the completion of $\phi/p$ at the maximal ideals $\m_\rho \subset S/pS$ and $\m'_\rho \subset S'/pS'$ corresponding to $\rho$, is an isomorphism of limits of graded rings. 

By Lemma \ref{lem: cohom bases} and Proposition \ref{prop: abstract presentation}(5), and the fact that $\hat \phi_\rho$ arises from applying the moduli interpretation of $S$ to $\tilde \rho \circ \pi$, we see that $\hat \phi_\rho$ gives rise to an isomorphism of cotangent spaces $\m_\rho/\m_\rho^2 \isoto \m'_\rho/\m'^2_\rho$ preserving the bases of these cotangent spaces that arise from the generators we have designated. Namely, $\hat\phi_\rho$ sends 
\[
\hat \phi_\rho :\  \bar a_0 \mapsto a_0, \ \bar a_1 \mapsto a_1, \ \bar b_0 \mapsto b_0, \ \bar b_1 \mapsto b_1, \ \bar c \mapsto c \quad \pmod{\m'^2_\rho}. 
\]
In particular, we now know that $\hat \phi_\rho$ is surjective. 

We can say, equivalent to the matching of bases above, that $\hat \phi_\rho(\bar a_0) - a_0 \in \m'^2_S$, and similarly for each of the other four matching pairs of basis elements. Therefore, by reading off the presentation of $S'$ in Theorem \ref{presentation of moduli stack non-generic 2}, the kernel of the composite map
\[
\F_p\lb \bar a_0, \bar a_1, \bar b_0, \bar b_1, \bar c\rb \buildrel\eta\over\twoheadrightarrow (S/pS)^\wedge_\rho \buildrel{\hat\phi_\rho}\over\twoheadrightarrow (S'/pS')^\wedge_\rho
\]
is a principal ideal with generator $\bar a_0 \bar b_1 + \bar a_1 \bar b_0$. Consequently, this generator divides any generator of the kernel of the surjection labeled $\eta$, which by Corollary \ref{cor: S-pres} is 
\[
(\bar a_0 \bar b_1 + \bar a_1 \bar b_0) \,\Big\vert \left(F+ \sum_{0 \leq i,j \leq 1} \alpha_{i,j}\bar a_i \bar b_j \right)
\]
in $\F_p\lb \bar a_0, \bar a_1, \bar b_0, \bar b_1, \bar c\rb$, a divisibility of power series that are in $\m^2$ and non-zero modulo $\m^3$. Therefore the quotient is a unit and $\hat\phi_\rho$ is an isomorphism. 
\end{proof}

\subsection{Coherent sheaves on $\Rep(E)$ in case non-generic II}\label{subsec: non-gen II coh sheaves}

We wish to describe some coherent sheaves on $\Rep(E)$ and compute their $\Ext$-groups. Computationally, the situation is most similar to \S \ref{subsec: stacks generic ps GL2}, but we will need more than line bundles, so the computations will become far more involved. Nevertheless, we start as in \S \ref{subsec: stacks generic ps GL2}. We simplify the notation by letting $\fX$ denote the stack $\Rep(E)$. By Theorem \ref{presentation of moduli stack non-generic 2}, we may present $\fX$ as 
\[
\fX \cong [ \Rep^{\mathrm{Ad},\Box}(E)/T] \cong [\Spec S / T],
\]  
and coherent sheaves on $\fX$ are equivalent to finitely generated graded $S$-modules. We write $\Hom_\fX$, $\Ext_\fX$ for $\Hom$ and $\Ext$ groups in $\Coh(\fX)$.

As in \S \ref{subsec: stacks generic ps GL2} we define $L_m = S(m)$ for $m\in \Z$. This is a line bundle on $\fX$ and again the vector bundle $\mc{V}$ on $\fX$ underlying the universal representation corresponds to the graded $S$-module $L_1 \oplus L_{-1}$. As in Theorem \ref{Ext groups of universal rep generic ps}, we obtain the following theorem.

\begin{theorem}\label{line bundles nongen 2}
All vector bundles on $\fX$ are projective objects in the category of quasicoherent sheaves. Moreover, $\End_{\fX}(\mc{V}) = E$ as rings.
\end{theorem}

In addition to $L_1$ and $L_{-1}$, we will need a third coherent sheaf $Q$ on $\fX$, which we now describe, and which is not a vector bundle. To shorten the notation somewhat, we set $\ap = a_1 +p$; the presentation of $S$ in Theorem \ref{presentation of moduli stack non-generic 2} then becomes
\[
S \cong \frac{\oo \lb a_0, \ap, b_0c, b_1c \rb [b_0,b_1,c] }{(\ap b_0 + a_0 b_1)}.
\]
Suppose now that $\wt{T}$ is any integral domain, that $f \in \wt{T}$ is nonzero, and that $\wt{M}$ and $\wt{N}$ are $n\times n$-matrices with entries in $\wt{T}$ satisfying $\wt{M}\wt{N}=\wt{N}\wt{M}=fI$ (where $I$ is the identity matrix). Set $T=\wt{T}/(f)$ and let $M$ and $N$ be the reductions of $\wt{M}$ and $\wt{N}$ modulo $f$, respectively. Then $MN=NM=0$ and one easily checks that 
\[
T^n \overset{M}{\to} T^n \overset{N}{\to} T^n \,\,\,\,\,\, \text{and} \,\,\,\,\,\, T^n \overset{N}{\to} T^n \overset{M}{\to} T^n
\]
are both exact, where we view $T^n$ as column vectors\footnote{Or row vectors; the choice does not matter.}. Having said this, we consider the matrices
\[
\wt{M} = \begin{pmatrix} b_0 & b_1 \\ -a_0 & \ap \end{pmatrix} \,\,\,\,\,\, \text{and} \,\,\,\,\,\, \wt{N} = \begin{pmatrix} \ap & -b_1 \\ a_0 & b_0 \end{pmatrix}
\] 
with entries in $\oo\lb a_0, \ap \rb [b_0,b_1,c]$. We have $\wt{M}\wt{N} = \wt{N} \wt{M} = (\ap b_0 + a_0 b_1)I$, so the discussion above applies for the reductions $M$ and $N$ to $S$. We can even view $M$ and $N$ as homomorphisms of graded $S$-modules in the following way: We have
\[
M : L_n \oplus L_n \to L_{n+2} \oplus L_n
\]
and
\[
N : L_n \oplus L_{n-2} \to L_n \oplus L_n,
\]
for any $n\in \Z$. Here, and in the rest of $\S\ref{subsec: non-gen II coh sheaves}$, when writing homomorphisms between graded $S$-modules we view elements of direct sums as column vectors and represent maps as matrices acting from the left. 

The graded $S$-module $Q$ is defined as
\[
Q = \Coker\left( M : L_{-1} \oplus L_{-1} \to L_1 \oplus L_{-1} \right)
\]
and there is a `periodic' projective resolution of $Q$ of period $2$ given by
\begin{equation}\label{resolution of Q}
\dots \overset{N}{\to} L_{-3} \oplus L_{-3} \overset{M}{\to} L_{-1} \oplus L_{-3} \overset{N}{\to} L_{-1} \oplus L_{-1} \overset{M}{\to}  L_1 \oplus L_{-1} \to Q \to 0.
\end{equation}
The definition of $Q$ was originally motivated by considering the short exact sequence (234) in \cite{paskunas-image}, see Remark \ref{rem:what are we doing} for more details.

In the rest of this section, we will compute various $\Hom$'s and $\Ext$'s involving $Q$. Our first goal is to show that
\[
\Ext^i_{\fX}(L_{-1} \oplus L_1 \oplus Q, L_{-1} \oplus L_1 \oplus Q) = 0
\]
for all $i\geq 1$. We start by observing that $\Ext^i_{\fX}(L_{-1} \oplus L_1, L_{-1} \oplus L_1 \oplus Q) = 0$ since the $L_n$ are projective, so it remains to show that $\Ext^i_{\fX}(Q, L_{-1}) = \Ext^i_{\fX}(Q,L_1) = \Ext^i_{\fX}(Q,Q) = 0$. We then have:

\begin{proposition}\label{ext calcs 1 non-gen 2}
As (ungraded) $S$-modules, $\Ext^i_S(Q,S)=0$ for all $i\geq 1$ (so $Q$ is a maximal Cohen-Macaulay module, since $S$ is Gorenstein). In particular, $\Ext^i_{\fX}(Q,L_n)=0$ for all $i\geq 1$ and all $n\in \Z$.
\end{proposition}

\begin{proof}
The resolution (\ref{resolution of Q}), viewed as ungraded $S$-modules, is simply
\[
\dots \overset{N}{\to} S^2 \overset{M}{\to} S^2 \overset{N}{\to} S^2 \overset{M}{\to} S^2 \to Q \to 0.
\]
Applying $\Hom_S(-,S)$ to the resolution, we get
\[
S^2 \overset{M^t}{\to} S^2 \overset{N^t}{\to} S^2 \overset{M^t}{\to} S^2 \overset{N^t}{\to} \dots,
\]
where we are still regarding $S^2$ as column vectors, and $-^t$ denotes matrix transpose. This is exact in degrees $i\geq 1$, as desired.
\end{proof}

It remains to show that $\Ext^i_{\fX}(Q,Q) = 0$. To do this, we start by considering the graded morphism $L_{-1} \to Q$ which is the composition $L_{-1} \to L_{1}\oplus L_{-1} \to Q$, where the first map sends $x$ to $\binom{0}{x}$ and the second map is the quotient map from the definition of $Q$. The composite is injective and the cokernel $\ol{Q}$ is isomorphic to $L_1/(b_0,b_1)L_{-1}$, i.e.\ $S/(b_0,b_1)$ with grading shifted by $1$. In particular, $(\ol{Q})_k = 0$ for $k\geq 0$. As a consequence, $\Hom_{\fX}(L_n,\ol{Q})=0$ for $n\leq 0$.

\begin{proposition}\label{vanishing of Exts non-gen 2}
We have $\Ext^i_{\fX}(Q,Q) = 0$ for $i\geq 1$.
\end{proposition} 

\begin{proof}
Consider the short exact sequence $0 \to L_{-1} \to Q \to \ol{Q} \to 0$. Taking the long exact sequence for $\Ext_{\fX}(Q,-)$ and using Proposition \ref{ext calcs 1 non-gen 2} we see that $\Ext^i_{\fX}(Q,Q) = \Ext^i_{\fX}(Q,\ol{Q})$ for $i\geq 1$, so it suffices to prove that $\Ext^i_{\fX}(Q,\ol{Q})=0$. But since $\Hom_{\fX}(L_n,\ol{Q})=0$ for $n\leq 0$, applying $\Hom_{\fX}(-,\ol{Q})$ to the resolution of $Q$ from (\ref{resolution of Q}) we simply get 
\[
\Hom_{\fX}(L_1,\ol{Q}) \to 0 \to 0 \to \dots.
\]
In particular, $\Ext^i_{\fX}(Q,\ol{Q})=0$ for $i\geq 1$ as desired.
\end{proof}

Our remaining task in this section is to compute $\End_{\fX}(L_{-1} \oplus L_1 \oplus Q)$ as an $R$-algebra, which we will treat as a $3\times 3$ generalilzed matrix algebra with scalar ring $R$, 
\begin{equation}\label{eq: 3x3 presentation 1}
\begin{pmatrix}
\End_{\fX}(L_{-1}) & \Hom_{\fX}(L_1,L_{-1}) & \Hom_{\fX}(Q,L_{-1}) \\ \Hom_{\fX}(L_{-1},L_1) & \End_{\fX}(L_1) & \Hom_{\fX}(Q,L_1) \\ \Hom_{\fX}(L_{-1},Q) & \Hom_{\fX}(L_1,Q) & \End_{\fX}(Q)
\end{pmatrix}.
\end{equation}
As we will see later, it will turn out to coincide with the endomorphism algebra $\widetilde{E}_{\mathfrak{B}}$ computed in \cite[\S10.5]{paskunas-image}. We already know that $\End_{\fX}(L_{-1} \oplus L_1) = E$, so we will start by computing the remaining entries as $R$-modules. We start by computing $\Hom_{\fX}(Q,L_1)$. Applying $\Hom_{\fX}(-,L_1)$ to the presentation
\[
L_{-1} \oplus L_{-1} \overset{M}{\to} L_1 \oplus L_{-1} \to Q \to 0
\]
we get a left exact sequence
\[
0 \to \Hom_{\fX}(Q,L_1) \to \Hom_{\fX}(L_1 \oplus L_{-1},L_1) \to \Hom_{\fX}(L_{-1} \oplus L_{-1},L_1),
\]
so upon identifying $\Hom_{\fX}(L_1 \oplus L_{-1},L_1) = R \oplus (b_0 R + b_1 R)$ and $\Hom_{\fX}(L_{-1} \oplus L_{-1},L_1)= (b_0 R + b_1 R) \oplus (b_0 R + b_1 R)$ as row vectors, we see that $\Hom_{\fX}(Q,L_1)$ is the kernel of the map $R \oplus (b_0 R + b_1 R) \to (b_0 R + b_1 R) \oplus (b_0 R + b_1 R)$ given by
\[
\begin{pmatrix} x & y \end{pmatrix} \mapsto \begin{pmatrix} x & y \end{pmatrix}\begin{pmatrix} b_0 & b_1 \\ -a_0 & a_1^\prime \end{pmatrix}.
\]
The kernel is isomorphic to $b_0 R + b_1 R$ via $(x,y)\mapsto y$; if $y=b_0 \alpha + b_1 \beta$ (with $\alpha,\beta \in R$) and $(x,y)$ is in the kernel, then one sees easily that $x=a_0 \alpha -\ap \beta$, and one can check that $x$ only depends on $y$ and not the choice of $\alpha$ and $\beta$ (e.g by computing in the fraction field of $S$). In particular, when viewed as a subspace of $R \oplus (b_0 R + b_1 R)$, $\Hom_{\fX}(Q,L_1)$ is generated by $(a_0,b_0)$ and $(-\ap,b_1)$ as an $R$-module.

The computation of $\Hom_{\fX}(Q,L_{-1})$ is entirely analogous; we see that it is the kernel of the map $\Hom_{\fX}(L_1 \oplus L_{-1},L_{-1}) \to \Hom_{\fX}(L_{-1} \oplus L_{-1},L_{-1})$ given by the same formula as above when identifying $\Hom_{\fX}(L_1 \oplus L_{-1},L_{-1})$ with $cR \oplus R$. By computations analogous to those above, we see that $y \mapsto (x,y)$ defines an injection of $b_0 cR + b_1 c R$ into the kernel where, if $y=b_0 c \alpha + b_1 c\beta$, $x= a_0 c \alpha - \ap c \beta$ (and, as above, $x$ only depends on $y$). It remains to show that this is the entire kernel. One of the conditions for $(x,y)$ to be in the kernel is $b_0 x = a_0 y$; we wish to show that this forces $y \in b_0 c R + b_1 c R$. As $x\in cR$, we may write $x=zc$ with $z\in R$ and consider $zb_0 c = a_0 y$ as an identity in $R$. Lifting $y$ and $z$ to $\wt{z}, \wt{y} \in \oo \lb a_0, \ap , b_0 c , b_1 c \rb$, the identity becomes an identity
\[
b_0 c \wt{z} = a_0 \wt{y} + (a_0b_1 c + \ap b_0 c)f
\]
in $\oo \lb a_0, \ap , b_0 c , b_1 c \rb$, which we may rewrite as $a_0(\wt{y} + b_1 c f) = b_0 c (\wt{z} - \ap f)$. Since $b_0 c$ is a prime in $\oo \lb a_0, \ap , b_0 c , b_1 c \rb$, we deduce that $b_0 c$ divides $\wt{y} + b_1 c f$, which implies that $y \in b_0 c R + b_1 c R$ as desired. Summing up, we see that when viewed as a subspace of $cR \oplus R$, $\Hom_{\fX}(Q,L_1)$ is generated by $(a_0 c,b_0 c)$ and $(-\ap c,b_1 c)$ as an $R$-module.

Next, we compute $\Hom_{\fX}(L_{-1},Q)$. Before Proposition \ref{vanishing of Exts non-gen 2}, we constructed an inclusion map $L_{-1} \to Q$. Let us call this map $\iota$; we wish to show that $\Hom_{\fX}(L_{-1},Q)$ is a free $R$-module of rank $1$ generated by $\iota$. Applying $\Hom_{\fX}(L_{-1},-)$ to $L_{-1} \oplus L_{-1} \overset{M}{\to} L_1 \oplus L_{-1} \to Q \to 0$ we get
\[
\Hom_{\fX}(L_{-1},L_{-1}\oplus L_{-1}) \to \Hom_{\fX}(L_{-1},L_1 \oplus L_{-1}) \to \Hom_{\fX}(L_{-1},Q) \to 0,
\]
so $\Hom_{\fX}(L_{-1},Q)$ is the cokernel of the map $R^2 \to (b_0 R + b_1 R) \oplus R$ given by 
\[
\begin{pmatrix} x \\ y \end{pmatrix} \mapsto \begin{pmatrix} b_0 & b_1 \\ -a_0 & a_1^\prime \end{pmatrix} \begin{pmatrix} x \\ y \end{pmatrix}.
\]
Since $\iota$ is the map in the cokernel represented by $\binom{0}{1}$, it is clear that $\iota$ generates the cokernel, and it is then clear that $\Hom_{\fX}(L_{-1},Q)$ is free since $\iota$ is an inclusion and $L_{-1}$ is $R$-torsionfree. This gives the desired result.

Next up is $\Hom_{\fX}(L_1,Q)$. The strategy is similar to the previous case; we have a right exact sequence
\[
\Hom_{\fX}(L_1,L_{-1}\oplus L_{-1}) \to \Hom_{\fX}(L_1,L_1 \oplus L_{-1}) \to \Hom_{\fX}(L_1,Q) \to 0
\]
which identifies $\Hom_{\fX}(L_1,Q)$ with the cokernel of the map $(cR)^2 \to R \oplus cR$ given by the same formula as above. This means that the cokernel is generated by $\binom{1}{0}$ and $\binom{0}{c}$ , with relations $b_0 c \binom{1}{0} = a_0 \binom{0}{c}$ and $b_1 c \binom{1}{0} = -\ap \binom{0}{c}$.

Finally, we come to $\End(Q)$. Applying $\Hom_{\fX}(-,Q)$ to the short exact sequence $0 \to L_{-1} \to Q \to \ol{Q} \to 0$ we get an injection $\End(Q) \to \Hom_{\fX}(L_{-1},Q)$, since $\Hom_{\fX}(\ol{Q},Q) = 0$ (even as ungraded $S$-modules, since $Q$ is torsionfree, being maximal Cohen-Macaulay). But $\Hom_{\fX}(L_{-1},Q)$ is freely generated by $\iota$ (by above) and $\End(Q) \to \Hom_{\fX}(L_{-1},Q)$ maps the identity on $Q$ to $\iota$, so we conclude that $\End(Q)=R$.

We now summarize the results above (including the fact that $E = \End_{\fX}(L_{-1}\oplus L_1)$) in a theorem, where we also give names to the generators, foreshadowing the comparison of our results here with those of \cite[\S 10]{paskunas-image}.

\begin{theorem}\label{module structure nongen 2}
The following holds:
\begin{enumerate}
\item $\End_{\fX}(L_{-1})$, $\End_{\fX}(L_1)$ and $\End_{\fX}(Q)$ are all free $R$-modules of rank $1$ generated by the respective identity functions.

\smallskip

\item We have $\Hom_{\fX}(L_1,L_{-1})=cR$, and we let $\varphi_{12}$ denote the generator $c \in cR$.

\smallskip

\item $\Hom_{\fX}(Q,L_{-1})$ is the subspace of $cR \oplus R$ generated by $\varphi_{13}^0 = (a_0 c, b_0 c)$ and $\varphi_{13}^1 = (-\ap c, b_1 c)$. The map $b_0 c R + b_1 c R \to \Hom_{\fX}(Q,L_{-1})$ given by $b_0 c x + b_1 c y \mapsto (a_0 c x -\ap c y, b_0 c x + b_1 c y)$ is an isomorphism.

\smallskip

\item We have $\Hom_{\fX}(L_{-1},L_1) = b_0 R + b_1 R$, and we let $\varphi_{21}^0 = b_0$ and $\varphi_{21}^1 = b_1$.

\smallskip

\item $\Hom_{\fX}(Q,L_1)$ is the subspace of $R \oplus (b_0 R + b_1 R)$ generated by $\varphi_{23}^0 = (a_0, b_0)$ and $\varphi_{23}^1 = (-\ap, b_1)$. The map $b_0 R + b_1 R \to \Hom_{\fX}(Q,L_1)$ given by $b_0 x + b_1 y \mapsto (a_0 x -\ap y, b_0 x + b_1 y)$ is an isomorphism.

\smallskip

\item $\Hom_{\fX}(L_{-1},Q)$ is a free $R$-module of rank $1$, generated by the inclusion $\varphi_{31} = \iota$.

\smallskip

\item $\Hom_{\fX}(L_1,Q)$ is a quotient of $R \oplus cR$, generated by $\beta = \binom{1}{0}$ and $\varphi_{32} = \binom{0}{c}$ under the relations $b_0 c \beta = a_0 \varphi_{32}$ and $b_1 c \beta = - \ap \varphi_{32}$. The map $R\oplus cR \to R$ given by $(x,y)\mapsto a_0x+b_0y$ gives an isomorphism $\Hom_{\fX}(L_1,Q) \cong (a_0,b_0c)$. 

\end{enumerate}
\end{theorem}

It remains to compute the ring structure on $\End_{\fX}(L_{-1}\oplus L_1 \oplus Q)$. We will do this by computing the individual composition maps
\[
\Hom_{\fX}(B,C) \times \Hom_{\fX}(A,B) \to \Hom_{\fX}(A,C), \,\,\,\,\,\, (f,g) \mapsto f \circ g,
\] 
for $A,B,C \in \{L_{-1},L_1,Q\}$. These are mostly straightforward but somewhat tedious computations. By the description in Theorem \ref{module structure nongen 2}(1), when $A=B$ the composition map is simply the $R$-module structure on $\Hom_{\fX}(B,C)$, and similarly when $B=C$, so that leaves the cases when $A\neq B$ and $B\neq C$. By Theorem \ref{line bundles nongen 2}, we also have $\End_{\fX}(L_{-1}\oplus L_1) = E$. In terms of the generators in Theorem \ref{module structure nongen 2}, this means that 
\[
\varphi_{12} \circ \varphi_{21}^i = b_i c\in R=\End_{\fX}(L_{-1})
\]
and 
\[
\varphi_{21}^i \circ \varphi_{12} = b_i c\in R=\End_{\fX}(L_1),
\]
for $i \in \{0,1\}$. Let us now move on to the composition maps involving $Q$, starting with 
\[
\Hom_{\fX}(Q,L_{-1}) \times \Hom_{\fX}(L_{-1},Q) \to \End_{\fX}(L_{-1}) \,\,\, \text{and} \,\,\, \Hom_{\fX}(Q,L_1) \times \Hom_{\fX}(L_{-1},Q)) \to \Hom_{\fX}(L_{-1},L_1).
\]
$\Hom_{\fX}(L_{-1},Q)$ is generated by the inclusion $\varphi_{31}$, and $\Hom_{\fX}(Q,L_{-1})$ has two generators $\varphi_{13}^0$ and $\varphi_{13}^1$, whose compositions with $L_1 \oplus L_{-1} \to Q$ are given by $(a_0 c, b_0 c)$ and $(-\ap c, b_1 c)$ in $cR \oplus R = \Hom_{\fX}(L_1 \oplus L_{-1},L_{-1})$. From this we see that
\[
\varphi_{13}^i \circ \varphi_{31} = b_i c \in R=\End_{\fX}(L_{-1})
\] 
for $i=0,1$. Next, $\Hom_{\fX}(Q,L_1)$ is the subspace of $R \oplus (b_0 R + b_1 R) = \Hom_{\fX}(L_1 \oplus L_{-1}, L_1)$ generated by $\varphi_{23}^0 = (a_0,b_0)$ and $\varphi_{23}^1 = (-\ap, b_1)$, so we see that
\[
\varphi_{23}^i \circ \varphi_{31} = \varphi_{21}^i \in \Hom_{\fX}(L_{-1},L_1)
\]
for $i=0,1$. Using that pre-composition with $\varphi_{31}$ is an isomorphism $\End_{\fX}(Q) \to \Hom_{\fX}(L_{-1},Q)$, we can now compute 
\[
\Hom_{\fX}(L_{-1},Q) \times \Hom_{\fX}(Q,L_{-1}) \to \End_{\fX}(Q) \,\,\, \text{and} \,\,\, \Hom_{\fX}(L_1,Q) \times \Hom_{\fX}(Q,L_1) \to \End_{\fX}(Q).
\]
Starting with $\Hom_{\fX}(L_{-1},Q) \times \Hom_{\fX}(Q,L_{-1}) \to \End_{\fX}(Q)$, consider the diagram
\[
\xymatrix{
\Hom_{\fX}(L_{-1},Q) \times \Hom_{\fX}(Q,L_{-1}) \ar[r] \ar[d] &  \End_{\fX}(Q) \ar[d] \\
\Hom_{\fX}(L_{-1},Q) \times \End_{\fX}(L_{-1}) \ar[r] & \Hom_{\fX}(L_{-1},Q).
}
\]
Since $\varphi_{13}^i \circ \varphi_{31} = b_i c$, we see that $\varphi_{31} \circ \varphi_{13}^i \circ \varphi_{31} = b_i c \varphi_{31}$ and hence
\[
\varphi_{31} \circ \varphi_{13}^i = b_i c \in \End(Q),
\]
for $i=0,1$. For $\Hom_{\fX}(L_1,Q) \times \Hom_{\fX}(Q,L_1) \to \End_{\fX}(Q)$, consider the diagram
\begin{equation}\label{eq: diagram}
\xymatrix{
\Hom_{\fX}(L_1,Q) \times \Hom_{\fX}(Q,L_1) \ar[r] \ar[d] &  \End_{\fX}(Q) \ar[d] \\
\Hom_{\fX}(L_1,Q) \times \Hom_{\fX}(L_{-1},L_1) \ar[r] & \Hom_{\fX}(L_{-1},Q).
}
\end{equation}
By Theorem \ref{module structure nongen 2}, $\Hom_{\fX}(L_1,Q)$ is a quotient of $\Hom_{\fX}(L_1, L_1 \oplus L_{-1}) = R \oplus cR$, generated by $\beta = \binom{1}{0}$ and $\varphi_{32} = \binom{0}{c}$, $\Hom_{\fX}(L_{-1},L_1) = b_0 R + b_1 R$, with $\varphi_{21}^i = b_i$ by definition. Also, recall that $\Hom_{\fX}(L_{-1},Q)$ is a quotient of $\Hom_{\fX}(L_{-1}, L_1 \oplus L_{-1}) = (b_0 R + b_1 R) \oplus R$, that $\varphi_{31} = \binom{0}{1}$ and that $\binom{b_0}{-a_0} = \binom{b_1}{\ap}=0$ in $\Hom_{\fX}(L_{-1},Q)$. In particular, we see that
\[
\varphi_{32} \circ \varphi_{21}^i = b_i c \varphi_{31}
\]
and
\[
\beta \circ \varphi_{21}^0 = a_0 \varphi_{31}, \,\,\,\,\,\, \beta \circ \varphi_{21}^1 = -\ap \varphi_{31}. 
\]
From diagram (\ref{eq: diagram}) we then see that 
\[
\beta \circ \varphi_{23}^0 = a_0, \,\,\,\,\,\, \beta \circ \varphi_{23}^1 = -\ap \,\,\, \text{and} \,\,\, \varphi_{32} \circ \varphi_{23}^i = b_i c,
\]
using that $\varphi_{23}^i \circ \varphi_{31} = \varphi_{21}^i$. Next, let us consider 
\[
\Hom_{\fX}(L_{-1},Q) \times \Hom_{\fX}(L_1,L_{-1}) \to \Hom_{\fX}(L_1,Q).
\]
Since $\varphi_{31} = \binom{0}{1}$ and $\varphi_{12} = c$ are the generators we see that we only need
\[
\varphi_{31} \circ \varphi_{12} = \varphi_{32}
\]
to describe this composition. We next consider
\[
\Hom_{\fX}(L_1,L_{-1}) \times \Hom_{\fX}(Q,L_1) \times \Hom_{\fX}(Q,L_{-1});
\]
from the descriptions of these $\Hom$-sets one sees directly that
\[
\varphi_{12} \circ \varphi_{23}^i = \varphi_{13}^i.
\]
Similarly one sees that the composition
\[
\Hom_{\fX}(L_{-1},L_1) \times \Hom_{\fX}(Q,L_{-1}) \times \Hom_{\fX}(Q,L_1) 
\]
is given by the relations
\[
\varphi_{21}^i \circ \varphi_{13}^j = b_i c \varphi_{23}^j,
\]
for $i,j=0,1$. Finally, we compute 
\[
\Hom_{\fX}(Q,L_{-1}) \times \Hom_{\fX}(L_1,Q) \to \Hom_{\fX}(L_1,L_{-1}) \,\,\, \text{and} \,\,\, \Hom_{\fX}(Q,L_1) \times \Hom_{\fX}(L_1,Q) \to \End_{\fX}(L_1).
\]
For the first one, by looking at the composition
\[
\Hom_{\fX}(L_1 \oplus L_{-1} , L_{-1}) \times \Hom_{\fX}(L_1, L_1 \oplus L_{-1}) \to \Hom_{\fX}(L_1,L_{-1})
\]
we see that
\[
\varphi_{13}^0 \circ \beta = a_0 \varphi_{12}, \,\,\,\,\,\, \varphi_{13}^1 \circ \beta = -\ap \varphi_{12} \,\,\, \text{and} \,\,\, \varphi_{13}^i \circ \varphi_{32} = b_i c \varphi_{12}.
\]
For the second one, we look at the composition
\[
\Hom_{\fX}(L_1 \oplus L_{-1} , L_1) \times \Hom_{\fX}(L_1, L_1 \oplus L_{-1}) \to \End_{\fX}(L_1)
\]
and see that
\[
\varphi_{23}^0 \circ \beta = a_0, \,\,\,\,\,\, \varphi_{23}^1 \circ \beta = -\ap \,\,\, \text{and} \,\,\, \varphi_{23}^i \circ \varphi_{32} = b_i c.
\]
This finishes the computation of the ring structure of $\End_{\fX}(L_{-1}\oplus L_1 \oplus Q)$. For ease of reference, we summarize the result in the following theorem.

\begin{theorem}\label{ring structure non-gen II}
With notation as in Theorem \ref{module structure nongen 2}, the $R$-algebra structure on $\widetilde{E}:=\End_{\fX}(L_{-1} \oplus L_1 \oplus Q)$ is determined by the following relations (for $i,j = 0,1$):
\begin{enumerate}
\item $\varphi_{12} \circ \varphi_{21}^i = b_i c$;

\item $\varphi_{12} \circ \varphi_{23}^i = \varphi_{13}^i$;

\item $\varphi_{13}^i \circ \varphi_{31} = b_i c$;

\item $\varphi_{21}^i \circ \varphi_{12} = b_i c$;

\item $\varphi_{21}^i \circ \varphi_{13}^j = b_i c \varphi_{23}^j$;

\item $\varphi_{23}^i \circ \varphi_{31} = \varphi_{21}^i$;

\item $\varphi_{31} \circ \varphi_{12} = \varphi_{32}$;

\item $\varphi_{31} \circ \varphi_{13}^i = b_i c$;

\item $\varphi_{13}^i \circ \varphi_{32} = b_i c \varphi_{12}$, $\varphi_{13}^0 \circ \beta = a_0 \varphi_{12}$ and $\varphi_{13}^1 \circ \beta = -\ap \varphi_{12}$;

\item $\varphi_{23}^i \circ \varphi_{32} = b_i c$, $\varphi_{23}^0 \circ \beta = a_0$ and $\varphi_{23}^1 \circ \beta = - \ap$;

\item $\varphi_{32} \circ \varphi_{21}^i = b_i c \varphi_{31}$, $\beta \circ \varphi_{21}^0 = a_0 \varphi_{31}$ and $\beta \circ \varphi_{21}^1 = -\ap \varphi_{31}$;

\item $\varphi_{32} \circ \varphi_{23}^i = b_i c$, $\beta \circ \varphi_{23}^0 = a_0$ and $\beta \circ \varphi_{23}^1 = -\ap$.

\end{enumerate}
\end{theorem}

Later on we will also need to consider the dual $Q^\ast = \Homi(Q,S)$, so we will now compute $Q^\ast$ explicitly. Since $Q$ is the cokernel of $M : L_{-1} \oplus L_{-1} \to L_1 \oplus L_{-1}$, $Q^\ast$ is the kernel of $M^t : L_{-1} \oplus L_1 \to L_1 \oplus L_1$ by duality. Consider the projective resolution (\ref{resolution of Q})
\[
\dots \overset{N}{\to} L_{-3} \oplus L_{-3} \overset{M}{\to} L_{-1} \oplus L_{-3} \overset{N}{\to} L_{-1} \oplus L_{-1} \overset{M}{\to}  L_1 \oplus L_{-1} \to Q \to 0.
\]
of $Q$. We can remove $Q$ and continue the resolution to the right to obtain an acyclic complex
\[
\dots \overset{N}{\to} L_{-1} \oplus L_{-1} \overset{M}{\to} L_1 \oplus L_{-1} \overset{N}{\to} L_1 \oplus L_1 \overset{M}{\to}  L_3 \oplus L_1 \overset{N}{\to} \dots.
\]
Dualizing this complex we obtain an acyclic complex
\[
\dots \overset{N^t}{\to} L_{-3} \oplus L_{-1} \overset{M^t}{\to} L_{-1} \oplus L_{-1} \overset{N^t}{\to} L_{-1} \oplus L_1 \overset{M^t}{\to}  L_1 \oplus L_1 \overset{N^t}{\to} \dots.
\]
From the acyclicity of this complex, we see that
\[
\Ker( L_{-1} \oplus L_1 \overset{M^t}{\to}  L_1 \oplus L_1 ) \cong \mathrm{Im}(L_{-1} \oplus L_{-1} \overset{N^t}{\to} L_{-1} \oplus L_1 ) \cong \Coker(L_{-3} \oplus L_{-1} \overset{M^t}{\to} L_{-1} \oplus L_{-1} ),
\]
showing that $Q^\ast$ is the cokernel of $M^t : L_{-3} \oplus L_{-1} \to L_{-1} \oplus L_{-1}$. We record this as a proposition.

\begin{proposition}\label{prop:Qvee description}
$Q^\ast$ may be explicitly described as the cokernel of $M^t : L_{-3} \oplus L_{-1} \to L_{-1} \oplus L_{-1}$, where we recall that $M^t = \left( \begin{smallmatrix} b_0 & -a_0 \\ b_1 & \ap \end{smallmatrix} \right)$ and that we view $L_{-3} \oplus L_{-1}$ and $L_{-1} \oplus L_{-1}$ as column vectors.
\end{proposition}

\subsection{Resolutions of simple modules in case non-generic II}\label{subsec:ng2 resolutions}
There are three isomorphism classes of simple left modules for the algebra $\widetilde{E}$. They are of the form $\widetilde{E}/(\varpi,J_i)$ for $i = 1,2,3$ where $J_i$ denotes one of the following two-sided ideals in $\widetilde{E}$:

\begin{align*}
		J_1&:=\begin{pmatrix}
		(a_0,a_1',b_0c,b_1c)\End_{\fX}(L_{-1}) & \Hom_{\fX}(L_1,L_{-1}) & \Hom_{\fX}(Q,L_{-1}) \\ \Hom_{\fX}(L_{-1},L_1) & \End_{\fX}(L_1) & \Hom_{\fX}(Q,L_1) \\ \Hom_{\fX}(L_{-1},Q) & \Hom_{\fX}(L_1,Q) & \End_{\fX}(Q)
	\end{pmatrix} \\
	J_2&:=\begin{pmatrix}
	\End_{\fX}(L_{-1}) & \Hom_{\fX}(L_1,L_{-1}) & \Hom_{\fX}(Q,L_{-1}) \\ \Hom_{\fX}(L_{-1},L_1) & (a_0,a_1',b_0c,b_1c)\End_{\fX}(L_1) & \Hom_{\fX}(Q,L_1) \\ \Hom_{\fX}(L_{-1},Q) & \Hom_{\fX}(L_1,Q) & \End_{\fX}(Q)
\end{pmatrix} \\
	J_3 &:=\begin{pmatrix}
	\End_{\fX}(L_{-1}) & \Hom_{\fX}(L_1,L_{-1}) & \Hom_{\fX}(Q,L_{-1}) \\ \Hom_{\fX}(L_{-1},L_1) & \End_{\fX}(L_1) & \Hom_{\fX}(Q,L_1) \\ \Hom_{\fX}(L_{-1},Q) & \Hom_{\fX}(L_1,Q) & (a_0,a_1',b_0c,b_1c)\End_{\fX}(Q)
\end{pmatrix}.
\end{align*}

We also consider the columns (from left to right) $C_1, C_2, C_3$ of $\widetilde{E}$, which are projective left $\widetilde{E}$-modules.

\begin{proposition}\label{prop:projresofsimple}
	The following complexes give projective resolutions of $\widetilde{E}/J_i$ for $i = 1,2,3$. The maps are written as matrices acting by right multiplication on row vectors (we act on the right so that we get maps of \emph{left} $\widetilde{E}$-modules).
	
	\begin{equation}\label{eq:res1}	C_3 \xrightarrow{\left(\begin{smallmatrix}\varphi_{31} &a_1'&-a_0\end{smallmatrix}\right)}  C_1\oplus C_3 \oplus C_3 \xrightarrow{\left(\begin{smallmatrix}
			-a_1' & a_0 & 0\\
			\varphi_{31}&0&-a_0\\
			0&\varphi_{31}&-a'_1
		\end{smallmatrix}\right)} C_1\oplus C_1 \oplus C_3 \xrightarrow{ \left(\begin{smallmatrix}a_0\\a_1'\\\varphi_{31}\end{smallmatrix}\right)} C_1 \to \widetilde{E}/J_1 \end{equation}

	\begin{equation}\label{eq:res2} C_3 \xrightarrow{\left(\begin{smallmatrix}\beta &\varphi_{31}\end{smallmatrix}\right)}  C_2\oplus C_1\xrightarrow{\left(\begin{smallmatrix}\varphi_{21}^0 & \varphi_{21}^1\\-a_0&a_1'\end{smallmatrix}\right)} C_1\oplus C_1 \xrightarrow{\left(\begin{smallmatrix}
			a_1' & -\varphi_{13}^1\\a_0&\varphi_{13}^0
		\end{smallmatrix}\right)} C_1\oplus C_3 \xrightarrow{ \left(\begin{smallmatrix}\varphi_{12}\\-\beta\end{smallmatrix}\right)} C_2 \to \widetilde{E}/J_2 \end{equation}
		
	\begin{equation}\label{eq:res3} C_2^{\oplus 2} \xrightarrow{\left(\begin{smallmatrix}1 &0&\varphi_{23}^1\\0&1&-\varphi_{23}^0\end{smallmatrix}\right)}  C_2^{\oplus 2}\oplus C_3\xrightarrow{\left(\begin{smallmatrix}a_1' & -b_1c\\a_0&b_0c\\\beta &\varphi_{32}\end{smallmatrix}\right)} C_2^{\oplus 2} \xrightarrow{M':=\left(\begin{smallmatrix}
			b_0c & b_1c\\-a_0&a_1'
		\end{smallmatrix}\right)} C_2^{\oplus 2} \xrightarrow{\pi := \left(\begin{smallmatrix}-\varphi_{23}^1\\\varphi_{23}^0\end{smallmatrix}\right)} C_3 \to \widetilde{E}/J_3  \end{equation}
\end{proposition}
\begin{proof}
	The computations required to check that these complexes are acyclic are similar in the three cases. We just explain the third one, as an example. To check that the image of $\pi$ is equal to the kernel in $C_3$ of the projection to $\widetilde{E}/J_3$, we use the facts that $\Hom_{\fX}(L_1,Q)\varphi_{23}^0 = (a_0,b_0c)\End_{\fX}(Q)$ and $\Hom_{\fX}(L_1,Q)\varphi_{23}^1 = (a'_1,b_1c)\End_{\fX}(Q)$, $\Hom_{\fX}(Q,L_1)$ is spanned by the maps $\varphi_{23}^i$, and $\Hom_{\fX}(Q,L_{-1})$ is spanned by the maps $\varphi_{12}\circ\varphi_{23}^i = \varphi_{13}^i$.
	
	We next check that the image of $M'$ is the kernel of $\pi$. Thinking about the various columns row-by-row, we need to check exactness of the following:
	
	\begin{align*}cR\oplus cR \xrightarrow{M'} cR\oplus cR \xrightarrow{\left(\begin{smallmatrix}
				a_1' & -b_1\\a_0 & b_0
			\end{smallmatrix}\right)} cR\oplus R \\ 
		R\oplus R \xrightarrow{M'} R\oplus R \xrightarrow{\left(\begin{smallmatrix}
				a_1' & -b_1\\a_0 & b_0
			\end{smallmatrix}\right)} R \oplus (b_0 R + b_1 R)\\ 
		\Hom_{\fX}(L_1,Q)^{\oplus 2} \xrightarrow{M'}  \Hom_{\fX}(L_1,Q)^{\oplus 2} \xrightarrow{\left(\begin{smallmatrix}
				-\varphi_{23}^1\\\varphi_{23}^0
			\end{smallmatrix}\right)} R   \end{align*}	The exactness of the first two rows can be shown in exactly the same way as for the resolution (\ref{resolution of Q}). For the first row, we use the fact that $c$ is not a zerodivisor. 
	
	For the third row, we first compute the kernel of the final map, recalling that $\Hom_{\fX}(L_1,Q)$ is spanned by $\varphi_{32}$ and $\beta$. The kernel is given by things of the form $(x_1\varphi_{32}+y_1\beta,x_2\varphi_{32}+y_2\beta)$ with $x_i, y_i \in R$ and $-x_1b_1c+y_1a_1'+x_2b_0c+y_2a_0 = 0$. Considering the relations in $\Hom_{\fX}(L_1,Q)$, we may assume that $x_i \in \OO\lb b_0c,b_1c\rb \subset R$ for $i=1,2$. But then the element $y_1a_1'+y_2a_0 = x_1b_1c - x_2b_0c \in \OO\lb b_0c,b_1c\rb \cap(a_0,a_1')=\{0\}$. We deduce from this that $x_1 = xb_0c$ and $x_2 = xb_1c$ for some $x \in \OO\lb b_0c,b_1c\rb $, and $(y_1,y_2) = (f,g)M'$ for some $f, g \in R$. Putting things together, we see that \[(x_1\varphi_{32}+y_1\beta,x_2\varphi_{32}+y_2\beta) = (x\varphi_{32}+f\beta,g\beta)M'\] is in the image of $M'$. 
	
	To show exactness of the third row in the next degree, we argue similarly: assume that $x_i \in \OO\lb b_0c,b_1c\rb $, and suppose $v = (x_1\varphi_{32}+y_1\beta,x_2\varphi_{32}+y_2\beta)$ is in the kernel of $M'$. We quickly deduce that $x_1 = 0$ and $(x_2-y_1,-y_2)M' = 0$, so $(y_1-x_2,y_2) = (x,y)N'$ for $x,y \in R$, where $N' = \left(\begin{smallmatrix}
		a_1' & -b_1c\\a_0 & b_0c
	\end{smallmatrix}\right)$. Now we have $v = (x\beta,y\beta)N'+ (x_2\beta,x_2\varphi_{32})$, as desired. Checking exactness everywhere else is straightforward.
\end{proof}

\section{Representation theory preliminaries}\label{chap: GL2preliminaries}
\subsection{Blocks for $\GL_2(\Q_p)$}\label{sec: blocks for GL2}

In this subsection we recall some material regarding the classification of smooth admissible irreducible $\Fpbar$-representations of $\GL_2(\Qp)$, describing only what we will need in this paper. The irreducibles fall into two groups, one consisting of the subquotients of principal series representations, and the other consisting of the supersingular representations. For simplicity of notation, set $G= \GL_2(\Qp)$. 

We begin by describing the former. Let $B$ be the subgroup of upper triangular matrices in $\GL_2(\Qp)$. Given two smooth characters $\chi_1,\chi_2 : \Qp^\times \to \Fpbar^\times$, we obtain a character $\chi_1 \otimes \chi_2 : B \to \Fpbar^\times$ by the formula
\[
\begin{pmatrix} a & b \\ 0 & d \end{pmatrix} \mapsto \chi_1(a)\chi_2(d).
\]
We let $\mbf{1} : \Qp^\times \to \Fpbar^\times$ denote the trivial character. The smooth parabolic induction $\Ind_B^G (\chi_1 \otimes \chi_2)$ is irreducible unless $\chi_1 = \chi_2$, in which case it is a non-split extension
\[
0 \to \chi \to \Ind_B^G (\chi \otimes \chi) \to \mathrm{St} \otimes (\chi \circ \det) \to 0,
\]
of irreducible representations, where $\mathrm{St} = \Ind_B^G (\mbf{1} \otimes \mbf{1}) / \mbf{1}$ is the Steinberg representation, and $\det : G \to \Qp^\times$ denotes the determinant. This describes all irreducibles that arise as subquotients of principal series representations. We remark that this parametrization is unique --  $\Ind_B^G (\chi_1 \otimes \chi_2)$ and $\Ind_B^G (\chi_1^\prime \otimes \chi_2^\prime)$ have no common irreducible subquotients unless $\chi_i = \chi_i^\prime $ for $i=1,2$.

The remaining irreducibles are the supersingular ones, which may be constructed as follows. Let $\F^2$ be the standard representation of $\GL_2(\Zp)$, and let $\sigma_r = \Sym^r \F^2$ for $r \in \{0,1,\dots,p-1\}$. Extend $\sigma_r$ to a representation of $K = Z.\GL_2(\Zp)$, where $Z$ denotes the center of $G$, by letting $\left( \begin{smallmatrix} p & 0 \\ 0 & p \end{smallmatrix}\right)$ act trivially. The compact induction $\mathrm{ind}_K^G \sigma_r$ has endomorphism ring isomorphic to a polynomial ring $\F[T]$, with $T$ being a certain Hecke operator, and the quotient
\[
\pi_r = (\mathrm{ind}_K^G \sigma_r)/(T)
\]
is an irreducible supersingular representation. More generally, we can consider $\pi_r \otimes (\chi \circ \det)$ for some $r \in \{0,1,\dots,p-1\}$ and some smooth $\chi : \Qp^\times \to \Fpbar^\times$. In this case, the data $(r,\chi)$ is not uniquely determined by the (isomorphism class of the) representation $\pi_r \otimes (\chi \circ \det)$. Firstly, twisting the character $\chi$ by the unique non-trivial unramified quadratic character does not change the isomorphism class. Secondly, we have $\pi_r \cong \pi_{p-1-r}\otimes (\omega^r \circ \det)$.

Next, we recall from the introduction the partition of irreducible representations of $\GL_2(\Qp)$ into blocks, as described in \cite{paskunas-blocks} (recall that we are assuming that $p\geq 5$). The blocks of $\Mod_{G,\zeta}^{lfin}(\oo)$ containing absolutely irreducible representations are:
\begin{enumerate}
\item $\mf{B} = \{ \pi \}$, where $\pi$ is supersingular;

\item $\mf{B} = \{ \Ind_B^G (\delta_1 \otimes \delta_2 \omega^{-1}), \Ind_B^G (\delta_2 \otimes \delta_1 \omega^{-1}) \}$ with $\delta_2 \delta_1^{-1} \neq \mbf{1}, \omega^{\pm 1}$;

\item $\mf{B} = \{ \Ind_B^G (\delta \otimes \delta \omega^{-1}) \}$;

\item $\mf{B} = \{ \delta \circ \det, \mathrm{St} \otimes (\delta \circ \det), \Ind_B^G (\delta \omega \otimes \delta \omega^{-1}) \}$;
\end{enumerate} 

In accordance with the terminology used in \cite{paskunas-image}, we refer to the blocks of type (1) as \emph{supersingular}, blocks of type (2) as \emph{generic principal series} (or \emph{generic residually reducible}), blocks of type (3) as \emph{non-generic case I} and blocks of type (4) as \emph{non-generic case II}. These blocks are in bijective correspondence with   isomorphism classes of {semisimple} continuous Galois representations $\Gamma_{\Qp} \to \GL_2({\F})$ which are either reducible or absolutely irreducible. We recall this briefly, using Colmez's Montr\'{e}al functor \cite{colmez-functor}; we will follow the notation in \cite[\S 5.7]{paskunas-image}. Let $\Mod_{\Gamma_{\Qp}}^{fin}(\oo)$ be the category of continuous $\Gamma_{\Qp}$-representations on finite length $\oo$-modules. Let $\Mod_{G,Z}^{fin}(\oo)$ be the full subcategory of
$\Mod_{G}^{sm}(\oo)$ consisting of representations of finite length with a central character. The Montr\'{e}al functor is an exact functor
\[
\mbf{V} : \Mod_{G,Z}^{fin}(\oo) \to \Mod_{\Gamma_{\Qp}}^{fin}(\oo).
\] 
If $\delta : \Qp^\times \to \oo^\times$ is a continuous character, then $\mbf{V}(\pi \otimes (\delta \circ \det)) \cong \mbf{V}(\pi) \otimes \delta $ naturally for all $\pi$ in $\Mod_{G,Z}^{fin}(\oo)$. Following Pa{\v s}k{\=u}nas, we will also use the renormalization 
\[
\vc (\pi) = \mbf{V}(\pi)^\vee \otimes \varepsilon \zeta_\pi,
\]
where $\zeta_\pi$ denotes the central character of $\pi$. The functor $\vc$ is contravariant, exact and still satisfies $\vc(\pi \otimes (\delta \circ \det)) \cong \vc(\pi) \otimes \delta $ . It has the following values:
\[
\vc(\pi_r) = \Ind_{\Gamma_{\mb{Q}_{p^2}}}^{\Gamma_{\Qp}} \omega_2^{r+1}, \,\,\, \vc(\Ind_B^G(\delta_1 \otimes \delta_2 \omega^{-1})) = \delta_1.
\] 
Here $\omega_2 : \Q_{p^2}^\times \to \F_{p^2}^\times$ is given by $\omega_2(x) = x \cdot |x| \mod p$. Note that the induced representation $ \Ind_{\Gamma_{\mb{Q}_{p^2}}}^{\Gamma_{\Qp}} \omega_2^{r+1}$ descends to a representation defined over $\Fp$. 

We can then define a map
\[
\mf{B} \mapsto \rho_{\mf{B}}
\]
from blocks containing an absolutely irreducible representation  to semisimple reducible or absolutely irreducible $2$-dimensional representations of $\Gamma_{\Qp}$ over $\F$, by sending a supersingular block $\mf{B} = \{ \pi_r \}$ to $\vc(\pi_r)$ and sending a block $\mf{B}$ of type (2), (3 or (4) above to
\[
\delta_1 \oplus \delta_2 = \vc(\Ind_B^G(\delta_1 \otimes \delta_2 \omega^{-1}) \oplus \Ind_B^G(\delta_2 \otimes \delta_1 \omega^{-1})),
\]
where $\delta_1$ and $\delta_2$ are the two characters defining the block (with $\delta_1 = \delta_2$ for blocks of type (3), i.e.\ non-generic case I). This map is a bijection. Extending scalars to a splitting field \cite[Prop.~5.3]{paskunas-image} shows that we have a bijection between arbitrary blocks in $\Mod_{G,Z}^{fin}(\oo)$ and $\Gal(\overline{\F}/\F)$-orbits of isomorphism classes of {semisimple} continuous Galois representations $\Gamma_{\Qp} \to \GL_2({\F})$. We can moreover identify this set with the set of 2-dimensional residual pseudorepresentations (relative to $\OO$), defined as in \cite[Defn.~3.11]{chenevier-determinant}.

\subsection{Categorical constructions}\label{sec: categorical constructions}

We now prove some results that will allow us to interpret the $p$-adic local Langlands correspondence for $\GL_2(\Qp)$ as a fully faithful embedding of derived (and sometimes abelian) categories, by abstracting the main properties of the situation. In this subsection, a \emph{finite} module will always mean a module with finite cardinality. 

Our starting point is a (not necessarily commutative) $\oo$-algebra $E$.

\begin{assumption}\label{assumptions on E}
Throughout this subsection, we make the following assumptions on $E$:
\begin{enumerate}

\item The center of $E$, which we denote by $R$, is a complete Noetherian local ring whose maximal ideal we denote by $\m$, and whose residue field we assume to be finite;

\item $E$ is a finitely generated $R$-module;

\item Every simple right $E$-module has finite projective dimension.

\end{enumerate}
\end{assumption}
These properties abstract the main properties of the endomorphism rings appearing in \cite{paskunas-image}. Note that every simple right $E$-module is killed by $\m$, by Nakayama's lemma, and is therefore finite. We equip every finitely generated $R$-module (and hence every finitely generated left or right $E$-module) with its $\m$-adic topology. With respect to this all finitely generated $R$-modules are profinite, and all $R$-linear maps are automatically continuous and closed. Note that the first two assumptions imply that $E$ is (left and right) Noetherian. We start by noting that $E$ has finite global dimension.

\begin{proposition}\label{finite global dimension}
$E$ has finite global dimension. 
\end{proposition} 

\begin{proof}
Since $E$ is Noetherian, the left and right global dimensions agree and are equal to the weak global dimension (cf.~e.g.~\cite[Cors.~2.6.7 and 2.6.8]{cohn}), so it suffices to show that the weak global dimension is finite. Since $\Tor^E_n(\varinjlim_i M_i, \varinjlim_j N_j) = \varinjlim_{i,j} \Tor^E_n(M_i,N_j)$, it suffices to consider finitely generated left and right $E$-modules.

Since $E/\m E$ is finite (as a set), there are only finitely many simple right $E$-modules. In particular, we may find a $d\in \Z_{\geq 1}$ such that every simple right $E$-module has projective dimension $\leq d$. By d\'{e}vissage, it follows that any finite right $E$-module has projective dimension $\leq d$. Now let $M$ be a finitely generated right $E$-module and let $N$ be a finitely generated left $E$-module. Choose a (possibly infinite) resolution $P_\bu \to N$ by finitely generated free left $E$-modules, and set $M_n = M/\m^n M$; we then have $M = \varprojlim_n M_n$. By exactness of inverse limits in the abelian category of compact Hausdorff abelian groups, we see that
\[
\Tor^E_i(M,N) = H_i(\varprojlim_n M_n \otimes_E P_\bu) = \varprojlim_n H_i(M_n \otimes_E P_\bu) = \varprojlim_n \Tor_n^E(M_n,N)
\]
(we use finite freeness of the terms in $P_\bu$ to equate $\varprojlim_n (M_n \otimes_E P_\bu)$ and $(\varprojlim_n M_n) \otimes_E P_\bu$). Since each $M_n$ is a finite right $E$-module and hence has projective dimension $\leq d$, we may deduce that $\Tor^E_i(M,N) = 0$ for all $i > d$. It follows that the weak global dimension of $E$ is $\leq d$, as desired, finishing the proof. 
\end{proof}

We will consider the abelian categories $\LMod_{disc}(E)$ and $\RMod_{cpt}(E)$ of discrete topological left $E$-modules and compact topological right $E$-modules, respectively. Note that $\LMod_{disc}(E)$ and $\RMod_{cpt}(E)$ are anti-equivalent to each other via Pontryagin duality (where $M^\vee = \Hom^{cts}_{\OO}(M,L/\OO)$). We have the following well known descriptions of $\LMod_{disc}(E)$ and $\RMod_{cpt}(E)$:

\begin{proposition}\label{description of top E-modules}
Any discrete left $E$-module is the direct limit of its finite $E$-submodules. Dually, every compact right $E$-module is an inverse limit of finite $E$-modules. In addition, this holds categorically, i.e.\ if $\LMod_{fin}(E)$ and $\RMod_{fin}(E)$ are the categories of finite\footnote{Note that finite $E$-modules are automatically discrete, since they are Artinian finitely generated $R$-modules and hence killed by a power of $\m$. In particular, an $E$-module is finite if and only if it is finitely generated and discrete.} left and right $E$-modules (respectively), then $\LMod_{disc}(E) = \Ind(\LMod_{fin}(E))$ and $\RMod_{cpt}(E) = \Pro(\RMod_{fin}(E))$.
\end{proposition}

\begin{proof}
Since Pontryagin duality preserves finiteness the statements about compact $E$-modules follow from those for discrete $E$-modules. To prove the first part, let $M\in \LMod_{disc}(E)$ and $m\in M$. By discreteness, the annihilator of $m$ is open, and hence the $E$-submodule of $M$ generated by $m$ is finite. It is also clear that if $M_1,M_2 \sub M$ are finite submodules, then $M_1+M_2$ is finite as well. This finishes the proof at the level of objects, and at the level of morphisms the first assertion follows from (categorical) compactness of finite $E$-modules (which is obvious).
\end{proof}

If $M$ is an abstract $E$-module, we let $M[\m^\infty]$ denote the submodule of $\m^\infty$-torsion elements, i.e.\ those elements that are killed by some power of $\m$. Proposition \ref{description of top E-modules} then shows that $\LMod_{disc}(E)$ is the full subcategory of $\m^\infty$-torsion modules in the category $\LMod(E)$ of all left $E$-modules. In particular, $\LMod_{disc}(E)$ is a Grothendieck abelian category (a generator is given by $\bigoplus_n E/\m^n E$).

As mentioned, our goal is to produce embeddings of $\LMod_{disc}(E)$ and $\RMod_{cpt}(E)$ into categories of quasicoherent sheaves (roughly speaking), as well as derived analogues. For quasicoherent sheaves, we will use the setup of \S \ref{sec: coh sheaves on stacks}, with a few additional assumptions. In particular, we let $G$ be a reductive group scheme over $\oo$ and let $A$ be a commutative Noetherian $\oo$-algebra with an action of $G$. Moreover, we also assume that $A^G$ is isomorphic to $R$ (which we treat as an equality $R=A^G$) and that $A$ is Gorenstein. As in \S \ref{sec: coh sheaves on stacks} we set $T = \Spec A$ and let $\mf{X}$ be the quotient stack $[T/G]$. We then have $\Coh(\mf{X})$ and $\QCoh(\mf{X})$ as defined in \S \ref{sec: coh sheaves on stacks}. We let $\Coh_{\m}(\mf{X})$ and $\QCoh_{\m}(\mf{X})$ denote the full subcategories of $\Coh(\mf{X})$ and $\QCoh(\mf{X})$, respectively, whose objects are $\m^\infty$-torsion.

In the abelian category setting, our starting point to produce functors is an object $V$ in $\Coh(\mf{X})$. 

\begin{assumption}[Abelian setting]\label{assumptions abelian}

We make the following assumptions on $V$ in the abelian setting:
\begin{enumerate}
\item $E = \End(V)$;

\smallskip

\item $V$ and its coherent dual $V^\ast = \Homi(V,\OO_{\mathfrak{X}})$ are projective in $\QCoh(\mf{X})$ (cf.~Remark \ref{remark:duality} on duality);

\smallskip

\item $V$ is a flat left $E$-module and  $V^\ast$ is a flat right $E$-module.
\end{enumerate}
\end{assumption}
Recall that we have defined $\QCoh(\mf{X})$ as the category of $G$-equivariant $A$-modules, so its objects may be viewed as $A$-modules (and in particular abelian groups). This allows us to view $V$ above as a left $E$-module, and $V^\ast$ (whose underlying $A$-module is $\Hom_A(V,A)$) as a right $E$-module. We may then define functors $F : \LMod(E) \to \QCoh(\mf{X})$ and $F^\prime : \RMod(E) \to \QCoh(\mf{X})$ by
\[
N \mapsto F(N) := V^\ast \otimes_E N, \,\,\, \text{and}\,\,\, M \mapsto F^\prime(M) := M \otimes_E V.
\]
Both functors are exact by flatness of $V$ and $V^\ast$. 

\begin{theorem}\label{abelian embedding}
The functors $F$ and $F^\prime$ are fully faithful. Moreover, $F$ sends $\LMod_{fin}(E)$ to $\Coh_{\m}(\mf{X})$ and $F^\prime$ sends $\RMod_{fin}(E)$ into $\Coh_{\m}(\mf{X})$. In particular, restriction of $F$ gives a fully faithful functor $F_{disc} : \LMod_{disc}(E) \to \QCoh_{\m}(\mf{X})$ and $F^\prime$ induces a fully faithful functor $F_{cpt} : \RMod_{cpt}(E) \to \Pro(\Coh_{\m}(X))$.
\end{theorem}

\begin{proof}
We prove the first two statements for $F$: the proofs for $F^\prime$ are exactly the same. We prove full faithfulness first. Let $\LMod_{fg}(E) \sub \LMod(E)$ be the full subcategory of finitely generated left $E$-modules. Note that $F$ sends $\LMod_{fg}(E)$ into $\Coh(\mf{X})$, that it commutes with direct limits, that $\LMod(E) = \Ind(\LMod_{fg}(E))$ and finally that $\QCoh(\mf{X}) = \Ind(\Coh(\mf{X}))$ by \cite[Lem.~2.9]{arinkin-bezrukavnikov}. It therefore suffices to prove that $F$ is fully faithful on $\LMod_{fg}(E)$. So, let $M,N \in \LMod_{fg}(E)$ and consider the map
\[
\Hom(M,N) \to \Hom(F(M),F(N)).
\]
For objects of the form $M=E^m$, $N=E^n$ we obtain an isomorphism by the assumption that $\End(V)=E$. Next, assume that $M=E^m$ but let $N$ be arbitrary and choose a presentation $E^r \to E^s \to N \to 0$. Applying $F$ we get a presentation $(V^\ast)^r \to (V^\ast)^s \to F(N) \to 0$ and an induced commutative diagram
\[
\xymatrix{
\Hom(E^m,E^r) \ar[r] \ar[d] & \Hom(E^m,E^s) \ar[r] \ar[d] & \Hom(E^m,N) \ar[r] \ar[d] & 0 \ar[d] \\
\Hom((V^\ast)^m,(V^\ast)^r) \ar[r] &  \Hom((V^\ast)^m,(V^\ast)^s) \ar[r] & \Hom((V^\ast)^m,F(N)) \ar[r] & 0. 
}
\]
The bottom row is exact by projectivity of $V^\ast$ and the two leftmost vertical arrows are isomorphism, so by the five lemma the third vertical arrow is an isomorphism as desired. It remains to deal with the case when both $M$ and $N$ are arbitrary. This is proved by the same type of argument, choosing a presentation for $M$. This finishes the proof of fully faithfulness.

For the final part, it is clear that $\m^\infty$-torsion $E$-modules are sent to $\m^\infty$-torsion sheaves, so $F$ and $F^\prime$ send $\LMod_{fin}(E)$ and $\RMod_{fin}(E)$) fully faithfully into $\Coh_{\m}(X)$. The final part is then proved by Ind-extension and Pro-extension, respectively.
\end{proof}

From now on, we will no longer talk about $F^\prime$ (it is entirely parallel to $F$, and its main purpose was just to define $F_{cpt}$). For completeness, we will record that the functors we construct have adjoints (see also Proposition \ref{derived adjoints} and the discussion following it). While we will not make use of these adjoints in this paper, they should play an interesting role in the categorical $p$-adic local Langlands program. For more motivation and a sample of this, we refer to \cite[\S 7.8]{egh}.
\begin{proposition}\label{abelian adjoints}
The functors $F$, $F_{disc}$ and $F_{cpt}$ have the following adjoints:
\begin{enumerate}
\item $F$ has a right adjoint $G : \QCoh(X) \to \LMod(E)$ given by $G(W) = \Hom(V^\ast,W)$. Moreover, $G$ is exact and commutes with limits and colimits.

\smallskip

\item $F_{disc}$ has a right adjoint $G_{disc} : \QCoh_{\m}(X) \to \LMod_{disc}(E)$ given by $G_{disc}(W) = \Hom(V^\ast,W)$. Moreover, $G_{disc}$ is exact and commutes with limits and colimits.

\smallskip

\item $F_{cpt}$ has a left adjoint $G_{cpt} : \Pro(\Coh_{\m}(X)) \to \RMod_{cpt}(E)$. Moreover, $G_{cpt}$ commutes with colimits and cofiltered limits.
\end{enumerate}
\end{proposition}

\begin{proof}
We start with (1). The adjunction between $F$ and $G$ is the usual Hom-Tensor adjunction (one checks easily that it is compatible with $G$-equivariance). Exactness of $G$ is then precisely projectivity of $V^\ast$. Finally, $G$ commutes with limits by definition, and it commutes with colimits since $V^\ast$ is compact in $\QCoh(X)$ and projective.

Given part (1), the statement in part (2) is that the restriction of $G$ to $\QCoh_{\m}(X)$ lands inside $\LMod_{disc}(E)$. Since $G$ commutes with colimits, it suffices to check that if $W \in \Coh_{\m}(X)$ is $\m^n$-torsion, then $\Hom(V,W)$ is $\m^n$-torsion, but this is clear (by compatibility of the two $R$-module structures we have on $V$).

Finally, for part (3), the existence of $G_{cpt}$ follows from the (special) adjoint functor theorem (see e.g.\ \cite[\S V.8, Cor.]{maclane}), since $F_{cpt}$ commutes with limits (it is exact, and commutes with cofiltered limits by definition). Being a left adjoint $G_{cpt}$ automatically commutes with colimits. We now show that it commutes with cofiltered limits. Let $(W_i)$ be a cofiltered diagram in $\Pro(\Coh_{\m}(X))$. Consider the natural map $G_{cpt}(\varprojlim_i W_i) \to \varprojlim_i G_{cpt}(W_i)$. To prove that it is an isomorphism, it suffices to show that the induced map
\[
\Hom(\varprojlim_i G_{cpt}(W_i), M) \to \Hom(G_{cpt}(\varprojlim_i W_i), M)
\]
is an isomorphism for all $M\in \RMod_{cpt}(E)$. Since $\RMod_{cpt}(E) = \Pro(\RMod_{fin}(E))$, we may assume that $M \in \RMod_{fin}(E)$ and hence is cocompact. Then, observing that $F_{cpt}$ preserves cocompact objects (since $\Coh_{\m}(X) \sub \Pro(\Coh_{\m}(X))$ are precisely the cocompact objects by construction), we see that
\[
\Hom(\varprojlim_i G_{cpt}(W_i), M) = \varinjlim_i \Hom( G_{cpt}(W_i), M) = \varinjlim_i \Hom(W_i, F_{cpt}(M)) =
\]
\[
=\Hom(\varprojlim_i W_i, F_{cpt}(M)) = \Hom(G_{cpt}(\varprojlim_i W_i), M),
\]
as desired.
\end{proof}

\begin{remark}
Note that, by the adjoint functor theorem, $G$ and $G_{disc}$ also have right adjoints. It is not clear to us if $G_{cpt}$ has a left adjoint (but its derived analogue will have a left adjoint, see Remark \ref{rmk on adjoints derived}).
\end{remark}

This gives us what we need for embeddings of abelian categories, and this setup will allow us to construct functors for the supersingular and generic blocks. For the non-generic blocks, we can only produce embeddings of derived categories (at least a priori), using objects $V$ with weaker properties than projectivity (and flatness). 

For the formulation we want, we need some more categorical preliminaries. We start by observing that, by \cite[Cor.~B.1.16]{egh}, injective objects in $\LMod_{disc}(E)$ are also injective in $\LMod(E)$\footnote{The ``Artin--Rees property'' needed to apply \cite[Cor.~B.1.16]{egh} just reduces to the usual Artin--Rees lemma for $R$-modules in our setup.}. We will use the conventions for derived ($\infty$-)categories that we set up in \S \ref{sec: coh sheaves on stacks}, with the following additions: set $\D^L(E) := \D(\LMod(E))$, $\D^{L,+}(E) := \D^+(\LMod(E))$ and $\D^R(E) := \D(\RMod(E))$. By \cite[Prop.~B.1.17]{egh} the natural map $\D^+(\LMod_{disc}(E)) \to \D^{L,+}(E)$ is fully faithful and its essential image, which we denote by $\D_{disc}^{L,+}(E)$, has objects the complexes in $\D^{L,+}(E)$ whose cohomology groups are in $\LMod_{disc}(E)$. In fact, we may extend full faithfulness to unbounded derived categories:

\begin{lemma}\label{embdding of unbounded deriv cat}
The natural map $\D(\LMod_{disc}(E)) \to \D^L(E)$ is fully faithful.
\end{lemma}

\begin{proof}
Consider the inclusion $\LMod_{disc}(E) \sub \LMod(E)$. As noted above, $\D^+(\LMod_{disc}(E)) \to \D^{L,+}(E)$ is fully faithful. To check that $\D(\LMod_{disc}(E)) \to \D^L(E)$ is fully faithful, it suffices to check that the derived functors of the right adjoint of $\LMod^{disc}(E) \sub \LMod(E)$ have bounded cohomological dimension by \cite[Prop.~A.7.3]{egh}. This right adjoint is $M \mapsto M[\m^\infty]$. It can be written as
\[
M \mapsto M[\m^\infty] = \varinjlim_n \Hom_E(E/\m^n E,M),
\]
so its derived functors are $M \mapsto \varinjlim_n \Ext^i_E(E/\m^n E,M)$. Since $E$ has finite global dimension by Proposition \ref{finite global dimension}, the derived functors vanish for $i$ sufficiently large, as desired.
\end{proof}

We will denote the essential image of $\D(\LMod_{disc}(E)) \to \D^L(E)$ by $\D_{disc}^L(E)$, and conflate it with $\D(\LMod_{disc}(E))$. Now consider the dg category $\Proj^L(E)$ consisting of bounded complexes of finitely generated projective left $E$-modules. Its dg nerve $\Perf^L(E)$ is the stable $\infty$-category of perfect (left) complexes. We recall that $\Perf^L(E)$ is equal to the full subcategory of compact objects of $\D^L(E)$. Moreover, it generates $\D^L(E)$, so we have $\Ind\Perf^L(E) \cong \D^L(E)$. We define $\Perf_{disc}^L(E)$ to be the full subcategory of $\D^L(E)$ whose objects are contained in both $\D_{disc}^L(E)$ and $\Perf^L(E)$.

\begin{proposition}\label{unbounded disc category}
We have an equivalence $\Ind(\Perf_{disc}^L(E)) \cong \D_{disc}^L(E)$.
\end{proposition}

\begin{proof}
This follows if we show that $\Perf_{disc}^L(E)$ generates $\D_{disc}^L(E)$ and consists of compact objects. First, note that the objects of $\Perf_{disc}^L(E)$ are compact in $\D_{disc}^L(E)$ since they are compact in the larger stable $\infty$-category $\D^L(E)$. It remains to show that $\Perf_{disc}^L(E)$ generates $\D_{disc}^L(E)$. To show this, first note that $E/\m^n E \in \Perf_{disc}^L(E)$ by Proposition \ref{finite global dimension}. It suffices to show that for any nonzero $C^\bu \in \D_{disc}^L(E)$, there exists an $n\in \Z_{\geq 0}$ and $m \in \Z$ such that $\Hom(E/\m^n E[-m],C^\bu)\neq 0$. Since $C^\bu \neq 0$, there is an $m$ such that $H^m(C^\bu)\neq 0$. Then we can find an element $x\in C^m$ which maps to a nonzero element in $H^m(C^\bu)$. Since $C^m$ is an $\m^\infty$-torsion $E$-module (we can always choose $C^\bu$ to have terms in $\LMod_{disc}(E)$), we can find an $n$ and a map $E/\m^n E \to C^m$ sending $1$ to $x$. This induces a nonzero map $E/\m^nE[-m] \to C^\bu$ in $\D_{disc}^L(E)$, as desired.
\end{proof}

We can now start the construction of the derived analogue of Theorem \ref{abelian embedding}. Our starting point is now a maximal Cohen--Macaulay sheaf $V$ on $\mf{X}$ and an assumption analogous to Assumption \ref{assumptions abelian}. 

\begin{assumption}[Derived setting]\label{assumptions derived}

We make the following assumptions on $V \in \MCM(\mf{X})$ in the derived setting:
\begin{enumerate}
\item $E=\End(V)$;

\smallskip

\item $\Ext^i(V,V)=0$ for all $i\geq 1$.
\end{enumerate}
\end{assumption}
By duality, $E^{op}=\End(V^\ast)$ and $\Ext^{i}(V^\ast,V^\ast)=0$ as well; note that $V^\ast \in \mathrm{MCM}(\mf{X})$ as well. Let $\Proj^L(E)$ and $\Proj^R(E)$ be the (strongly pretriangulated) dg categories of bounded chain complexes of finitely generated projective left and right $E$-modules, respectively, and let $\Ch^b(\Coh(\mf{X}))$ be the (strongly pretriangulated) dg category of bounded chain complexes in $\Coh(\mf{X})$. The sheaves $V$ and $V^\ast$ give dg functors
\[
F : \Proj^L(E) \to \Ch^b(\Coh(\mf{X})), \,\,\, F(P_\bu) = V^\ast \otimes_E P_\bu;
\]
\[
F^\prime : \Proj^R(E) \to \Ch^b(\Coh(\mf{X})), \,\,\, F^\prime(Q_\bu) = Q_\bu \otimes_E V.
\]
Taking dg nerves and inverting the quasi-isomorphisms on the right hand side, we get induced exact functors
\begin{equation}\label{eqn: derived F Fprime}
F : \Perf^L(E) \to \D^b_{coh}(\mf{X}) \,\,\,\, \text{and} \,\,\, F^\prime : \Perf^R(E) \to \D^b_{coh}(\mf{X}).
\end{equation}
To check full faithfulness, by a standard argument (cf.~the proof of \cite[Thm.~4.30]{hellmann-derived}) it suffices to check that these functors induce isomorphisms on $\Ext$ groups when applied to finite projective $E$-modules. This follows from our assumptions on $V$. 

Note that this construction does give us $\D_{coh}^b(\mf{X})$ on the right hand side; this follows from \cite[Cor.~2.11]{arinkin-bezrukavnikov}. Here (and throughout this subsection) we have used the remark in \S \ref{sec: coh sheaves on stacks} that, for the $\infty$-categories we consider here, full faithfulness can be checked on the underlying homotopy category.

\begin{proposition}
The functors $F$ and $F^\prime$ map $\Perf^L_{disc}(E)$ and $\Perf^R_{disc}(E)$ into the full sub-$\infty$-category $\D_{coh,\m}^b(\mf{X})$ of $\D_{coh}^b(\mf{X})$ whose objects are those whose cohomology groups are $\m^\infty$-torsion.
\end{proposition}

\begin{proof}
We prove it for $F$; the argument for $F^\prime$ is identical. Assume that $P_\bu \in \Proj^L(E)$ has $\m^\infty$-torsion cohomology; we need to show that $V^\ast \otimes_E P_\bu$ has $\m^\infty$-torsion cohomology. This follows from the hypertor spectral sequence (see e.g.\ \cite[Application 5.7.8]{weibel}): We have a spectral sequence
\[
E^2_{ij} = \Tor_{i}^E(V^\ast,H_j(P_\bu)) \implies H_{i+j}(V^\ast \otimes_E P_\bu),
\]
and hence if all $H_j(P_\bu)$ are $\m^\infty$-torsion it follows that all $H_{i+j}(V^\ast \otimes_E P_\bu)$ will be $\m^\infty$-torsion as well. This finishes the proof.
\end{proof}

Summing up, we have fully faithful embeddings $F : \Perf^L(E) \to \D^b_{coh}(\mf{X})$ and $F^\prime : \Perf^R(E) \to \D_{coh}^b(\mf{X})$ which restrict to fully faithful embeddings $F : \Perf_{disc}^L(E) \to \D^b_{coh,\m}(\mf{X})$ and $F^\prime : \Perf_{disc}^R(E) \to \D_{coh,\m}^b(\mf{X})$. Taking $\Ind$-completions of the functors $F$, we obtain a fully faithful embeddings
\[
F : \D^L(E) = \Ind(\Perf^L(E)) \to \IndCoh(\mf{X}) := \Ind \D^b_{coh}(\mf{X}).
\]
and (using Proposition \ref{unbounded disc category})
\[
F_{disc} : \D_{disc}^L(E) = \Ind(\Perf_{disc}^L(E)) \to \IndCoh_{\m}(\mf{X}) := \Ind \D^b_{coh,\m}(\mf{X}).
\]
This gives two out the three functors that we want. For the third one, we need some more preparation.

\begin{lemma}
Every $P_\bu \in \Perf_{disc}^L(E)$ is quasi-isomorphic to a bounded complex of injectives $J_\bu$ in $\LMod_{disc}(E)$ such that the Pontryagin dual $J_n^\vee$ is a finitely generated projective right $E$-module for all $n$. Conversely, every bounded complex $J_\bu$ in $\LMod_{disc}(E)$ of injectives with $J_n^\vee$ a finitely generated projective right $E$-module for all $n$ and with finite cohomology groups is perfect.
\end{lemma}

\begin{proof}
We prove this by induction on the amplitude of $P_\bu$. When the homology of $P_\bu$ is concentrated in a single degree $n$, then $P_\bu$ is quasi-isomorphic to $M := H_n(P_\bu)[-n]$ and the latter is finite. Consider $M^\vee$, which is a finite discrete right $E$-module. By Proposition \ref{finite global dimension}, $E$ has finite global dimension, so $M^\vee$ has a finite resolution by finitely generated projective right $E$-modules. Taking Pontryagin duals we obtain the desired resolution of $M$. For the induction step, we may choose $P_\bu \in \Proj^L(E)$ with discrete cohomology and $P_n \neq 0$ only if $n \in [r,s]$, and with $H_r(P^\bu)\neq 0$. Consider the truncations $\tau_{> r}P_\bu$ and $\tau_{\leq r}P_\bu = H_r(P_\bu)[-r]$. We know that $H_r(P_\bu)[-r]$ is perfect since $E$ has finite global dimension, so it follows that $\tau_{> r}P_\bu$ is perfect as well, since it is the cone of $P_\bu[-1] \to \tau_{\leq r}P_\bu[-1]$. We can therefore apply the induction hypothesis to $\tau_{> r}P_\bu$ and $\tau_{\leq r}P_\bu$, and get the result for $P_\bu$ by writing it as the cone of $\tau_{\leq r}P_\bu[-1] \to \tau_{> r}P_\bu$. This gives the first statement, and the proof of the converse is entirely dual.
\end{proof}

The following corollary is then immediate.

\begin{corollary}\label{cor to description of discrete perfect objects}
Pontryagin duality induces an equivalence $\Perf^R_{disc}(E) \cong \Perf^L_{disc}(E)^{op}$. 
\end{corollary}

If $\mc{A}$ is a Grothendieck abelian category, then we recalled in \S \ref{sec: coh sheaves on stacks} that the unbounded derived $\infty$-category $\D(\mc{A})$ is defined in \cite[\S 1.3.5]{lurie-ha}. If $\mc{A}$ is an abelian category such that $\mc{A}^{op}$ is a Grothendieck abelian category, we may define $\D(\mc{A}) := \D(\mc{A}^{op})^{op}$. Note that this is a stable $\infty$-category by \cite[Rem.~1.1.1.3]{lurie-ha}, and that one already has a canonical equivalence $\D^-(\mc{A}) \cong \D^+(\mc{A}^{op})^{op}$ \cite[Variant 1.3.2.8]{lurie-ha}, so this definition of $\D(\mc{A})$ is reasonable.

\begin{corollary}\label{unbounded cpt category}
We have a natural equivalence $\D(\RMod_{cpt}(E)) \cong \Pro(\Perf^R_{disc}(E))$.
\end{corollary}

\begin{proof}
Corollary \ref{cor to description of discrete perfect objects} gives an equivalence $\Pro(\Perf_{disc}^R(E)) \cong \Pro((\Perf_{disc}^L(E))^{op})$, and the right hand side here is equivalent to $(\Ind(\Perf^L_{disc}(E)))^{op}$, which is equivalent to $\D(\LMod_{disc}(E))^{op}$ by Proposition \ref{unbounded disc category}. We then have $\D(\LMod_{disc}(E))^{op} = \D(\LMod_{disc}(E)^{op})$ by definition, and the latter is equivalent to $\D(\RMod_{cpt}(E))$ by Pontryagin duality.
\end{proof}

To simplify our notation, we write $\D_{cpt}^R(E)$ for $\D(\RMod_{cpt}(E))$. We can now define our third functor by taking $\Pro$ of $F^\prime : \Perf^R_{disc}(E) \to \D^b_{\Coh,\m}(\mf{X})$ to get
\[
\Pro(\Perf^R_{disc}(E)) \to \ProCoh_{\m}(\mf{X}) := \Pro(\D^b_{\Coh,\m}(\mf{X})).
\]
Applying Corollary \ref{unbounded cpt category}, we get a fully faithful embedding $F_{cpt} : \D_{cpt}^R(E) \to \ProCoh_{\m}(\mf{X})$, as desired. We summarize these results in a theorem.

\begin{theorem}\label{derived embedding}
There are fully faithful exact functors $F : \D^L(E) \to \IndCoh(\mf{X})$, $F_{disc} : \D_{disc}^L(E) \to \IndCoh_{\m}(\mf{X})$ and $F_{cpt} : \D_{cpt}^R(E) \to \ProCoh_{\m}(\mf{X})$ induced by 
$F: \Perf^L(E) \to \D^b_{coh}(\fX)$ in the first two cases and $F': \Perf^R(E) \to \D^b_{coh}(\fX)$ in the third case (these functors are defined in (\ref{eqn: derived F Fprime})).
\end{theorem}

As in the abelian case, we also have adjoint functors.

\begin{proposition}\label{derived adjoints}
The functors $F$, $F_{disc}$ and $F_{cpt}$ from Theorem \ref{derived embedding} have the following adjoints:
\begin{enumerate}
\item $F$ has a right adjoint $G : \IndCoh(\mf{X}) \to \D^L(E)$ given by $G(W) = \RHom(V,W)$. Moreover, $G$ commutes with limits and colimits.

\smallskip

\item $F_{disc}$ has a right adjoint $G_{disc} : \IndCoh_{\m}(\mf{X}) \to \D_{disc}^L(E)$ given by $G_{disc}(W) = \RHom(V,W)$. Moreover, $G_{disc}$ commutes with limits and colimits.

\smallskip

\item $F_{cpt}$ has a left adjoint $G_{cpt} : \ProCoh_{\m}(\mf{X}) \to \D_{cpt}^R(E)$. Moreover, $G_{cpt}$ commutes with limits and colimits.

\end{enumerate}
\end{proposition}

\begin{proof}
We start with (1). The adjunction between $F$ and $G$ can be checked directly; it is a Hom-Tensor adjunction. It is clear that $G$ commutes with limits, and it commutes with colimits since $V$ is compact in $\IndCoh(\mf{X})$ (by definition, since it lies in $\D_{coh}^b(\mf{X})$).

For part (2), it suffices to prove that $G(W)$ has $\m^\infty$-torsion cohomology when $W \in \IndCoh_{\m}(\mf{X})$. Since $G$ commutes with colimits, it suffices to check this for $W \in \D_{coh,\m}^b(\mf{X})$. By induction on the amplitude and shifting, we may assume that $W \in \Coh_{\m}(\mf{X})$. Then it is clear that $\Ext^i(V,W)$ is killed by any power of $\m$ that kills $W$, independent of $i$. This finishes the proof.

Finally, existence of $G_{cpt}$ in part (3) follows from Lurie's adjoint functor theorem \cite[Cor.~5.5.2.9]{lurie-htt} (note that this applies to $\infty$-categories whose opposites are presentable as well), and the fact that it commutes with cofiltered limits is proved in exactly the same way as in Proposition \ref{abelian adjoints}(3). Since it is exact (by \cite[Prop.~1.1.4.1]{lurie-ha}) it then commutes with all limits.
\end{proof}

\begin{remark}\label{rmk on adjoints derived}
As in the abelian case, Lurie's adjoint functor theorem implies that the functors $G$ and $G_{disc}$ have right adjoints and that $G_{cpt}$ has a left adjoint. Perhaps more interestingly, the adjoint pairs $(F,G)$, $(F_{disc},G_{disc})$ and $(G_{cpt},F_{cpt})$ also induce semiorthogonal decompositions on $\IndCoh(X)$, $\IndCoh_{\m}(X)$ and $\ProCoh_{\m}(X)$. Let us spell this out for $(F,G)$, the details for $(F_{disc},G_{disc})$ are identical and the details for $(G_{cpt},F_{cpt})$ are dual. We refer to \cite[\S A.8]{egh} for generalities on semiorthogonal decompositions. Write $\mc{A}$ for the kernel of $G$ (i.e.\ the full subcategory of $\IndCoh(X)$ of objects $W$ satisfying $\RHom(V,W)=0$) and let $\mc{B}$ denote the essential image of $F$. Then $(\mc{B},\mc{A})$ is easily seen to be a semiorthogonal decomposition for $\IndCoh(\mf{X})$ (cf. \cite[Lem.~A.8.4]{egh}).
\end{remark}

In the case when $V$ satisfies the hypotheses in the abelian situation, we now have two a priori different definitions of functors at the level of derived categories: Those given by Theorem \ref{derived embedding} and those given by deriving the functors in Theorem \ref{abelian embedding}. As expected, they agree, in a suitable sense. In our discussion of this (only), we will use $F$, $F_{disc}$ and $F_{cpt}$ to denote the functors from Theorem \ref{abelian embedding}, and $\mc{F}$, $\mc{F}_{disc}$ and $\mc{F}_{cpt}$ to denote the functors from Theorem \ref{derived embedding}. From these, we can form new functors in the following way. First, by composing $\mc{F}$ with the natural functor
\[
\IndCoh(\mf{X}) \to \D_{qcoh}(\mf{X}),
\]
we obtain a functor $\ol{\mc{F}} : \D^L(E) \to \D_{qcoh}(\mf{X}))$. Similarly, we obtain functors $\ol{\mc{F}}_{disc} : \D_{disc}^L(E) \to \D_{qcoh,\m}(\mf{X})$ and $\ol{\mc{F}}_{cpt} : \D_{cpt}^R(E) \to \D(\Pro(\Coh_{\m}(\mf{X})))$. On the other hand, we may derive the functors $F$, $F_{disc}$ and $F_{cpt}$. For this we use the model-theoretic framework, for which we refer to \cite[\S 7.5]{cisinski}: Our functors, as well as their adjoints $G$, $G_{disc}$ and $G_{cpt}$, extend to functors on the abelian categories of chain complexes in the corresponding abelian categories, and these have model structures described by \cite[Prop.~1.3.5.3]{lurie-ha} (or its dual).

\begin{lemma}\label{Quillen adjunction}
The pairs $(F,G)$, $(F_{disc},G_{disc})$ and $(G_{cpt},F_{cpt})$ are Quillen adjunctions. In particular, they induce adjunctions $(LF,RG)$, $(LF_{disc},RG_{disc})$ and $(LG_{cpt},RF_{cpt})$ at the level of derived $\infty$-categories.
\end{lemma}

\begin{proof}
Since $F$ and $F_{disc}$ are exact, they preserve cofibrations and weak equivalences (directly from the definitions of the model relevant structures), and hence $(F,G)$ and $(F_{disc},G_{disc})$ are Quillen adjunctions. Similarly, exactness of $F_{cpt}$ means that it preserves fibrations and weak equivalences, making $(G_{cpt},F_{cpt})$ a Quillen adjunction. The second statement is then \cite[Thm.~7.5.30]{cisinski}.
\end{proof}

We can now formulate and prove the compatibility between our abelian and derived embeddings, in the abstract setting of this subsection.

\begin{proposition}\label{abelian = derived 1}
We have natural equivalences of functors $\ol{\mc{F}} \cong LF$, $\ol{\mc{F}}_{disc} \cong LF_{disc}$ and $\ol{\mc{F}}_{cpt} \cong RF_{cpt}$.
\end{proposition}

\begin{proof}
We give the proof that $\ol{\mc{F}} \cong LF$, the proof that $\ol{\mc{F}}_{disc} \cong LF_{disc}$ is identical and the proof that $\ol{\mc{F}}_{cpt} \cong RF_{cpt}$ is dual. First, we observe that cofibrant replacement is not needed to define $RF$, since $F$ is exact and hence preserves all weak equivalences. In particular, it follows from the defining formulas that $RF$ and $\ol{\mc{F}}$ agree on $\Perf^L(E)$. Since they also commute with colimits, they have to agree on all of $\D^L(E)$ (to see that they commute with colimits, one can e.g.\ use that $LF$ is a left adjoint by Lemma \ref{Quillen adjunction}, and for $\ol{\mc{F}}$ that $\mc{F}$ and the natural map $\IndCoh(\mf{X}) \to \D_{qcoh}(\mf{X})$ commute with colimits).
\end{proof}

In the situations when we wish to apply the abelian construction, we can give a slightly more precise result. For this, we need the following lemma.

\begin{lemma}\label{QCoh = IndCoh}
Assume that $G$ is linearly reductive over $\oo$. Assume moreover that $A$ has finite global dimension. Then every complex in $D_{coh}^b(\mf{X})$ is perfect, and the natural functors $\IndCoh(\mf{X}) \to \D_{qcoh}(\mf{X})$ and $\IndCoh_{\m}(\mf{X}) \to \D_{qcoh,\m}(\mf{X})$ are equivalences.
\end{lemma}

\begin{proof}
First, we show that $\Coh(\mf{X})$ has enough projectives. The first step is to show that $\QCoh([\Spec \oo/G])$ has enough projectives. Since $G$ is linearly reductive, $V \in \QCoh([\Spec \oo/G])$ is projective if and only if its underlying $\oo$-module is projective. If $G$ is diagonalizable with Cartier dual $M$, then $\QCoh([\Spec \oo/G])$ is the category of $M$-graded $\oo$-modules (cf. \cite[\S  I.2.11]{jantzen}), which visibly has enough projectives. This is the only case we will needed in applications, so we only sketch how the general case follows. First, one checks that the claim may be checked after a finite \'etale extension of $\oo$. Thus, by \cite[Lemma 2.20]{aov}, we may assume that $G$ is an extension of a finite constant group scheme of order prime to $p$ by a diagonalizable group scheme. Using induction from the diagonalizable subgroup (which is now a left adjoint as well to the restriction functor) one reduces to the diagonal case. We then deduce that $\Coh([\Spec \oo/G])$ also has enough projectives using \cite[\S 1.2.13]{jantzen}. To show that $\Coh(\mf{X})$ has enough projectives, pick $W \in \Coh(\mf{X})$ and choose (by \cite[\S 1.2.13]{jantzen} again) a $G$-equivariant finitely generated $\oo$-submodule $W^\prime \sub W$ which generates $W$ as an $A$-module. projective $V \in \Coh([\Spec \oo / G])$ with a map. We may then choose a surjection $V \to W^\prime$ from a projective $V \in \Coh([\Spec \oo/G])$, and from this we obtain a surjection $V_A := V\otimes_\oo A \to W$. Note $V_A$ is projective in $\Coh(\mf{X})$ since its underlying $A$-module is projective.

Next, we show that every complex in $D_{coh}^b(\mf{X})$ is perfect.  It suffices to show that every $M \in \Coh(\mf{X})$ is perfect. Since $\Coh(\mf{X})$ has enough projectives, we may find resolutions
\[
0 \to W_{s+1} \to V^s \to \dots \to V^0 \to W \to 0
\]
for all $s\geq 0$, with $V^j \in \Coh(\mf{X})$ projective for all $j$. Since the global dimension of $A$ is finite, $W_{s+1}$ is automatically projective as an $A$-module for large enough $s$, and hence as an object of $\Coh(\mf{X})$. This shows that $W$ is perfect, as desired. 

To show that $\IndCoh(\mf{X}) \cong \D_{qcoh}(\mf{X})$, it suffices to show that objects of $\D_{coh}^b(\mf{X})$ are compact and generate $\D_{qcoh}(\mf{X})$. Since $\Coh(\mf{X})$ has enough projectives, we see (\cite[\S 1.2.13]{jantzen} again) that these projectives generate $\D_{qcoh}(\mf{X})$. Since (in our situation) every perfect complex is quasi-isomorphic to a bounded complex of projective objects, the usual proof for rings (see e.g. \cite[\href{https://stacks.math.columbia.edu/tag/07LQ}{Tag 07LQ}]{stacks-project}) shows that perfect complexes are compact in $\D_{qcoh}(\mf{X})$. Putting this together, we have shown that $\IndCoh(\mf{X}) \to D_{qcoh}(\mf{X})$ is an equivalence. 

Finally, to show that $\IndCoh_{\m}(\mf{X}) \to \D_{qcoh,\m}(\mf{X})$ is an equivalence, it suffices to show that $\D_{coh,\m}^b(\mf{X})$ generates $\D_{qcoh,\m}(\mf{X})$. In fact, the objects $V_A/\m^r V_A \in \D_{coh,\m}^b(\mf{X})$ for $r\in \Z_{\geq 1}$ and $V\in \Coh([\Spec \oo /G])$ generate $\D_{qcoh,\m}(\mf{X})$; this follows by essentially the same argument as in the proof of Proposition \ref{unbounded disc category}.
\end{proof}

\begin{remark}\label{rmk: IndCoh not eq to QCoh}
Assume that $A$ has a maximal ideal fixed by $G$, so that $G$ occurs as the stabilizer of a point of $X$. Then the assumption that $G$ is linearly reductive is essential for compact generation of $\D_{qcoh}(\mf{X})$ (see Remark \ref{issues with qc sheaves}) and hence for $\IndCoh(\mf{X}) \to \D_{qcoh}(\mf{X})$ (or even for $\Ind \Perf(\mf{X}) \to \D_{qcoh}(\mf{X})$) to have a chance of being an equivalence). In particular, $\IndCoh(\mf{X}) \to \D_{qcoh}(\mf{X})$ is not an equivalence in the ``non-generic case I'' situation considered in \S \ref{subsec: non-generic I}, even though the $A$ there has finite global dimension. 
\end{remark}

\begin{corollary}\label{abelian = derived 2}
Assume that $G$ is linearly reductive and that $A$ has finite global dimension. Then we have natural equivalences $\mc{F} \cong LF$ and $\mc{F}_{disc} \cong LF_{disc}$. Moreover, the $RG$ and $RG_{disc}$ are naturally equivalent to the adjoints $G$ and $G_{disc}$ from Proposition \ref{derived adjoints}.
\end{corollary}

\begin{proof}
The first statement follows directly from Lemma \ref{QCoh = IndCoh} and Proposition \ref{abelian = derived 1}. The final statements then follow from Lemma \ref{Quillen adjunction} by uniqueness of adjoints.
\end{proof}

\section{Geometric interpretation of $p$-adic Local Langlands for $\GL_2(\Qp)$}\label{chap: geometric interpretation}
\subsection{The $p$-adic Local Langlands correspondence for $\GL_2(\Qp)$ as an embedding of categories}\label{subsec: main thm}

In this section we will apply the material of \S \ref{sec: categorical constructions} to give our interpretation of $p$-adic local Langlands as an embedding of ($\infty$-)categories. We will freely use the notation used there for categories of modules and sheaves (but our notation for rings, groups and stacks might be slightly different), as well as the notation for blocks etc. from \S \ref{sec: blocks for GL2}. Recall that we have fixed a determinant $\psi$ and the corresponding central character $\zeta$. For each block $\mf{B}$ with corresponding semisimple Galois representation $\rho_{\mf{B}}$ with pseudorepresentation $D_{\mf{B}}$, we set $\fX_{\mf{B}} := \Rep^{\psi}_{D_{\mf{B}}}$. Set $G=\GL_2(\Qp)$ and recall from \cite[Prop.~5.34]{paskunas-image} that the category $\Mod_{G,\zeta}^{lfin}(\oo)$ has a decomposition
\[
\Mod_{G,\zeta}^{lfin}(\oo) = \prod_{\mf{B}} \Mod_{G,\zeta}^{lfin}(\oo)_{\mf{B}}
\]
according to blocks. On the dual side, we get a decomposition 
\[
\mf{C}(\oo) = \prod_{\mf{B}} \mf{C}(\oo)_{\mf{B}}.
\]
Recall that if $\wt{P}_{\mf{B}}$ is a projective envelope of $\pi_{\mf{B}}^\vee$ in $\mf{C}(\oo)_{\mf{B}}$, then we have an equivalence $\mf{C}(\oo)_{\mf{B}} \cong \RMod_{cpt}(\wt{E}_{\mf{B}})$, where $\wt{E}_{\mf{B}} := \End(\wt{P}_{\mf{B}})$. Since $\Mod_{G,\zeta}^{lfin}(\oo)_{\mf{B}}$ is equivalent to $\mf{C}(\oo)_{\mf{B}}^{op}$ (via Pontryagin duality), $\Mod_{G,\zeta}^{lfin}(\oo)_{\mf{B}}$ is equivalent to $\LMod_{disc}(\wt{E}_{\mf{B}})$. We let $R_{\mf{B}}$ denote the center of $\wt{E}_{\mf{B}}$ and recall that, by \cite[Thm.~1.5]{paskunas-image}, $R_{\mf{B}}$ is naturally isomorphic to the universal deformation ring of the pseudorepresentation $D_{\mf{B}}$. Let $\m\sub R_{\mf{B}}$ denote the maximal ideal of $R_{\mf{B}}$. We now formulate the main results of this section, which are our main results on $p$-adic local Langlands for $\GL_2(\Qp)$. We start with a general result, applying to all blocks (although we recall our running assumption that $p \ge 5$). 

\begin{theorem}\label{main derived geometrization GL2}
For each block $\mf{B}$, there are exact fully faithful embeddings $
F_{disc} : \D(\Mod_{G,\zeta}^{lfin}(\oo)_{\mf{B}}) \to \IndCoh_{\m}(\fX_{\mf{B}}) $ and $F_{cpt} : \D(\mf{C}(\oo)_{\mf{B}}) \to \ProCoh_{\m}(\fX_{\mf{B}}) $ of stable $\infty$-categories. They satisfy the following properties:
\begin{enumerate}
\item $F_{disc}$ commutes with colimits and preserves compact objects. It has a right adjoint $G_{disc}$ which commutes with colimits. It is induced by the functor \begin{align*}F_{disc}: \Perf^L(\widetilde{E}_{\mathfrak{B}}) &\to \mathcal{D}^b_{coh}(\mathfrak{X}_{\mathfrak{B}})\\
P_\bullet &\mapsto X_{\mf{B}}^*\otimes_{\widetilde{E}_{\mathfrak{B}}} P_\bullet \end{align*} for a coherent sheaf $X_{\mf{B}} \in MCM(\frak{X}_{\frak{B}})$ equipped with an isomorphism $\widetilde{E}_{\mathfrak{B}} \cong \End(X_{\mf{B}})$.

\smallskip

\item $F_{cpt}$ commutes with limits and preserves cocompact objects. It has a left adjoint $G_{cpt}$ which commutes with limits. It is induced by the functor \begin{align*}F_{cpt}: \Perf^R(\widetilde{E}_{\mathfrak{B}}) &\to \mathcal{D}^b_{coh}(\mathfrak{X}_{\mathfrak{B}})\\
	P_\bullet &\mapsto P_\bullet\otimes_{\widetilde{E}_{\mathfrak{B}}} X_{\mf{B}}. \end{align*}
\end{enumerate} 
\end{theorem}

When the block $\mf{B}$ is supersingular or reducible generic, we get embeddings at the level of abelian categories:

\begin{theorem}\label{main abelian geometrization GL2}
Assume that $\mf{B}$ is supersingular or reducible generic. Then there are exact fully faithful embeddings $
F_{disc} : \Mod_{G,\zeta}^{lfin}(\oo)_{\mf{B}} \to \QCoh_{\m}(\fX_{\mf{B}}) $ and $F_{cpt} : \mf{C}(\oo)_{\mf{B}} \to \Pro(\Coh_{\m}(\fX_{\mf{B}})) $ of abelian categories. They satisfy the following properties:
\begin{enumerate}
\item $F_{disc}$ commutes with colimits and preserves compact objects. It has a right adjoint $G_{disc}$ which commutes with colimits.

\smallskip

\item $F_{cpt}$ commutes with limits and preserves cocompact objects. It has a left adjoint $G_{cpt}$ which commutes with cofiltered limits.
\end{enumerate}
Moreover, the derived functor of $F_{disc}$ agrees with the functor $F_{disc}$ from Theorem \ref{main derived geometrization GL2}, and the derived functor of $F_{cpt}$ agrees with the functor $F_{cpt}$ from Theorem \ref{main derived geometrization GL2}, after composing the latter with the canonical functor $\ProCoh_{\m}(X) \to \D(\Pro(\Coh_{\m}(X)))$.
\end{theorem}

\begin{remark}
These functors will be constructed by applying the material from \S \ref{sec: categorical constructions}, i.e.\ by constructing suitable objects $X_{\mf{B}} \in \MCM(\mf{X}_{\mf{B}})$ satisfying the conditions given there and then applying Theorems \ref{abelian embedding} and \ref{derived embedding}. In particular, we make no claims about our functors being `canonical' (whatever the reader might read into this word), or unique. Indeed, given an $X_{\mf{B}}$, any twist of this $X_{\mf{B}}$ by a line bundle will have the same properties. We do remark, however, that the objects $X_{\mf{B}}$ that we will present seem rather natural; they are closely related to the vector bundle underlying the universal representation on $\mf{X}_{\mf{B}}$. To us, this seems unlikely to be a coincidence. On the other hand, it is not the case that the $X_{\mf{B}}$ are uniform in $\mf{B}$ either (but see Remark \ref{kernel of functor}). We note that our $X_\mf{B}$, at least for supersingular and generic principal series blocks, occur in the description of the functor of \cite{dotto-emerton-gee}; see \cite[Thm.~7.3.5]{egh}. 
\end{remark}

We now start the proof of Theorems \ref{main derived geometrization GL2} and \ref{main abelian geometrization GL2} by pointing out the general steps; we will then finish the proof on a block by block basis. The strategy is to construct an object $X_{\mf{B}} \in \MCM(\mf{X}_{\mf{B}})$ which satisfies Assumption \ref{assumptions derived} (and the stronger Assumption \ref{assumptions abelian} when $\mf{B}$ is supersingular or generic principal series). Using the coherent dual $X_{\mf{B}}^\ast$, we then obtain a fully faithful embedding
\[
F_{disc} : \D(\Mod^{lfin}_{G,\zeta}(\oo)_{\mf{B}}) \cong \D_{disc}^L(\wt{E}_{\mf{B}}) \to \IndCoh_{\m}(\mf{X}_{\mf{B}})
\]
from Theorem \ref{derived embedding} (and the abelian version from Theorem \ref{abelian embedding} when $\mf{B}$ is supersingular or generic principal series). On the other hand, using the object $X_{\mf{B}}$, we obtain a fully faithful embedding
\[
F_{cpt} : \D(\mf{C}(\oo)_{\mf{B}}) \cong \D_{cpt}(\wt{E}_{\mf{B}}) \to \ProCoh_{\m}(X)
\]
from Theorem \ref{derived embedding} (and the abelian version from Theorem \ref{abelian embedding} when $\mf{B}$ is supersingular or generic principal series) again. The statements about adjoints are then provided by Propositions \ref{abelian adjoints} and \ref{derived adjoints}, and the statement in Theorem \ref{main abelian geometrization GL2} about compatibility between the abelian and derived functors follows from Proposition \ref{abelian = derived 1} for $F_{cpt}$, and the stronger compatibility for $F_{disc}$ follows from Corollary \ref{abelian = derived 2} (we will verify the assumptions needed for this statement in our discussion of the blocks).

It therefore remains to describe $X_{\mf{B}}$. In essence, this has already been done in \S \ref{chap: stacks for GL2Qp}, so all we have to do is to collect the results. As indicated above, we will do this block by block. We will put added emphasis on the functor $F_{disc}$, since its formulation is closest to the formulation of categorical $p$-adic local Langlands conjecture from \cite{egh}. In particular, we will also use the calculations from \S \ref{chap: stacks for GL2Qp} to describe where the irreducible objects in $\Mod^{lfin}_{G,\zeta}(\oo)_{\mf{B}}$ go under $F_{disc}$. For simplicity, we will mostly drop the subscript $-_{\mf{B}}$ from the notation since it is fixed at the start of the discussion of each block.

\begin{remark}
	In what follows, we compute $F_{disc}(\pi)$ for irreducible representations $\pi$. We can also consider $F_{cpt}(\pi^\vee)$. Our computations suggest that we have $F_{disc}(\pi) = F_{cpt}((\mathcal{S}\pi)^\vee)$, where $\mathcal{S}$ is a shift of the (derived) smooth dual introduced by Kohlhaase \cite{kohlhaase} (this is an easy check in the supersingular and generic principal series cases). 
	
	We note that this is different to the compatibility with duality functors in \cite[Conj. 6.1.14]{egh} --- the difference between $F_{cpt}$ and $F_{disc}$ comes from Pontryagin duality and coherent duality, whilst the duality in \emph{loc.~cit.}\ involves coherent duality and the `dual Galois representation' involution on the Galois stack.
\end{remark}

\subsection{Supersingular blocks}\label{subsubsec: geom ss} In the supersingular case, we recall from \S \ref{subsec: supersingular case} that $\mf{X} \cong [ \Spec R / \mu_2 ]$ , with $R \cong \oo \lb X_1, X_2, X_3 \rb$ and $\mu_2$ acting trivially on $R$. In particular, $R$ is a regular local ring (and so has finite global dimension) and $\QCoh(\mf{X})$ is the category of $\Z/2$-graded $R$-modules.  On the $\GL_2(\Qp)$-side, we have $\wt{E} \cong R$ by \cite[Prop.~6.2]{paskunas-image}. It is the clear that $\wt{E}=R$ satisfies Assumption \ref{assumptions on E}. We then see that there are two obvious candidates for $X$: $L_0$ or $L_1$, where $L_n$ denotes the $R$-module $R$, with grading concentrated in degree $n$. We note that these are both self-dual and projective in $\QCoh(\mf{X})$, and flat as $R$-modules (and so satisfy Assumption \ref{assumptions abelian}). For concreteness, we pick $X=L_1$ (one motivation for this choice is Theorem \ref{local-global formula}). Then we get functors
\[
F_{disc} : \Mod_{G,\zeta}^{lfin}(\oo)_{\mf{B}} \to \QCoh_{\m}(\mf{X}), \,\,\,\,\,\, F_{cpt} : \mf{C}(\oo)_{\mf{B}} \to \Pro(\Coh_{\m}(\mf{X})).
\]
The functor $F_{disc}$ identifies the source with the summand of $\Z/2$-graded modules concentrated in degree $1$ of the target. In particular, $F_{disc}$ sends the (unique) supersingular representation $\pi$ in $\mf{B}$ to the skyscraper sheaf $L_1 \otimes_R R/\mf{m}$ on $\mf{X}$ (i.e.\ $R/\mf{m}$ but concentrated in degree $1$).

\subsection{Generic principal series blocks}\label{subsubsec: geom gps} In this case, we recall from \S \ref{subsec: stacks generic ps GL2} that $\mf{X}$ has a presentation $[\Spec S /\G_m]$, where $S \cong \oo \lb a_0, a_1,bc \rb [b,c]$ with $a_0$ and $a_1$ in degree $0$, $b$ in degree $2$ and $c$ in degree $-2$. The pseudodeformation ring $R$ is the subring $\oo \lb a_0, a_1,bc \rb$ of degree zero elements. We also know that the Cayley--Hamilton algebra $E$ is
\[
E = \begin{pmatrix} R & Rb \\ Rc & R  \end{pmatrix},
\]
and it is equal to $\End(\mc{V})$, where $\mc{V}$ is the vector bundle underlying the universal Galois representation. Moreover, $\mc{V}$ is a projective object in $\QCoh(\mf{X})$ (all this is Theorem \ref{Ext groups of universal rep generic ps}). Note also that $\mc{V}$ is self-dual by Proposition \ref{selfduality of universal vb}. We prove the last few things we need about these objects.

\begin{proposition}\label{aux results for generic}
$S$ has finite global dimension, $\mc{V}$ is a projective left $E$-module and $\mc{V}^\ast$ is a projective right $E$-module.
\end{proposition}

\begin{proof}
We have an isomorphism $S \cong R[x,y]/(xy-bc)$, where we note that $bc$ is a prime element in the regular local ring $R$ (which has Krull dimension $4$). It then follows easily that $S$ is regular of dimension $5$, and hence has global dimension $5$. 
We now prove projectivity of $\mc{V}$; the proof of projectivity of $\mc{V}^\ast$ is entirely analogous (using row vectors and right actions). Note that the underlying $S$-module of $\mc{V}$ is simply $S^2$, and that $E$ acts via the embedding $E\sub M_2(S)$, with the usual left action of $M_2(S)$ on $S^2$. The decomposition of $\mc{V}$ into graded pieces is then
\[
\mc{V} = \left( \bigoplus_{n=0}^\infty \begin{pmatrix} c^n R \\ c^{n+1}R \end{pmatrix} \right) \oplus \left( \bigoplus_{n=0}^\infty \begin{pmatrix} b^{n+1} R \\ b^n R \end{pmatrix} \right),
\]
which is a left $E$-module decomposition, so projectivity of $\mc{V}$ is equivalent to projectivity of all of these summands. For this, note that 
\[
\begin{pmatrix} c^n R \\ c^{n+1}R \end{pmatrix} \cong \begin{pmatrix} R \\ cR \end{pmatrix} \,\,\,\,\, \text{and} \,\,\,\,\, \begin{pmatrix} b^{n+1} R \\ b^n R \end{pmatrix} \cong \begin{pmatrix} bR \\ R \end{pmatrix}
\]
as left $E$-modules and that the right hand sides of these isomorphisms are direct summands of $E$ itself. This finishes the proof.
\end{proof}

Let us now compare this with the $\GL_2(\Qp)$-side. The block $\mf{B}$ consists of two irreducible representations $\pi_1 = \Ind_B^G(\delta_1 \otimes \delta_2 \omega^{-1})$ and $\pi_2 = \Ind_B^G(\delta_2 \otimes \delta_1 \omega^{-1})$. Let $P_i$ be the projective envelope of $\pi_i^\vee$, for $i=1,2$. We we will use the decomposition $P_{\mf{B}} = P_2 \oplus P_1$, to match with the local-global considerations in \S \ref{chap: local-global}. By \cite[Cor.~8.7, Lem.~8.10 and Prop.~B.26]{paskunas-image}, we have
\[
\wt{E} \cong \begin{pmatrix} R & R \Phi_{21} \\ R\Phi_{12} & R  \end{pmatrix}
\]
with $\Phi_{12} \circ \Phi_{21} = \wt{c}$ and $\Phi_{21} \circ \Phi_{12} = \wt{c}$, where $\wt{c}$ is a generator of the reducibility ideal in $R$ ($\wt{c}$ is called $c$ in \cite{paskunas-image}) and $\Phi_{ij} \in \Hom(P_j,P_i)$ is a generator. Looking at this, we set $X=\mc{V}$. Since $bc$ also generates the reducibility ideal (by Theorem \ref{computation of rings generic ps}), we see that $\wt{E} \cong E$ as desired, matching up the matrix entries (in particular, $P_2$ will correspond to $L_1$ and $P_1$ corresponds to $L_{-1}$). Let us now verify that $E$ satisfies Assumption \ref{assumptions on E}:

\begin{proposition}\label{conditions on E generic}
$E$ (or equivalently $\wt{E}$) satisfies satisfies Assumption \ref{assumptions on E}.
\end{proposition}

\begin{proof}
From the description above, one sees that $R$ is the center of $E$ (see also \cite[Cor.~8.11]{paskunas-image}), and that $E$ is finitely generated over $R$. Finally, we need to verify that every simple right $E$-module has finite projective dimension. By the equivalence $\RMod_{cpt}(E)^{op} \cong \Mod_{G,\zeta}^{lfin}(\oo)_{\mf{B}}$, this translates into showing that every irreducible $\pi \in \mf{B}$ has finite injective dimension. By \cite[Rem.~10.11]{paskunas-image}, this is equivalent to $\Ext_{G,\zeta}^i(\pi,\pi^\prime)$ vanishing for all $\pi^\prime \in \mf{B}$ and all sufficiently large $i$. Since all members of $\mf{B}$ are induced, this follows from the general fact that $\Ext_{G,\zeta}^i(\Ind_B^G U,V) = 0$ for all $i\geq 4$, all representations $U$ of the diagonal torus $T$ and $V$ of $G$ (both with central character $\zeta$); this is contained in the discussion in \cite[\S 7.1]{paskunas-image}, preceding Proposition 7.1 of \emph{loc.\ cit.}
\end{proof}

This then gives us our functors
\[
F_{disc} : \Mod_{G,\zeta}^{lfin}(\oo)_{\mf{B}} \to \QCoh_{\m}(\mf{X}), \,\,\,\,\,\, F_{cpt} : \mf{C}(\oo)_{\mf{B}} \to \Pro(\Coh_{\m}(\mf{X})).
\]
Here, we note that the target category is much larger than the source: The essential image is anything that can be built from $\mc{V} = L_1 \oplus L_{-1}$, whereas one needs all the $L_n$, $n\in \Z$, to generate the whole of $\QCoh(\mf{X})$. We finish our discussion of this case by computing $F_{disc}(\pi_i)$ for $i=1,2$. Note that the definition of $F_{disc}$ uses $\mc{V}^\ast = L_{-1} \oplus L_1$, viewed as row vectors acted on from the right by $E$. 

\begin{proposition}
We have (canonical) isomorphisms $F_{disc}(\pi_1) = L_1/(\vp,a_0,a_1,b)$ and $F_{disc}(\pi_2) = L_{-1}/(\vp,a_0,a_1,c)$.
\end{proposition}

\begin{proof}
The two simple right $E$-modules $\sigma_i^\prime$, $i=1,2$, are both isomorphic to $k = R/\m$ as $R$-modules, with action given by
\[
u \cdot \begin{pmatrix} x_2 & y_2 b \\ y_1 b & x_1  \end{pmatrix} = x_i u,
\]
for $u\in k$ and $x_1,x_2,y_1,y_2\in R$. Letting $\sigma_i$ be the Pontryagin dual of $\sigma_i^\prime$, we see that $\sigma_i$ is the left $E$-module with action given by
\[
\begin{pmatrix} x_1 & y_1 b \\ y_2 c & x_2  \end{pmatrix} \cdot u = x_i u.
\]
This action defines a surjection $E \to \sigma_i$ and its kernel is the (in fact two-sided) ideal $I_i = \{A \in E \mid x_i \in \m \}$, where $A = \left( \begin{smallmatrix} x_2 & y_2 b \\ y_1 c & x_1 \end{smallmatrix} \right)$. By the definition of $F_{disc}$ we then have 
\[
F_{disc}(\pi_i) =  \mc{V}^\ast \otimes_{E}  (E/I_i)= \mc{V}^\ast/I_i \mc{V}^\ast,
\]
so it remains to make the right hand side explicit. Consider the decomposition
\[
\mc{V}^\ast = \left( \bigoplus_{n=0}^\infty \begin{pmatrix} c^{n+1} R & c^n R \end{pmatrix} \right) \oplus \left( \bigoplus_{n=0}^\infty \begin{pmatrix} b^n R & b^{n+1} R \end{pmatrix} \right)
\]
into its graded pieces, which are right $E$-modules. Note that 
$\mc{V}^\ast_{2n+1} = (\,b^n R \,\,\, b^{n+1}R\,)$ and $\mc{V}^\ast_{1-2n} = (\,c^{n+1}R \,\,\, c^n R \, )$ for $\n\in \Z_{\geq 0}$. By direct computation we observe that 
\[
\begin{pmatrix} c^{n+1} R & c^n R \end{pmatrix} I_1  = \begin{pmatrix} c^{n+1} R & c^n \m \end{pmatrix},\,\,\,\,  \begin{pmatrix} c^{n+1} R & c^n R \end{pmatrix} I_2 = \begin{pmatrix} c^{n+1} R & c^n R \end{pmatrix},
\]
\[
\begin{pmatrix} b^n R & b^{n+1} R \end{pmatrix} I_1 = \begin{pmatrix} b^n R & b^{n+1} R \end{pmatrix} \,\,\, \text{and} \,\,\, \begin{pmatrix} b^n R & b^{n+1} R \end{pmatrix} I_2 = \begin{pmatrix} b^n \m & b^{n+1} R \end{pmatrix}.
\]
From this, we deduce that 
\[
\mc{V}^\ast/I_1\mc{V}^\ast \cong \bigoplus_{n=0}^\infty \begin{pmatrix} 0 & c^n k \end{pmatrix} \,\,\,\,\, \text{and} \,\,\,\,\, \mc{V}^\ast/I_2\mc{V}^\ast \cong \bigoplus_{n=0}^\infty \begin{pmatrix} b^n k & 0 \end{pmatrix}
\]
and hence that we have isomorphisms $F_{disc}(\pi_1) = L_1/(\vp,a_0,a_1,b)$ and $F_{disc}(\pi_2) = L_{-1}/(\vp,a_0,a_1,c)$, as desired.
\end{proof}

Note in particular that these are not skyscraper sheaves. As $S$-modules their supports are $1$-dimensional, their union being the two lines that make up $S/\m S = k[b,c]/(bc)$. 

\subsection{Non-generic case I}\label{subsubsec: geom ng1} In this case, the block consists of a single irreducible representation of the form $\pi = \Ind_B^G(\delta \otimes \delta \omega^{-1})$. The ring $\wt{E}$, while not as explicit as for previous blocks, is studied in detail by Pa{\v s}k{\=u}nas \cite[\S 9]{paskunas-image}. By (the proof of) \cite[Cor.~9.33]{paskunas-image}, $\wt{E}$ is isomorphic to the Cayley--Hamilton algebra $E$. On the other hand, Proposition \ref{vanishing of exts non-generic I} says that $\Ext^i(\mc{V},\mc{V})=0$ for $i\geq 1$, and Theorem \ref{endomorphisms in non-generic 1} says that $E= \End(\mc{V})$, so Assumption \ref{assumptions derived} is satisfied. Therefore, we may set $X=\mc{V}$ for this block as well. Recall that $\mc{V}^\ast \cong \mc{V}$, so this gives us our functors
\[
F_{disc} : \D(\Mod_{G,\zeta}^{lfin}(\oo)_{\mf{B}}) \to \IndCoh_{\m}(\mf{X}), \,\,\,\,\,\, F_{cpt} : \D(\mf{C}(\oo)_{\mf{B}}) \to \ProCoh_{\m}(\mf{X}),
\]
now directly at the derived level\footnote{We cannot apply the abelian construction since we do not know if $\mc{V}^\ast$ is projective, or if it is flat as a right $E$-module. However, we will still get a result at the level of abelian categories; see Proposition \ref{nongen 1 is abelian!}.},
 if we can verify that $E$ satisfies Assumption \ref{assumptions on E}. We now verify this.

\begin{proposition}
$E$ (or equivalently $\wt{E}$) satisfies Assumption \ref{assumptions on E}.
\end{proposition}

\begin{proof}
First, $R$ is the center of $E$ by \cite[Cors.~9.13, 9.24  and 9.27]{paskunas-image}, and $E$ is finitely generated over $R$ by \cite[Cor.~9.25]{paskunas-image}. Finally, that every simple right $E$-module has finite projective dimension follows exactly as in the proof of Proposition \ref{conditions on E generic} since $\pi$ is induced.
\end{proof}

We now compute $F_{disc}(\pi)$. We recall from \S\ref{subsec: non-generic I} that $\gamma,\delta$ are pro-generators for the maximal pro-$p$ quotient $\mathcal{G}$ of $\Gamma$, and $2(1+t_1), 2(1+t_2), 2(1+t_3)$ are the traces of $\gamma$, $\delta$ and $\gamma\delta$ respectively under the universal pseudorepresentation $\mathcal{G} \to R$.  

\begin{proposition}\label{computation of irreducible nongen 1}
We have $F_{disc}(\pi) = \mathcal{V}^\ast/(\varpi,\im(u^\ast),\im(v^\ast))[0]$, where $u, v \in \End(\mathcal{V})$ are as in \S\ref{subsec: non-generic I}, and $u^\ast,v^\ast$ are the dual endomorphisms of $\mathcal{V}^\ast$. 

The (scheme-theoretic) support of $F_{disc}(\pi)$ on $\mathfrak{X}\otimes_{\OO}\OO/\varpi$ is cut out by the equations  $(\gamma-1)(\delta-1) = (\delta-1)(\gamma-1) = 0$ and $t_i = 0$ for $i = 1,2,3$. 
\end{proposition}

\begin{proof}
We use the resolution (\ref{eq:ng1-resolution}) to compute $F_{disc}(\pi)$. Indeed, we have a perfect complex of left $E$-modules
\[P_{\OO}
 = \left[ E \xrightarrow{\left(\begin{smallmatrix}
			v & u
		\end{smallmatrix}\right)} E^{\oplus 2} \xrightarrow{\left(\begin{smallmatrix}
			vu & -u^2\\ -v^2 & uv
		\end{smallmatrix}\right)} E^{\oplus 2} \xrightarrow{\left(\begin{smallmatrix}
			u \\ v
		\end{smallmatrix}\right)}  E \right] \]
	such that $\pi$ corresponds to the mapping cone of $P_{\OO} \xrightarrow{\times\varpi} P_{\OO}$ in $\Perf^L({E})$.

Now we must understand the complex $\mathcal{V}^\ast \otimes_E P_{\OO}$ in $\Coh(\mf{X})$. We do this after pulling back by the map $\pi: \Spec A \to \mathfrak{X}$, where $\Spec A$ represents $\Rep(E)^\square$. We have $\pi^*\mathcal(\mc{V}^\ast) = A^2$ and we can write explicit matrices for each map in the complex $\mathcal{C}_{\OO}:=\pi^*(\mathcal{V}^\ast \otimes_E P_{\OO})$. To prove the Proposition, it suffices to show that $\mathcal{C}_{\OO}$ is acyclic and that $H_0(\mathcal{C}_{\OO})$ is $\varpi$-torsion free with support in $\Spec A$ cut out by the equations
\[
(\gamma-1)(\delta-1) = (\delta-1)(\gamma-1) = 0, \qquad t_i = 0 \quad \text{ for }i = 1,2,3.
\]

In \S\ref{subsec: non-generic I}, we described an explicit presentation of $A$ as an $R$-algebra. In fact, the complex $\mathcal{C}_{\OO}$ descends to a perfect complex of $A_F$-modules, and we can even replace the coefficient ring $\OO$ with $\Z$. At this point we have a finite type $\Z$-algebra $A_{F,\Z}$ and a perfect complex $\mathcal{C}_{F,\Z}$ of $A_{F,\Z}$-modules with $\mathcal{C}_{F,\Z}\otimes_{\Z[t_1,t_2,t_3]}R = \mathcal{C}_{\OO}$. We used Macaulay2 \cite{M2} to check that $\mathcal{C}_{F,\Z}$ has $H_i(\mathcal{C}_{F,\Z}) = 0$ for $i \neq 0$. The annihilator of $H_0(\mathcal{C}_{F,\Z})$ is given by the equations  $(\gamma-1)(\delta-1) = (\delta-1)(\gamma-1) = 0$ and $t_i = 0$ for $i = 1,2,3$, once we invert $2$ and $t_i + 2$ for $i = 1,2,3$ (since $p \ne 2$, these elements are invertible in $R$). The Macaulay2 commands for these verifications can be found at \url{https://github.com/jjmnewton/p-adic-LLC}.

It remains to show that $M = H_0(\mathcal{C}_{\OO})$ is $\varpi$-torsion free. It suffices to show that $\varpi$ is not contained in an associated prime of $M$. For a contradiction, suppose $\varpi \in \pf \in \mathrm{Ass}(M)$. Let $\m$ be a maximal ideal of $A$ containing $\pf$. We necessarily have $\m_R \subset \m$. We note that $A$ is locally complete intersection of relative dimension 6 over $\OO$ (hence Cohen--Macaulay). This follows from \cite[\S3]{BIP2023}, and it can also be deduced directly from the presentation for $A$ in \S\ref{subsec: non-generic I}. Using Auslander--Buchsbaum and the projective resolution $\mathcal{C}_{\OO}$ for $M$, we have $\mathrm{depth}(M_\m) = 4 \le \dim(A/\pf)$. On the other hand, $A/\pf$ is a quotient of the ring $A/\mathrm{Ann}(M)$ which is finite over $\F[c_1,c_2,d_1,d_2]/(c_1d_2-c_2d_1)$. Hence $\dim(A/\pf) \le 3$, a contradiction. \end{proof}

From this, we can actually deduce that $F_{disc}(\pi)$ induces a fully faithful embedding of \emph{abelian categories}. For this, we need some quick recollections on the natural t-structure on $\IndCoh_{\m}(\mf{X})$. We refer to \cite[\S A.6]{egh} and the references given there for more details (see also \cite[\S 1.2]{gaitsgory-indcoh}). The inclusion of $\D_{\Coh,\m}^b(\mf{X})$ into $\D_{qcoh}(\mf{X})$ endows $\D_{coh,\m}^b(\mf{X})$ with its natural t-structure, and this extends to a t-structure on  $\IndCoh_{\m}(\mf{X})$ characterized by the properties that the truncation functors on $\IndCoh_{\m}(\mf{X})$ extend those of $\D_{coh,\m}^b(\mf{X})$ and commute with filtered colimits. The natural map $\IndCoh_{\m}(\mf{X}) \to \D_{qcoh,\m}(\mf{X})$ is then t-exact and induces an equivalence of hearts, so the heart of the natural t-structure on $\IndCoh_{\m}(\mf{X})$ is $\QCoh_{\m}(\mf{X})$.

\begin{proposition}\label{nongen 1 is abelian!}
If $V \in \Mod_{G,\zeta}^{lfin}(\oo)_{\mf{B}}$ then $F_{disc}(V)\in \IndCoh_{\m}(\mf{X})$ is concentrated in degree $0$, hence lies in $\QCoh_{\m}(\mf{X})$. In particular, $F_{disc}$ is t-exact and the resulting functor $\Mod_{G,\zeta}^{lfin}(\oo)_{\mf{B}} \to \QCoh_{\m}(\mf{X})$ is an exact fully faithful embedding of abelian categories.
\end{proposition}

\begin{proof}
The second part is a straightforward consequence of the first. For the first statement, recall that $\pi$ is the unique irreducible object in the block $\Mod_{G,\zeta}^{lfin}(\oo)_{\mf{B}}$ and that $F_{disc}(\pi)$ is concentrated in degree $0$ by Proposition \ref{computation of irreducible nongen 1}. By devissage and exactness of $F_{disc}$ (in the triangulated sense), $F_{disc}(V)$ is concentrated in degree $0$ for any finite length representation $V\in \Mod_{G,\zeta}^{lfin}(\oo)_{\mf{B}}$. Finally, any $V\in \Mod_{G,\zeta}^{lfin}(\oo)_{\mf{B}}$ is a filtered colimit of finite length objects, so $F_{disc}(V)$ is concentrated in degree $0$ since $F_{disc}$ and the truncation functors on both sides commutes with filtered colimits (for $\D(\Mod_{G,\zeta}^{lfin}(\oo)_{\mf{B}})$, see \cite[Prop.~1.3.5.21, Rem.~1.3.5.23]{lurie-ha}).
\end{proof}

\begin{remark}\label{rmk: abelian vs derived for ng1}
In this remark, let us write $F_{disc}^{ab}$ for the exact embedding $\Mod_{G,\zeta}^{lfin}(\oo)_{\mf{B}} \to \QCoh_{\m}(\mf{X})$ given by Proposition \ref{nongen 1 is abelian!}. Deriving $F_{disc}^{ab}$ produces a functor 
\[
LF_{disc}^{ab} : \D(\Mod_{G,\zeta}^{lfin}(\oo)_{\mf{B}}) \to \D_{qcoh,\m}(\mf{X}),
\]
which is not equal to $F_{disc}$ for the simple reason that their codomains differ (by Remark \ref{rmk: IndCoh not eq to QCoh}). While the difference does matter (for example, $LF_{disc}^{ab}$ is probably not an embedding), it is also not that big. Indeed, $LF_{disc}^{ab}$ is the composition of $F_{disc}$ with the natural map $\IndCoh_{\m}(\mf{X}) \to \D_{qcoh,\m}(\mf{X})$. Moreover, $F_{disc}$ can be reconstructed from $LF_{disc}^{ab}$ by first restricting the domain to $\D^b(\Mod_{G,\zeta}^{fin}(\oo)_{\mf{B}})$ and the codomain to $\D_{coh,\m}^b(\mf{X})$, and then taking the ind-completion. Let us emphasize, however, that we do not know how to construct $F_{disc}^{ab}$ directly (i.e.\ without constructing $F_{disc}$ and showing that it is t-exact).  
\end{remark}

\subsection{Non-generic case II}\label{subsubsec: geom ng2} Up to twist, the block is given by $\mf{B}= \{ \mbf{1}, \mathrm{St},\Ind_B^G(\omega \otimes \omega^{-1})\}$. The ring $\wt{E}$ is described in \cite[\S 10]{paskunas-image}; we now recall this in detail. To make the comparison easier, we will try to follow Pa{\v s}k{\=u}nas' notation for the representation theoretic objects (we will continue to write $R$ for the pseudodeformation ring; Pa{\v s}k{\=u}nas writes $R^\psi$). Pa{\v s}k{\=u}nas denotes $\mbf{1}$, $\mathrm{St}$ and $\Ind_B^G(\omega \otimes \omega^{-1})$ by $\mbf{1}_G$, $\mathrm{Sp}$ and $\pi_\alpha$, respectively; their Pontryagin duals are denoted by $\mbf{1}_G^\vee$, $\mathrm{Sp}^\vee$ and $\pi_\alpha^\vee$. Let $\wt{P}_{\mbf{1}_G^\vee}$, $\wt{P}_{\mathrm{Sp}^\vee}$ and $\wt{P}_{\pi_\alpha^\vee}$ be projective envelopes in $\mf{C}(\oo)_{\mf{B}}$ of $\mbf{1}_G^\vee$, $\mathrm{Sp}^\vee$ and $\pi_\alpha^\vee$, respectively.  The ring $\wt{E} = \wt{E}_{\mf{B}}=\End(\wt{P}_{\pi_\alpha^\vee} \oplus \wt{P}_{\mathrm{Sp}^\vee} \oplus \wt{P}_{\mbf{1}_G^\vee})$ is the $3 \times 3$ generalized matrix algebra 
\begin{equation}\label{eq: 3x3 pres 2}
\wt{E}_{\mf{B}} = 
\begin{pmatrix}
\End(\wt{P}_{\pi_\alpha^\vee}) & \Hom(\wt{P}_{\mathrm{Sp}^\vee},\wt{P}_{\pi_\alpha^\vee}) & \Hom(\wt{P}_{\mbf{1}_G^\vee},\wt{P}_{\pi_\alpha^\vee}) \\ \Hom(\wt{P}_{\pi_\alpha^\vee},\wt{P}_{\mathrm{Sp}^\vee}) & \End(\wt{P}_{\mathrm{Sp}^\vee}) & \Hom(\wt{P}_{\mbf{1}_G^\vee},\wt{P}_{\mathrm{Sp}^\vee}) \\ \Hom(\wt{P}_{\pi_\alpha^\vee},\wt{P}_{\mbf{1}_G^\vee}) & \Hom(\wt{P}_{\mathrm{Sp}^\vee},\wt{P}_{\mbf{1}_G^\vee}) & \End(\wt{P}_{\mbf{1}_G^\vee})
\end{pmatrix},
\end{equation}\label{eq: 3x3 pres 3}
which has the following description (see just after \cite[Cor.~10.94]{paskunas-image}):
\begin{equation}
\wt{E}_{\mf{B}} = 
\begin{pmatrix}
Re_1 & R\varphi_{12} & R\varphi_{13}^0 + R\varphi_{13}^1 \\ R\varphi_{21}^0 + R\varphi_{21}^1 & Re_2 & R\varphi_{23}^0 + R\varphi_{23}^1 \\ R\varphi_{31} & R\varphi_{32} + R\beta & Re_3
\end{pmatrix}.
\end{equation}

\begin{proposition}
$\wt{E}$ satisfies Assumption \ref{assumptions on E}.
\end{proposition}

\begin{proof}
The center of $\wt{E}$ is $R$ and $\wt{E}$ is finitely generated over $R$ by \cite[Thm.~10.87, Lem.~10.90]{paskunas-image}, respectively. It remains to show that every simple right $\wt{E}$-module has a finite injective resolution. Again (as in the proof of Proposition \ref{conditions on E generic}), this follows from vanishing of $\Ext^i(\pi_1,\pi_2)$ for all sufficiently large $i$ and all $\pi_1,\pi_2 \in \mf{B}$. This vanishing (for $i\geq 5$) is proved in \cite[\S 10.1]{paskunas-image}; see the table on p.~128 of \emph{loc.\ cit}.
\end{proof}

We now compare $\wt{E}$ to the Galois side. We use the notation of \S \ref{subsec: non-gen II coh sheaves} freely, and we set $X = L_{-1} \oplus L_1 \oplus Q$. We have $\Ext^i(X,X)=0$ for $i\geq 1$ by Propositions \ref{ext calcs 1 non-gen 2} and \ref{vanishing of Exts non-gen 2}, so to verify Assumption \ref{assumptions derived} it remains to show that $\wt{E} \cong \End(X)$. In our embedding of categories, the individual coherent sheaves $L_{-1}$, $L_1$ and $Q$ will correspond to $\wt{P}_{\pi_\alpha^\vee}$, $\wt{P}_{\mathrm{Sp}^\vee}$ and $\wt{P}_{\mbf{1}_G^\vee}$, respectively. Comparing equation (\ref{eq: 3x3 pres 3}) with equation (\ref{eq: 3x3 presentation 1}) and Theorem \ref{module structure nongen 2} we see that we have matched up the $R$-module generators of $\End(X)$ and $\wt{E}_{\mf{B}}$ by giving them the same name (the identity morphisms $e_i$ match up with $1 \in R$ in each case). It remains to check the relations, first for the $R$-module structure and then for the ring structure.

To make these comparisons we need to compare the notation used for the elements in $R$ in \S \ref{subsec: non-gen II} with that used by Pa{\v s}k{\=u}nas. In our presentation, we have 
\[
R=\oo \lb a_0,a_1,b_0 c,b_1 c \rb / (pb_0 c + a_1 b_0 c +a_0 b_1 c),
\]
and recall that we had set $\ap = a_0 + p$. In \cite[Lem.~10.93]{paskunas-image}, Pa{\v s}k{\=u}nas has a presentation\footnote{Note that \cite[Cor.~B.5]{paskunas-image} gives a slightly different presentation using the same variables, but the one we use is the one that is used in \cite[\S 10]{paskunas-image}.}
\[
R = \oo \lb c_0,c_1,d_0,d_1 \rb / (c_0 d_1 - c_1 d_0).
\]
The comparison between the two presentations is that $a_0$ corresponds to $d_0$, $a_1 + p =\ap$ corresponds to $-d_1$, and $b_i c$ corresponds to $c_i$ for $i=0,1$.

Let us now compare the $R$-module structures. There are five entries in the presentation (\ref{eq: 3x3 pres 3}) that are free of rank $1$, and the corresponding entries in (\ref{eq: 3x3 pres 2}) are also free of rank $1$ by \cite[Cor.~10.78 and Lem.~10.74, eqs (237) and (238)]{paskunas-image}. That leaves four entries, and we start with $\Hom(Q,L_{-1})$, which corresponds to $\Hom(\wt{P}_{\mbf{1}_G^\vee},\wt{P}_{\pi_\alpha^\vee})$. By \cite[Lem.~10.74, eq (241)]{paskunas-image}, we have an injection
\[
\Hom(\wt{P}_{\mbf{1}_G^\vee},\wt{P}_{\pi_\alpha^\vee}) \hookrightarrow \End(\wt{P}_{\mbf{1}_G^\vee})
\]
given by postcomposition with $\varphi_{31}$, and by \cite[Eq (246)]{paskunas-image} we have $\varphi_{31} \circ \varphi_{13}^i = c_i e_3$ for $i=0,1$. This shows that $\Hom(\wt{P}_{\mbf{1}_G^\vee},\wt{P}_{\pi_\alpha^\vee})$ is isomorphic to $c_0 R + c_1 R \sub R$ with $\varphi_{13}^i$ mapping to $c_i$, which matches with the structure of $\Hom(Q,L_{-1})$ from Theorem \ref{module structure nongen 2}(3). Next, we look at $\Hom(\wt{P}_{\pi_\alpha^\vee},\wt{P}_{\mathrm{Sp}^\vee})$ where we have an injection
\[
\Hom(\wt{P}_{\pi_\alpha^\vee},\wt{P}_{\mathrm{Sp}^\vee}) \hookrightarrow \End(\wt{P}_{\pi_\alpha^\vee})
\]
by \cite[Lem.~10.74, eq (240)]{paskunas-image}, given by postcomposing with $\varphi_{12}$, and by \cite[Eq (246)]{paskunas-image} we have $\varphi_{12} \circ \varphi_{21}^i = c_i e_1$ for $i=0,1$. This again shows that $\Hom(\wt{P}_{\pi_\alpha^\vee},\wt{P}_{\mathrm{Sp}^\vee})$ is isomorphic to $c_0 R + c_1 R $ with $\varphi_{21}^i$ mapping to $c_i$, which matches with the structure of $\Hom(L_{-1},L_1)$ from Theorem \ref{module structure nongen 2}(4) (note that $b_0 R + b_1 R$ is isomorphic to $c_0 R + c_1 R$ via multiplication by $c$ inside $S$). Next up is $\Hom(\wt{P}_{\mbf{1}_G^\vee},\wt{P}_{\mathrm{Sp}^\vee})$. By \cite[Lem.~10.74, eq (239)]{paskunas-image}, we have an isomorphism 
\[
\Hom(\wt{P}_{\mbf{1}_G^\vee},\wt{P}_{\mathrm{Sp}^\vee}) \cong \Hom(\wt{P}_{\mbf{1}_G^\vee},\wt{P}_{\pi_\alpha^\vee})
\]
given by postcomposing with $\varphi_{12}$, which satisfies $\varphi_{12} \circ \varphi_{23}^i = \varphi_{13}^i$. Hence $\Hom(\wt{P}_{\mbf{1}_G^\vee},\wt{P}_{\mathrm{Sp}^\vee})$ is isomorphic to $c_0 R + c_1 R$ with $\varphi_{23}^i$ mapping to $c_i$, matching the structure of $\Hom(Q,L_1)$ from Theorem \ref{module structure nongen 2}(5) (with the same remark as in the previous case). The final case is to compare $\Hom(\wt{P}_{\mathrm{Sp}^\vee},\wt{P}_{\mbf{1}_G^\vee})$ and $\Hom(L_1,Q)$. They both have generators $\varphi_{32}$ and $\beta$, so we need to check that the relations match. In the case of $\Hom(\wt{P}_{\mathrm{Sp}^\vee},\wt{P}_{\mbf{1}_G^\vee})$ the relations are
$c_i \beta = d_i \varphi_{32}$ for $i=0,1$ by \cite[Lem.~10.92]{paskunas-image}\footnote{Note that this reference has a typo.}, and this matches the result for $\Hom(L_1,Q)$ given in Theorem \ref{module structure nongen 2}(7).

This finishes the discussion of the $R$-module structure, so it remains to verify that the ring structures match; i.e.\ that the composing the generators gives the same results in both cases. For $\End(L_{-1} \oplus L_1 \oplus Q)$ this was computed in (the twelve parts of) Theorem \ref{ring structure non-gen II}. Parts (1), (8) and the first identity in (12) correspond to \cite[Eq (246)]{paskunas-image}. Parts (3), (4) and the first identity in (10) correspond to \cite[Eq (247)]{paskunas-image}. Parts (2) and (7) correspond to \cite[Eq (248)]{paskunas-image}. Part (5) and the first identity in (11) correspond to \cite[Eq (249)]{paskunas-image}. Part (6) and the first identity in (9) correspond to \cite[Eq (250)]{paskunas-image}. Finally, the last two identities in parts (9), (10), (11) and (12) correspond to \cite[Eq (251)]{paskunas-image}.

This finishes the verification that $\wt{E} \cong \End(X)$ as $R$-algebras, and gives us our functors 
\[
F_{disc} : \D(\Mod_{G,\zeta}^{lfin}(\oo)_{\mf{B}}) \to \IndCoh_{\m}(\mf{X}), \,\,\,\,\,\, F_{cpt} : \D(\mf{C}(\oo)_{\mf{B}}) \to \ProCoh_{\m}(\mf{X})
\]
at the derived level.

\begin{remark}\label{rem:what are we doing}
	At this point we can explain our motivation for the definition of $X$. We started out with the hypothesis that $F_{cpt}(\wt{P}_{\pi_\alpha^\vee})$ and $F_{cpt}(\wt{P}_{\mathrm{Sp}^\vee})$ should be $L_{-1}$ and $L_{1}$ respectively --- the correct assignment is determined by the fact that $\Hom(L_1,L_{-1})$ is a cyclic $R$-module. See also Proposition \ref{rem:vytas question}.
	
	Then we considered the short exact sequences (234) and (235) in \cite{paskunas-image}. The cokernel $L_{-1}/cL_{1}$ is supported on a substack of reducible Galois representations (cut out by the condition $c=0$). This is, at least heuristically, compatible with the fact that the cokernel of the corresponding map $\varphi_{12} \in \Hom(\wt{P}_{\mathrm{Sp}^\vee},\wt{P}_{\pi_\alpha^\vee})$ is (dual to) a parabolic induction (sequence (235)).  
	
	Looking at sequence (234), we were then naturally led to guess that the cokernel of the map \[F_{cpt}(\wt{P}_{\pi_\alpha^\vee}) \xrightarrow{\varphi_{31}}F_{cpt}(\wt{P}_{\mathbf{1}_G^\vee})\] 
	would be supported on the reducible substack cut out by $b_0=b_1=0$. This led us to consider the module $Q$ as a candidate for $F_{cpt}(\wt{P}_{\mathbf{1}_G^\vee})$. It is an extension of $\overline{Q}= L_{1}/(b_0,b_1)L_{-1}$ by $L_{-1}$. It is not hard to check that $\Ext^1(\overline{Q},L_{-1})$ is a cyclic $R$-module (isomorphic to $R/(b_0c,b_1c)$) and the extension class of $Q$ is a generator for this module. 
\end{remark}

\begin{proposition}\label{computation of irreducibles non-gen II}
We have \begin{align}F_{disc}(\pi_{\alpha}) &=  k[c][0] = L_1/(a_0,a_1',b_0,b_1, \varpi)[0]  \text{\quad($c$ in graded degree $-3$)}\\ F_{disc}(\Sp) &= k[b_0,b_1][0] = L_{-1}/(a_0,a_1',c, \varpi)[0] \text{\quad($b_0,b_1$ in graded degree $3$)} \\ F_{disc}(\mathbf{1}_G) &= k[b_0,b_1][-1] = L_{-3}/(a_0,a_1',c,\varpi)[-1] \text{\quad($b_0,b_1$ in graded degree $5$)} \end{align}
\end{proposition}
\begin{proof}
	We explain the details of the third case, which is most interesting. The first two are established in a very similar way. From Proposition \ref{prop:projresofsimple}, we have a perfect complex of left $\widetilde{E}$-modules \[P_{\OO,3}= \left[C_2^{\oplus 2} \xrightarrow{\left(\begin{smallmatrix}1 &0&\varphi_{23}^1\\0&1&-\varphi_{23}^0\end{smallmatrix}\right)}  C_2^{\oplus 2}\oplus C_3\xrightarrow{\left(\begin{smallmatrix}a_1' & -b_1c\\a_0&b_0c\\\beta &\varphi_{32}\end{smallmatrix}\right)} C_2^{\oplus 2} \xrightarrow{M':=\left(\begin{smallmatrix}
			b_0c & b_1c\\-a_0&a_1'
		\end{smallmatrix}\right)} C_2^{\oplus 2} \xrightarrow{\pi := \left(\begin{smallmatrix}-\varphi_{23}^1\\\varphi_{23}^0\end{smallmatrix}\right)} C_3\right]\] such that $\mathbf{1}_G$ corresponds to the mapping cone of $P_{\OO,3} \xrightarrow{\times\varpi} P_{\OO,3}$ in $\Perf^{L}(\wt{E})$. We deduce that $F_{disc}(\mathbf{1}_G)$ is the mapping cone of  $X^\ast \otimes_{\widetilde{E}} P_{\OO,3} \xrightarrow{\times\varpi} X^\ast \otimes_{\widetilde{E}} P_{\OO,3}$ in $\D_{coh}^b(\mf{X})$. This leaves us needing to understand the complex $ X^\ast \otimes_{\widetilde{E}} P_{\OO,3}$ which is 
	
	\begin{align*}\mathcal{C}_{\OO,3} = L_{-1} \oplus L_{-1} \xrightarrow{\left(\begin{smallmatrix}1 &0&(\varphi_{23}^1)^\ast\\0&1&-(\varphi_{23}^0)^\ast\end{smallmatrix}\right)}   L_{-1} \oplus L_{-1}\oplus Q^\ast\xrightarrow{\left(\begin{smallmatrix}a_1' & -b_1c\\a_0&b_0c\\\beta^\ast &\varphi_{32}^\ast\end{smallmatrix}\right)} &L_{-1} \oplus L_{-1} \xrightarrow{M'} L_{-1} \oplus L_{-1} \\& \xrightarrow{ \left(\begin{smallmatrix}-(\varphi_{23}^1)^\ast\\(\varphi_{23}^0)^\ast\end{smallmatrix}\right)} Q^\ast .\end{align*}
	
Note that here our maps are again given by matrices acting on row vectors from the right.

Comparing with the description of $Q^\ast$ in Proposition \ref{prop:Qvee description} and switching to column vectors, we see that \[H_0(\mathcal{C}_{\OO,3})=0 \text{ and }H_1(\mathcal{C}_{\OO,3}) = M^t(L_{-3}\oplus L_{-1})/(M')^t(L_{-1}\oplus L_{-1}).\] We claim that the map \begin{align*}L_{-3} &\to H_1(\mathcal{C}_{\OO,3})\\ x &\mapsto \begin{pmatrix}
		b_0x\\b_1x
	\end{pmatrix}
\end{align*} induces an isomorphism $\OO[b_0,b_1] = L_{-3}/(a_0,a_1',c) \cong H_1(\mathcal{C}_{\OO,3})$.

The map is clearly surjective, and factors through the specified quotient of $L_{-3}$, since \[\begin{pmatrix}
	b_0a_0\\b_1a_0
\end{pmatrix} = b_0\begin{pmatrix}
	a_0\\-a_1'
\end{pmatrix},~\begin{pmatrix}
	b_0a_1'\\b_1a_1'
\end{pmatrix} = b_1\begin{pmatrix}
	-a_0\\a_1'
\end{pmatrix}.\] The surviving graded pieces in $L_{-3}/(a_0,a_1',c)$ map to $\OO[b_0,b_1]\begin{pmatrix}
	b_0\\b_1
\end{pmatrix}$ which has trivial intersection with $(M')^t(L_{-1}\oplus L_{-1})$. To complete the computation of $F_{disc}(\mathbf{1}_G)$, it remains to check acyclicity of $\mathcal{C}_{\OO,3}$ in degree 2 and 3. This can be checked with Macaulay2 \cite{M2} which computes over the $\Z$-algebra \[\Z[a_0,a_1',b_0,b_1,c]/(a_0b_1+a_1'b_0).\] See \url{https://github.com/jjmnewton/p-adic-LLC} for the relevant Macaulay2 commands. Because $S$ is flat over this $\Z$-algebra, we deduce acyclicity for our complex of $S$-modules by base change. It is also not too difficult to check acyclicity of $\mathcal{C}_{\OO,3}$ in degree $2$ and $3$ by hand. 
\end{proof}

\begin{remark}
Since $F_{disc}(\mbf{1})$ is concentrated in homological degree $1$, we see that $F_{disc}$ does not come from deriving an embedding $\Mod_{G,\zeta}^{lfin}(\oo)_{\mf{B}} \to \QCoh_{\m}(\mf{X})$.
\end{remark}

\begin{remark}
Categorical formulations of the local Langlands correspondence have introduced the condition of nilpotent singular support (\cite{AG}, cf.~also \cite[\S VIII.2.2]{fargues-scholze}). In non-generic case II, the singularity stack $\mathrm{Sing}(\fX/\OO)$ is given by $\left[\Spec(\Sym_S \OO[c])/T\right]$, where $\OO[c]$ is the cyclic $S$-module $S/(a_0,a'_1,b_0,b_1)$, with $T$-action corresponding to $c$ being in graded degree $-2$. We have a zero section $\fX \to \mathrm{Sing}(\fX/\OO)$ with complement $\left[\G_{m,\OO[c]}/T\right]$. Its image in $\fX$ is the closed substack $\left[\Spec\OO[c]/T\right]$ cut out by $a_i = b_i = 0$, which is, as expected, the singular locus. Without making a general definition of nilpotent singular support, it seems clear in this situation that any member of $\IndCoh_{\m}(\mf{X})$ will have nilpotent singular support, since $c$ corresponds to a unipotent deformation.
\end{remark}

\section{The Montr\'{e}al functor and local-global compatibility}\label{chap: local-global}

In this section we show how to recover the Montr\'eal functor from our functors, and prove a local-global compatibility formula relating the singular homology of modular curves to the output of our functor, in the spirit of \cite[Exp.~Thm.~9.4.2]{egh}. We remark that the construction which recovers the Montr\'eal functor is a familiar and important construction in geometric Langlands; a Whittaker coefficent (cf.\ e.g.\ \cite{faergeman2022nonvanishing}). As for local-global compatibility, the most general statements of such formulas will involve the analogous functors for $\ell \neq p$, as considered in \cite{bzchn,hellmann-derived,zhu-coherent}. Our goal here will only be to illustrate how our functors fit in with such statements, rather than proving the strongest possible results. For this reason, we prove our results in the simplified setting of \cite[\S 7]{ceggps2} and \cite[\S 5]{gee-newton} (with $F=\Q$), where one ultimately does not need to worry about contributions from ramified primes $\ell \neq p$. We make one conceptual addition in that we work with $p$-arithmetic (co)homology\footnote{In general, the versions involving local Langlands functors for $\ell \neq p$ should be formulated using $S$-arithmetic (co)homology, where $S$ is set of ramified primes and $p$.}, as defined for example in \cite{tarrach2022sarithmetic}, instead of the (co)homology of modular curves. This matches very well with our functors (and those of \cite{egh}) and allows us to extend our formula to spaces of interest in the theory of eigenvarieties as well. 

\subsection{Local considerations}\label{subsec: local considerations}

In this subsection we will prove all the local preparations needed for the local-global compatibility statement. To be able to prove a statement valid for homology of modular curves with essentially arbitrary ($p$-adic) locally constant coefficients, we need to expand the domain of our functors. Fix a block $\mf{B}$. Recall that our functor
\[
F_{disc} : \D(\Mod_{G,\zeta}^{lfin}(\oo)_{\mf{B}}) \to \IndCoh_{\m}(\mf{X}_{\mf{B}}) \sub \IndCoh(\mf{X}_{\mf{B}})
\] 
is the composition of the fully faithful embeddings
\[
F : \D(\LMod(\wt{E}_{\mf{B}})) \to \IndCoh(\mf{X}_{\mf{B}})
\]
and
\[
J : \D(\Mod_{G,\zeta}^{lfin}(\oo)_{\mf{B}}) \to \D(\LMod(\wt{E}_{\mf{B}})),
\]
where the latter is $t$-exact and given by 
\[
\sigma \mapsto \Hom_G(P_{\mf{B}},\sigma^\vee)^\vee = \Hom_{G}(\sigma, P_{\mf{B}}^\vee)^\vee
\]
already at the level of abelian categories. We will expand the domain of $F_{disc}$ by expanding the domain of $J$. To this end, we wish to show that
\[
\Hom_{G}(\sigma, P_{\mf{B}}^\vee)^\vee = P_{\mf{B}} \otimes_{\oo \lb G \rb_\zeta }\sigma
\]
for all $\sigma \in \Mod_{G^{ad}}^{lfin}(\oo)_{\mf{B}}$. Here we recall that $\oo \lb G \rb $ is the ring of compactly supported measures on $G$, originally considered by Kohlhaase \cite{kohlhaase} (we refer to \cite[\S 3]{shotton} and \cite[Definition E.1.1]{egh} for the definition in our context). The ring $\oo \lb G \rb_{\zeta}$ is the quotient of $\oo \lb G \rb$ by the two-sided ideal generated by $\{z-\zeta(z)\mid z\in Z\}$. Every smooth $G$-representation over $\oo$ with central character $\zeta$ is a $\oo \lb G \rb_\zeta$-module in a unique way by \cite[Lem.~3.5]{shotton}, so the tensor product above makes sense and we may rewrite $\Hom_{G}(\sigma, P_{\mf{B}}^\vee)^\vee$ as $\Hom_{\oo \lb G \rb_\zeta }(\sigma, P_{\mf{B}}^\vee)^\vee$. Note that we have a natural transformation $i_\sigma : P_{\mf{B}} \otimes_{\oo \lb G \rb_\zeta }\sigma \to \Hom_{\oo \lb G \rb_\zeta }(\sigma, P_{\mf{B}}^\vee)^\vee$ defined by linearly extending the formula
\[
x \otimes v \mapsto \big(( f : \sigma \to P_{\mf{B}}^\vee) \mapsto f(v)(x) \big).
\]
This makes sense not only for $\sigma \in \Mod_{G,\zeta}^{lfin}(\oo)_{\mf{B}}$, but also for finitely presented $\oo \lb G \rb_\zeta$-modules.

\begin{lemma}\label{hom = tensor for fp}
We have $\Hom_{\oo \lb G \rb_\zeta}(\sigma, P_{\mf{B}}^\vee)^\vee = P_{\mf{B}} \otimes_{\oo \lb G \rb_\zeta }\sigma$ as functors from finitely presented $\oo \lb G \rb_\zeta$-modules to left $\wt{E}_{\mf{B}}$-modules.
\end{lemma}

\begin{proof}
The proof follows a standard pattern: first, $i_\sigma$ is an isomorphism for $\sigma = \oo \lb G \rb_\zeta $, and from this one gets that it is an isomorphism in general by taking a presentation and using the five lemma (note that both functors are right exact). 
\end{proof}

We then get the formula we want on $\Mod_{G_\zeta}^{lfin}(\oo)_{\mf{B}}$. 

\begin{proposition}\label{hom = tensor for ladm}
We have $\Hom_{\oo \lb G \rb_\zeta }(\sigma, P_{\mf{B}}^\vee)^\vee = P_{\mf{B}} \otimes_{\oo \lb G \rb_\zeta }\sigma$ as functors from  $\Mod_{G,\zeta}^{lfin}(\oo)_{\mf{B}}$ to left $\wt{E}_{\mf{B}}$-modules.
\end{proposition}

\begin{proof}
Since both functors commute with direct limits, it suffices to show that we have an isomorphism for finite length representation. In view of Lemma \ref{hom = tensor for fp}, it therefore suffices to show that any finite length representation is finitely presented as an $\oo \lb G \rb_\zeta$-module. But this follows from \cite[Thm.~1.1(2)(i)]{vigneras-foncteur-de-colmez} and \cite[Prop.~3.8]{shotton}.
\end{proof}

Thus, we see that $J(\sigma) = P_{\mf{B}} \otimes_{\oo \lb G \rb_\zeta }\sigma$ for $\sigma \in \Mod_{G,\zeta}^{lfin}(\oo)_{\mf{B}}$. We can then attempt to expand the domain of $J$ to all of $\LMod(\oo \lb G \rb_\zeta)$ by defining
\[
J_{ext} : \LMod(\oo \lb G \rb_\zeta) \to \LMod(\wt{E}_{\mf{B}})
\]
by $J_{ext}(\sigma) = P_{\mf{B}} \otimes_{\oo \lb G \rb_\zeta }\sigma$ and taking the unbounded left derived functor $LJ_{ext}$ of $J_{ext}$\footnote{Recall that this may be defined since $\D(\LMod(\oo \lb G \rb_\zeta ))$ has enough K-projective complexes, by \cite{spaltenstein}.}. Although we will, strictly speaking, not need it, we will prove that $LJ_{ext}$ really is an extension of $J$. We start by noting that $P_{\mf{B}}$ is a flat right $\oo \lb K \rb_\zeta$-module, where $K=\GL_2(\Z_p)$ and $\oo \lb K \rb_\zeta$ is the quotient of $\oo \lb K \rb$ by the two-sided ideal generated by $z-\zeta(z)$, for $z\in Z\cap K$. If $\tau$ is a left $\oo \lb K \rb_\zeta$-module, we write $\mathrm{ind}_{KZ}^G\tau$ for $\oo \lb G \rb_\zeta \otimes_{\oo \lb K \rb_\zeta} \tau$ (if $\tau$ is smooth, this is the usual compact induction with fixed central character).

\begin{lemma}\label{compact inductions are acyclic}
$P_{\mf{B}}$ is a flat right $\oo \lb K \rb_\zeta$-module. As a consequence, $\Tor_i^{\oo \lb G \rb_\zeta}(P_{\mf{B}},\mathrm{ind}^G_{KZ}\tau) = 0$ for all $i \geq 1$ and all left $\oo \lb K \rb_\zeta$-modules $\tau$.
\end{lemma}

\begin{proof}
By \cite[Cor.~5.18]{paskunas-image}, $P_{\mf{B}}^\vee$ is injective as a smooth $G$-representation with central character $\zeta$. As a consequence, the restriction to $K$ is also injective as a smooth $K$-representation with central character $\zeta$ (compact induction is an exact left adjoint to restriction). Dually, $P_{\mf{B}}$ is then projective as a compact right $\oo \lb K \rb_\zeta$-module, hence exact for the completed tensor product, and hence exact for the usual tensor product and finitely generated right $\oo \lb K \rb_\zeta$-modules. Hence $P_{\mf{B}}$ is a flat right $\oo \lb K \rb_\zeta$-module. The second part then follows since $\oo \lb G \rb_\zeta$ is flat as a (left and right) $\oo \lb K \rb_\zeta$-module.
\end{proof}

\begin{proposition}\label{LJ is an extension}
Write $\iota$ for the inclusion $\Mod_{G,\zeta}^{lfin}(\oo)_{\mf{B}} \sub \LMod(\oo \lb G \rb_\zeta)$ and its unbounded derived functor. Then $LJ_{ext}\circ \iota = J$. 
\end{proposition}

\begin{proof}
At the level of abelian categories we have $J_{ext}\circ \iota = J$, so we have a natural transformation $J \to LJ_{ext}\circ \iota$. Both functors commute with colimits, so it suffices to check that the natural transformation is an isomorphism on irreducible objects, i.e.\ that $P_{\mf{B}} \otimes_{\oo \lb G \rb_\zeta }\pi = P_{\mf{B}} \otimes^L_{\oo \lb G \rb_\zeta }\pi$ for irreducible $\pi$. Pick such a $\pi$. Viewed as a smooth representation, $\pi$ is finitely presented, and the category of finitely presented smooth representations is abelian (see e.g.\ \cite[Thm.~1.2]{shotton}), so there is a resolution $\mathrm{ind}_{KZ}^G\tau_\bu \to \pi$ with the $\tau_i$ finitely presented smooth $K$-representations with central character $\zeta$. By Lemma \ref{compact inductions are acyclic}, we have
\[
P_{\mf{B}} \otimes^L_{\oo \lb G \rb_\zeta }\pi = P_{\mf{B}} \otimes_{\oo \lb G \rb_\zeta } \mathrm{ind}_{KZ}^G\tau_\bu.
\]
By Lemma \ref{hom = tensor for fp}, we have $P_{\mf{B}} \otimes_{\oo \lb G \rb_\zeta } \mathrm{ind}_{KZ}^G\tau_\bu = \Hom_G(\mathrm{ind}_{KZ}^G\tau_\bu,P_{\mf{B}}^\vee)^\vee$. Since $P_{\mf{B}}^\vee$ is injective as a smooth $G$-representation with central character $\zeta$ \cite[Cor.~5.18]{paskunas-image} and Pontryagin duality is exact, the homology of $\Hom_G(\mathrm{ind}_{KZ}^G\tau_\bu,P_{\mf{B}}^\vee)^\vee$ is concentrated in degree $0$, which finishes the proof.
\end{proof}

\begin{remark}
As a sanity check, we remark that $LJ_{ext}$ kills the other blocks in $\Mod_{G,\zeta}^{lfin}(\oo)$. Indeed, let $\mf{B}^\prime \neq \mf{B}$ be a block. To show that $LJ_{ext}(\sigma)=0$ for $\sigma \in \Mod_{G,\zeta}^{lfin}(\oo)_{\mf{B}^\prime}$, it suffices (as in the proof above) to show this for irreducible $\sigma$. But then (by Lemma \ref{hom = tensor for fp} again) we have $LJ_{ext}(\sigma)= \Hom_G(\sigma,P_{\mf{B}}^\vee)^\vee$, which vanishes.
\end{remark}

We can now extend $F_{disc}$ to $\D(\LMod(\oo \lb G \rb_\zeta))$. For simplicity, and since it is the only functor we need for the local-global formula, we will take the codomain of our extension to be $\D_{qcoh}(\mf{X}_{\mf{B}})$ instead of $\IndCoh(\mf{X}_{\mf{B}})$. Write $\ol{F}$ for $F$ composed with the natural functor $\IndCoh(\mf{X}_{\mf{B}}) \to \D_{qcoh}(\mf{X}_{\mf{B}})$, and define 
\[
F_{ext} : \D(\LMod(\oo \lb G \rb_\zeta)) \to \D_{qcoh}(\mf{X}_{\mf{B}})
\]
by $F_{ext} = \ol{F} \circ LJ_{ext}$. Explicitly, we have $F_{ext}(\sigma) = X_{\mf{B}}^\ast \otimes^L_{\wt{E}_{\mf{B}}} P_{\mf{B}} \otimes^L_{\oo \lb G \rb_\zeta} \sigma$ for $\sigma \in \D(\LMod(\oo \lb G \rb_\zeta))$.

In the rest of this subsection, we will compute $j^\ast F_{ext}(\sigma)$ for certain open immersions $j$, as preparation for the local-global formula. Our starting point is then a continuous representation $\rho : \Gamma_{\Qp} \to \GL_2(\Fpbar)$. We assume that $\End_{\Gamma_{\Qp}}(\rho) = \Fpbar$ and that if $\rho$ is reducible of the form
\begin{equation}
\label{rho extn form}
0 \to \chi_2 \to \rho \to \chi_1 \to 0,
\end{equation}
then $\chi_2 \chi_1^{-1} \neq \omega$. Note that the assumption on endomorphisms implies that $\chi_1 \neq \chi_2$. With $\rho$, we associate an irreducible $G$-representation $\pi$ via the recipe of \cite[Lem.~2.15(5)]{ceggps2} twisted by $\omega^{-1}$ (in particular $\pi$ is, up to twist, a quotient of the compact induction of a Serre weight for $\rho$). We make the twist in order to match our normalization of the bijection between blocks and semisimple two-dimensional $\Gamma_{\Qp}$-representation from \S \ref{sec: blocks for GL2}; it ensures that $\pi$ lies in the block $\mf{B}$ corresponding to the semisimplification of $\rho$. In particular, $\pi$ satisfies $\vc(\pi^\vee) = \rho$ when $\rho$ is irreducible; and when $\rho$ is reducible of the form \eqref{rho extn form}, it follows from \cite[Thm.~30]{barthel-livne} that $\pi = \Ind_B^G(\chi_1 \otimes \chi_2\omega^{-1})$. 

We let $P$ be the projective envelope of $\pi^{\vee}$ and let $R_{\rho}$ be the universal deformation ring of $\rho$ (with fixed determinant corresponding to the central character in $\mf{B}$). Writing $R$ for the universal pseudodeformation ring of the trace of $\rho$, we note that the natural map $R \to R_\rho$ is an isomorphism (see \S \ref{subsec: supersingular case} for the irreducible case and e.g.\ \cite[Cor.~B.16, Prop.~B.17]{paskunas-image} for the reducible case). In light of this, we will simply write $R$ for $R_\rho$. Any choice of representation in the strict equivalence class of the universal representation 
\[
\rho^{univ} : \Gamma_{\Qp} \to \GL_2(R)
\]
is a compatible representation, and hence determines a morphism
\begin{equation}
\label{eq: R-point}
\Spec R \to \mf{X}_\mf{B}
\end{equation}
which is a section to the map $\mf{X}_\mf{B} \to \Spec R$ sending a representation to its pseudorepresentation. Moreover, the map in (\ref{eq: R-point}) is independent of the choice of $\rho^{univ}$ up to $\SL_2$-conjugacy (and hence the choice in the strict equivalence class), so it factors through a map
\begin{equation}
\label{eq: j def} 
j : \mf{X}_{\rho} := [\Spec R /\mu_2] \to \mf{X}_\mf{B}.
\end{equation}
When $\rho$ is irreducible (i.e.\ $\mf{B}$ is a supersingular block), $j$ is simply the identity map, $\mf{X}_{\mf{B}} = [\Spec R /\mu_2]$, $X_{\mf{B}}^\ast = R(1)$ (i.e the $R$-module $R$, viewed as a $\Z/2$-graded $R$-module concentrated in degree $1$) and $P=P_{\mf{B}}$. Thus, we have the following formula:

\begin{proposition}\label{pullback formula ss}
Assume that $\mf{B}$ is supersingular. Then we have $j^\ast(F_{ext}(\sigma)) = P(1)\otimes^L_{\oo \lb G \rb_\zeta} \sigma$ for $\sigma \in \D(\LMod(\oo \lb G \rb_\zeta))$.
\end{proposition}

Let us now analyze the map $j$ when $\rho$ is reducible. Our assumption on $\rho$ puts us in one of two cases: either $\mf{B}$ is generic principal series or non-generic II. While the concrete description of $j$ in \eqref{eq: R-point} and \eqref{eq: j def} characterizes it, it will be helpful to realize it as a case of a more general phenomenon studied in \cite[\S2.2]{WEthesis}, especially Corollary  2.2.4.3 \textit{ibid}. In \textit{loc.\ cit.}, the $R$-projectivity\footnote{That is, representability by projective scheme over the pseudodeformation ring.} of these substacks of moduli stacks of representations is emphasized, but these subspaces are also open in $\mf{X}_{\mf{B}}$, which is what is more relevant here. 
\begin{proposition}
\label{prop: open proj}
Adopting the notation for coordinates $E_{i,j}$ of $E$ from Proposition \ref{prop: GMA description}, the substack of $\mf{X}_{\mf{B}} = [\Spec S / \bG_m]$ of adapted representations of the form 
\[
\rho : E \to M_2(B), \qquad 
\rho = \begin{pmatrix}
\rho_{1,1} & \rho_{1,2} \\
\rho_{2,1} & \rho_{2,2}
\end{pmatrix} : \begin{pmatrix}
E_{1,1} & E_{1,2} \\
E_{2,1} & E_{2,2}
\end{pmatrix} \to M_2(B)
\]
such that $\rho_{2,1}(E_{2,1})$ generates $B$ (as a $B$-module) is represented by the fiber product of $\Proj_R E_{2,1}$ and $\mf{X}_{\mf{B}}$ over $[\Spec \Sym_R^* E_{2,1} / \bG_m]$ (where $E_{2,1}$ has graded degree $-2 \in X^*(\bG_m)$). This subspace of $\mf{X}_{\mf{B}}$ is open and is presentable as a $\Spec R$-projective scheme equipped with the trivial action of $\mu_2$. 
\end{proposition}
We recall that in our description of the GMA structure, the character $\chi_1$ corresponds to the top left entry and $\chi_2$ to the bottom right. Note that the morphism $\mf{X}_{\mf{B}} \to [\Spec \Sym_R^* E_{2,1} / \bG_m]$ arises naturally from the presentation of $S$ stated in Proposition \ref{prop: GMA description}. We also remark that our use of $\Proj$ refers to the usual notion of a (closed substack of a) weighted projective stack. 

\begin{proof}
The representability of the stated moduli subgroupoid by the stated fiber product follows from comparing the condition on $\rho_{2,1}(E_{2,1})$ to the definition of $\Proj_R M$ as a subgroupoid of $[\Spec \Sym_R^* M / \bG_m]$. This is open because $\Proj_R M$ is open in $[\Spec \Sym_R^* M / \bG_m]$, having arisen by removing the origin. 
\end{proof}

What is common to the generic principal series and non-generic II cases is that $E_{2,1}$ is a free cyclic $R$-module generated by $c$. Therefore the condition that $\rho_{2,1}(E_{2,1})$ generates the $(2,1)$-coordinate amounts to $\rho$ being conjugate to a deformation of the unique (up to isomorphism of representations) non-trivial extension $\rho$ of $\chi_1$ by $\chi_2$, and this condition is cut out by inverting $c$. Thus our morphism $j$ is the base change of $\mf{X}_{\mf{B}} \to [\Spec \Sym_R^* E_{2,1} / \bG_m]$ along $\Proj_R E_{2,1}$. To summarize this analysis, we see that $\Proj_R cR = [\Spec R/ \mu_2]$ and state
\begin{corollary}
\label{cor: j is open}
$j : [\Spec R / \mu_2] \to \mf{X}_{\mf{B}}$ is an open immersion obtained by adjoining $c^{-1}$. 
\end{corollary}

It will be helpful to make explicit computations with the graded $R$-algebra map corresponding to $j$ as we apply Proposition \ref{prop: open proj}, writing it using the generator $c$ as 
\[
\phi : S \to R[c,c^{-1}], \text{ uniquely determined by } S \ni c \mapsto c. 
\]
We begin with the generic principal series case, using the computation of $S$ of \S\ref{subsec: stacks generic ps GL2}. In odd degrees both sides are $0$. In degree $-2n$, for $n\geq 0$, $\phi$ is the identity $c^n R \to c^n R$. In degree $2n$, for $n\geq 1$, $\phi$ is given by the inclusion $b^n R \to c^{-n} R$. In particular, $\phi$ is injective and equates $R[c,c^{-1}]$ with $S[c^{-1}]$. 

To prove the analogue of Proposition \ref{pullback formula ss} in the generic principal series case, we will also need to understand $P$. From our choice of $\rho$, we have $\pi = \Ind_B^G(\chi_1 \otimes \chi_2\omega^{-1})$. Recall that we fixed an isomorphism $\wt{E}_{\mf{B}} \cong \left( \begin{smallmatrix} R & bR \\ cR & R \end{smallmatrix} \right)$ in \S \ref{subsubsec: geom gps}, and that under this isomorphism $P$ (the projective envelope of $\pi^\vee$) corresponds to the right $\wt{E}_{\mf{B}}$-module  $ \begin{pmatrix} cR & R \end{pmatrix}$. 

\begin{proposition}\label{pullback of universal module generic ps}
The pullback $j^\ast(X_{\mf{B}}^\ast \otimes_{\wt{E}_{\mf{B}}} P_{\mf{B}})$ is $P(1)$ (viewed as a $\Z/2$-graded $R$-module).
\end{proposition}

\begin{proof}
First note that $j^\ast(X_{\mf{B}}^\ast \otimes_{\wt{E}_{\mf{B}}} P_{\mf{B}}) = j^\ast(X_{\mf{B}}^\ast) \otimes_{\wt{E}_{\mf{B}}} P_{\mf{B}}$. Recall from \S \ref{subsubsec: geom gps} that $X_{\mf{B}}$ is the graded module $L_1 \oplus L_{-1}$, in the notation of \S \ref{subsec: stacks generic ps GL2}; the left $\wt{E}_{\mf{B}}$-module structure is then obtained by viewing $L_1 \oplus L_{-1}$ as column vectors. The right module $\wt{E}_{\mf{B}}$-module $X_{\mf{B}}^\ast$ is therefore $L_{-1} \oplus L_1$, now viewed as row vectors. We have a decomposition
\[
L_{-1} \oplus L_1 = \left( \bigoplus_{n=0}^\infty \begin{pmatrix} b^n R & b^{n+1}R \end{pmatrix} \right) \oplus \left( \bigoplus_{n=0}^\infty \begin{pmatrix} c^{n+1}R & c^n R \end{pmatrix} \right)
\]
into graded pieces, and these pieces are right $\wt{E}_{\mf{B}}$-modules. It follows that $j^\ast(X_{\mf{B}}^\ast)$ is the graded $S[c^{-1}]=R[c,c^{-1}]$-module
\[
(L_{-1} \oplus L_1)[c^{-1}] = \left( \bigoplus_{n=0}^\infty \begin{pmatrix} c^{-n}R & c^{-n-1}R \end{pmatrix} \right) \oplus \left( \bigoplus_{n=0}^\infty \begin{pmatrix} c^{n+1}R & c^n R \end{pmatrix} \right)
\]
Applying $-\otimes_{\wt{E}_{\mf{B}}}P_{\mf{B}}$, we see that $j^\ast(X_{\mf{B}}^\ast \otimes_{\wt{E}_{\mf{B}}} P_{\mf{B}})$ is the graded $R[c,c^{-1}]$-module $P[c,c^{-1}](1)$. This corresponds to the $\Z/2$-graded $R$-module in the statement of the proposition.
\end{proof}

In general, we have the following formula.

\begin{corollary}\label{pullback formula gps}
We have $j^\ast(F_{ext}(\sigma)) = P(1) \otimes^L_{\oo \lb G \rb_\zeta } \sigma$ for all $\sigma \in \D(\LMod(\oo \lb G \rb_\zeta))$ .
\end{corollary}

\begin{proof}
Since $j^\ast$ is exact at the level of abelian categories, we have
\[
j^\ast(F_{ext}(\sigma)) = j^\ast((X_{\mf{B}}^\ast \otimes_{\wt{E}_{\mf{B}}} P_{\mf{B}}) \otimes^L_{\oo \lb G \rb_\zeta } \sigma) = j^\ast((X_{\mf{B}}^\ast \otimes_{\wt{E}_{\mf{B}}} P_{\mf{B}})) \otimes^L_{\oo \lb G \rb_\zeta } \sigma.
\]
The result then follows from Proposition \ref{pullback of universal module generic ps}.
\end{proof}

Finally, we come to the non-generic II case. Recall from \S\ref{subsec: non-gen II} the presentation $\mf{X}_{\mf{B}} = [\Spec S /\G_m]$ with $S= \oo \lb a_0,a_1^\prime, b_0c, b_1 c \rb [b_0,b_1,c] / (a_0 b_1 + a_1^\prime b_0)$. As in the generic principal series case, $\mf{X}_{\rho}$ is the open substack of $\mf{X}_{\mf{B}}$ given by the condition $c\neq 0$ according to Corollary \ref{cor: j is open}, and moreover $\pi = \pi_\alpha = \Ind_B^G(\omega \otimes \omega^{-1})$. Let us explicate the map $\phi: S \to S[c^{-1}]= R[c,c^{-1}]$ like we did in the generic principal series case. In odd degrees, both sides are $0$. In degrees $-2n$, $n\geq 0$, it is the identity $c^n R \to c^n R$, and in degrees $2n$, $n\geq 0$, it is the inclusion $(b_0 R + b_1 R)^n  \to c^{-n}R$.

We now aim to prove the analogue of Proposition \ref{pullback of universal module generic ps}. The object $X_{\mf{B}}$ is defined to be $L_{-1} \oplus L_1 \oplus Q$, in the notation of \S \ref{subsec: non-gen II coh sheaves}. First, recall from \S \ref{subsec: non-gen II coh sheaves} and \S \ref{subsubsec: geom ng2} that
\[
\wt{E}_{\mf{B}} = 
\begin{pmatrix}
\End(L_{-1}) & \Hom(L_1,L_{-1}) & \Hom(Q,L_{-1}) \\ \Hom(L_{-1},L_1) & \End(L_1) & \Hom(Q,L_1) \\ \Hom(L_{-1},Q) & \Hom(L_1,Q) & \End(Q)
\end{pmatrix}.
\]
As recalled in \S \ref{subsec: notation}, if $M$ is a finitely generated graded $S$-module, then its dual $M^\ast$ has grading given by $(M^\ast)_k = \Hom(M,L_k)$. In particular, we see that the first row in $\wt{E}_{\mf{B}}$ is the grade $-1$ part of $X_{\mf{B}}^\ast = (L_{-1} \oplus L_1 \oplus Q)^\ast = L_1 \oplus L_{-1} \oplus Q^\ast$. As a right $\wt{E}_{\mf{B}}$-module, it corresponds to $P = P_{\pi_\alpha^\vee} \in \mf{C}(\oo)_{\mf{B}}$ under the equivalence $\mf{C}(\oo)_{\mf{B}} \cong \RMod^{cpt}(\wt{E}_{\mf{B}})$. 

We now show that the maps $L_1 \to L_1[c^{-1}]$, $L_{-1} \to L_{-1}[c^{-1}]$ and $Q^\ast \to Q^\ast[c^{-1}]$ are isomorphisms in degree $-1$. For completeness, we say a bit more about them. First, note that they are all $0$ in even degrees, because both sides are $0$. Let $n$ be odd. From the description of the map $S \to S[c^{-1}]$, we see that $L_n \to L_n[c^{-1}]$  is a non-zero isomorphism in odd degrees $\leq -n$ and injective but not an isomorphism in odd degrees $> -n$. In particular, both $L_1 \to L_1[c^{-1}]$ and $L_{-1} \to L_{-1}[c^{-1}]$ are isomorphisms in degree $-1$. For $Q^\ast \to Q^\ast[c^{-1}]$, recall from Proposition \ref{prop:Qvee description} that $Q^\ast$ is the cokernel of
\[
\begin{pmatrix} b_0 & -a_0 \\ b_1 & \ap \end{pmatrix} : L_{-3} \oplus L_{-1} \to L_{-1} \oplus L_{-1}.
\] 
From the remark above about the map $L_n \to L_n[c^{-1}]$, it then follows that $Q^\ast \to Q^\ast[c^{-1}]$ is an isomorphism in odd degrees $\leq 1$. In particular, we now see that $X_{\mf{B}}^\ast \to j^\ast X_{\mf{B}}^\ast = X_{\mf{B}}^\ast [c^{-1}]$ is an isomorphism in degree $-1$. We can now prove the analogue of Proposition \ref{pullback of universal module generic ps}.

\begin{proposition}\label{pullback of universal module non-gen II}
The pullback $j^\ast(X_{\mf{B}}^\ast \otimes_{\wt{E}_{\mf{B}}} P_{\mf{B}})$ is $P(1)$ (viewed as a $\Z/2$-graded $R$-module).
\end{proposition}

\begin{proof}
Since $X_{\mf{B}}^\ast \otimes_{\wt{E}_{\mf{B}}} P_{\mf{B}}$ is concentrated in odd degrees, we know that $j^\ast(X_{\mf{B}}^\ast \otimes_{\wt{E}_{\mf{B}}} P_{\mf{B}})$ will be concentrated in the non-zero degree (as a $\Z/2$-graded module), so it suffices to prove that $(j^\ast(X_{\mf{B}}^\ast \otimes_{\wt{E}_{\mf{B}}} P_{\mf{B}}))_{-1} = P$. But we have 
\[
j^\ast(X_{\mf{B}}^\ast \otimes_{\wt{E}_{\mf{B}}} P_{\mf{B}}) = (j^\ast X_{\mf{B}}^\ast) \otimes_{\wt{E}_{\mf{B}}} P_{\mf{B}}
\]
and above we have shown that $X_{\mf{B}}^\ast \to j^\ast X_{\mf{B}}^\ast$ is an isomorphism in degree $-1$, and that $(X^\ast_{\mf{B}})_{-1}$ is the right $\wt{E}_{\mf{B}}$-module corresponding to $P$. The result follows.
\end{proof}

We then get the analogue of Corollary \ref{pullback formula gps} from Proposition \ref{pullback of universal module non-gen II}, with the same proof.

\begin{corollary}\label{pullback formula non-gen II}
We have $j^\ast(F_{ext}(\sigma)) = P(1) \otimes^L_{\oo \lb G \rb_\zeta } \sigma$.
\end{corollary}

\begin{remark}\label{kernel of functor}
Propositions \ref{pullback of universal module generic ps} and \ref{pullback of universal module non-gen II} suggests that the `kernel' $X_\mf{B}^\ast \otimes_{\wt{E}_\mf{B}}P_\mf{B}$ used to define $F_{ext}$ is an interpolation of the projective envelopes of the irreducibles over the moduli stack of Galois representations. In particular, it appears to be more `canonical' than $X_{\mf{B}}$ itself.
\end{remark}
\subsection{Recovering the Montr\'{e}al functor}\label{subsec:montreal}
This subsection contains a result that is proved using similar considerations to the previous subsection. It answers a question raised by Pa{\v s}k{\=u}nas in correspondence with us. We use the covariant functor $\vc: \mf{C}(\oo) \to  \Mod_{\Gamma_{\Qp}}^{cpt}(\oo)$ to continuous $\Gamma_{\Qp}$-representations on compact $\oo$-modules introduced in \cite[\S5.7]{paskunas-image}. On finite length objects it is defined as $\vc(M) = \vc(M^\vee)$, in terms of the renormalized Montr\'{e}al functor on smooth representations we recalled in \S\ref{sec: blocks for GL2}. It extends to $\mf{C}(\oo)$ by taking limits. 

Our first proposition will describe the Montr\'{e}al functor applied to projective envelopes in $\mf{C}(\oo)$, in terms of our functor $F_{cpt}$. In fact, we compose \[F_{cpt} : \D(\mf{C}(\oo)_{\mf{B}}) \to \ProCoh(\mf{X}_{\mf{B}})\] with the functor $\ProCoh(\mf{X}_{\mf{B}})\to \D_{qcoh}(\mf{X}_{\mf{B}})$ given by taking limits to get \[\ol{F}_{cpt}: \D(\mf{C}(\oo)_{\mf{B}}) \to \D_{qcoh}(\mf{X}_{\mf{B}}).\] For the projective envelope $\wt{P}_{\pi^\vee}$ of the dual of an absolutely irreducible representation, we have prescribed the image $\ol{F}_{cpt}(\wt{P}_{\pi^\vee})$ in \S\ref{chap: geometric interpretation}. When $\pi$ is infinite-dimensional, $\ol{F}_{cpt}(\wt{P}_{\pi^\vee})$ is a vector bundle. 
\begin{proposition}\label{rem:vytas question}
Fix a block $\mf{B}$ containing an absolutely irreducible representation $\pi$, and assume that $\pi$ is infinite-dimensional. Let $\mc{V}$ be the vector bundle on $\fX_{\mf{B}}$ carrying the universal Galois representation. Then we have a $R[\Gamma_{\Qp}]$-equivariant isomorphism
\[
R\Gamma(\fX_{\mf{B}},\mathcal{V}\otimes_{\OO_{\fX_{\mf{B}}}}\ol{F}_{cpt}(\wt{P}_{\pi^\vee})) \cong \vc(\wt{P}_{\pi^\vee}).
\]
\end{proposition}

\begin{proof}
	We split up into cases based on the type of block $\mf{B}$. 
	
	Firstly, suppose we are in the supersingular case. So $\mf{X}_{\mf{B}} = [\Spec R_\rho/\mu_2]$ for an irreducible $\rho$, $R \cong R_\rho$, and $\ol{F}_{cpt}(\wt{P}_{\pi^\vee}) = R(1)$.  We can identify $\mc{V}$ with $\rho^{univ}(1)$ (i.e.\ $\rho^{univ}$ concentrated in the non-zero degree). This identifies $R\Gamma(\fX_{\mf{B}},\mathcal{V}\otimes_{\OO_{\fX_{\mf{B}}}}\ol{F}_{cpt}(\wt{P}_{\pi^\vee}))$ with $\rho^{univ}$. On the other hand, $\vc(\wt{P}_{\pi^\vee})$ is also isomorphic to $\rho^{univ}$ \cite[Prop.~6.3]{paskunas-image}.
	
	Now suppose we are in the generic principal series case with $\pi = \Ind_B^G (\chi_1 \otimes \chi_2 \omega^{-1})$. As in the previous subsection, we let $\rho: \Gamma_{\Qp} \to \GL_2(\Fpbar)$ be a non-split extension of the form  
	\[
	0 \to \chi_2 \to \rho \to \chi_1\to 0,
	\] and consider the open immersion $j: \mf{X}_{\rho} \to \mf{X}_{\mf{B}}$ given by inverting $c$ (which has graded degree $-2$). We have $\ol{F}_{cpt}(\wt{P}_{\pi^\vee}) = L_{-1}$ and $\mc{V} = L_1 \oplus L_{-1}$. The tensor product is $L_0 \oplus L_{-2}$ and the map $L_0 \oplus L_{-2} \to (L_0 \oplus L_{-2})[c^{-1}]$ is an isomorphism in graded degree $0$. So we can identify $R\Gamma(\fX_{\mf{B}},\mathcal{V}\otimes_{\OO_{\fX_{\mf{B}}}}\ol{F}_{cpt}(\wt{P}_{\pi^\vee}))$ with $R\Gamma([\Spec(R_\rho) / \mu_2], j^*\mc{V}\otimes R_{\rho}(1))$. As in the supersingular case, this gives the universal deformation of $\rho$ and we conclude by \cite[Cor.~8.7]{paskunas-image}.
	
	Next, we suppose we are in the non-generic II case. After twisting, we can assume that $\mf{B}$ contains the trivial representation. Suppose $\pi = \pi_\alpha = \Ind_B^G(\omega\otimes\omega^{-1})$. The same argument as in the generic principal series case identifies $R\Gamma(\fX_{\mf{B}},\mathcal{V}\otimes_{\OO_{\fX_{\mf{B}}}}\ol{F}_{cpt}(\wt{P}_{\pi_{\alpha}^\vee}))$ with the universal deformation of $\rho$, a non-split extension of $\omega$ by the trivial character. Now we apply \cite[Cor.~10.72]{paskunas-image}, which shows that $\vc(\wt{P}_{\pi_{\alpha}^\vee})$ has the same description. The other possibility for $\pi$ is $\pi = \Sp$. Here we can follow the strategy of \cite[Remark 10.97]{paskunas-image}, which computes $\vc(\wt{P}_{\Sp^\vee})$ using knowledge of $\vc(\wt{P}_{\pi_{\alpha}^\vee})$ and a short exact sequence \[0 \to \vc(\wt{P}_{\Sp^\vee}) \to \vc(\wt{P}_{\pi_{\alpha}^\vee}) \to N_\omega \to 0 \] where $N_\omega$ is a deformation of $\omega$ to the reducible locus $R_{\rho}/(b_0c,b_1c)$, given by the `lower right' entry of the universal reducible deformation of $\rho$. We have a completely parallel story for our functor: there is a short exact sequence
	\[0 \to L_1 = \ol{F}_{cpt}(\wt{P}_{\Sp^\vee}) \xrightarrow{c} L_{-1}= \ol{F}_{cpt}(\wt{P}_{\pi_{\alpha}^\vee}) \to L_{-1}/cL_1 \to 0 \] which after tensoring with $\mc{V}$ and taking global sections gives a short exact sequence 
	\[0 \to \Gamma(\fX_{\mf{B}},\mc{V}\otimes_{\oo_{\fX_{\mf{B}}}}L_1) \to \rho^{univ} \to N \to 0\] where $N$ is a free rank one module over $R_{\rho}/(b_0c,b_1c)$. Moreover, $N$ comes from the first component $L_1$ in $\mc{V}$. This means that the Galois action on $N$ deforms $\omega$. We deduce that the surjective map $\rho^{univ}\otimes_{R_{\rho}}R_{\rho}/(b_0c,b_1c) \to N$ factors through a surjective map from $N_\omega$. We deduce from the freeness of $N$ that this map is an isomorphism. This finally shows that  $\Gamma(\fX_{\mf{B}},\mc{V}\otimes_{\oo_{\fX_{\mf{B}}}}L_1)$ isomorphic to $\vc(\wt{P}_{\Sp^\vee})$. There are no higher cohomology groups, since $\fX_{\mf{B}}$ is quotient of an affine scheme by a linearly reductive group. 
	
The remaining case is non-generic I. We have $\ol{F}_{cpt}(\pi) = \mc{V}\cong\mc{V}^*$.  So we have
\[
R\Gamma(\fX_{\mf{B}},\mathcal{V}\otimes_{\OO_{\fX_{\mf{B}}}}\mc{V}^*) = \End(\mc{V}) = E,
\] 
the Cayley--Hamilton algebra. The action of $\Gamma_{\Qp}$ is via left multiplication on $E$ (recall that we have a universal representation $\Gamma_{\Qp} \to E^\times$). On the other hand, Pa{\v s}k{\=u}nas shows that $\vc(\wt{P}_{\pi^\vee})$ is a (non-commutative) deformation of a one-dimensional representation of $\Gamma_{\Qp}$ over $k$ to $\wt{E}$, and uses this to produce a map $\oo\lb\mc{G}\rb^{op} \to \wt{E}$ which factors through an isomorphism from $E^{op}$ to $\wt{E}$ \cite[\S9]{paskunas-image}. After twisting, we may assume that the one-dimensional Galois representation is trivial. Then its universal (non-commutative) deformation is given by $\oo\lb \mc{G} \rb$, viewed as a left $\oo\lb\mc{G}\rb^{op}$-module by the right regular action, and with left regular $\Gamma_{\Qp}$-action. We may now identify $\vc(\wt{P}_{\pi^\vee})$ with $E^{op}\otimes_{\oo\lb\mc{G}\rb^{op}}\oo\lb\mc{G}\rb$, with $\Gamma_{\Qp}$-action given by the left regular action on $\oo\lb\mc{G}\rb$. This can in turn be identified with $E$, with the left regular action of $\Gamma_{\Qp}$. 
\end{proof}

When $\pi$ is finite dimensional, one can show that $R\Gamma(\fX_{\mf{B}},\mc{V} \otimes_{\oo_{\fX_{\mf{B}}}} \ol{F}_{cpt}(\pi^\vee))=0$ by direct computation. From this and Proposition \ref{rem:vytas question}, one can deduce that 
\[
\Gamma(\fX_{\mf{B}},\mc{V} \otimes_{\oo_{\fX_{\mf{B}}}} \ol{F}_{cpt}(\sigma^\vee)) = R\Gamma(\fX_{\mf{B}},\mc{V} \otimes_{\oo_{\fX_{\mf{B}}}} \ol{F}_{cpt}(\sigma^\vee)) \cong \vc(\sigma^\vee)
\]
for all $\sigma^\vee \in \mf{C}(\oo)_{\mf{B}}$. In particular, this recovers the renormalized Montr\'eal functor $\vc$ from the categorical embedding as a (spectral) Whittaker coefficient (with extra structure) in the sense of the geometric Langlands program. In keeping with our focus on the discrete functor, we will not give the details of the above assertions for $F_{cpt}$, but instead prove a version relating $F_{disc}$ and the original Montr\'eal functor $\V$. To start with, we give a (partial) reinterpretation of Proposition \ref{rem:vytas question}. Let $E$ be the universal Cayley--Hamilton algebra for a block $\mf{B}$. The canonical isomorphism $V \cong V^\ast \otimes \det(V)$, for any two-dimensional representation $V$, induces an isomorphism $E \to E^{op}$ which makes the diagram
\[
\xymatrix{
\oo \lb \Gamma_{\Qp} \rb \ar[r] \ar[d] & E \ar[d] \\
\oo \lb \Gamma_{\Qp} \rb^{op} \ar[r] & E^{op}
}
\]
commute, where the left vertical map sends $\gamma \in \Gamma_{\Qp}$ to $(\epsilon \zeta)(\gamma)\gamma^{-1}$. 

\begin{corollary}\label{reinterpretation}
Write $P_{\mf{B}}^{inf}$ for the direct sum of the projective envelopes of the Pontrygain duals of the infinite dimensional irreducible representations in $\mf{B}$. When $\mf{B}$ is not supersingular, we have $\End(P_{\mf{B}}^{inf}) = E$, and $\vc(P_{\mf{B}}^{inf})$ is isomorphic to $E$ as a $(\Gamma_{\Qp},E^{op})$-bimodule, where $\Gamma_{\Qp}$ acts on $E$ via the left $E$-action. As a consequence, we have $\vc(P_{\mf{B}}^{inf})^\ast(\epsilon \zeta) \cong E$ as $(\Gamma_{\Qp},E)$-bimodules as well.
\end{corollary}

\begin{proof}
The second statement follows from the first, so it suffices to prove the first statement. When $\mf{B}$ is of type non-generic I, this follows from the last sentence of the proof of Proposition \ref{rem:vytas question}. For the other two cases, it follows from the fact that $\ol{F}_{cpt}(P_{\mf{B}})^{inf} \cong \mc{V} \cong \mc{V}^\ast$ and hence
\[
\vc(\wt{P}_{\mf{B}}^{inf}) \cong R\Gamma(\fX_{\mf{B}},\mathcal{V}\otimes_{\OO_{\fX_{\mf{B}}}}\ol{F}_{cpt}(\wt{P}^{inf}_{\mf{B}})) \cong R\Gamma(\fX_{\mf{B}},\mathcal{V}\otimes_{\OO_{\fX_{\mf{B}}}}\mc{V}^*) = \End(\mc{V}) = E
\]
by Proposition \ref{rem:vytas question}, and one checks that the actions match. 
\end{proof}

Now let $\mf{B}$ be any block and consider the functor
\[
H : \IndCoh(\mf{X}_{\mf{B}}) \to \D^L(E)
\]
given by $H(\mc{F}) = R\Gamma(\mf{X}_{\mf{B}},\mc{V}\otimes_{\oo_{\mf{X}_\mf{B}}}\mc{F})$, with the $E$-action coming from the left $E$-action on $\mc{V}$. Alternatively, we may write the functor as $H(\mc{F}) =  \RHom(\mc{V}^\ast,\mc{F})$. In particular, $H$ commutes with all colimits. 

\begin{lemma}\label{exactness of Whittaker coefficient}
The composition $H \circ F_{disc} : \D(\Mod_{G,\zeta}^{lfin}(\oo)_\mf{B}) \to \D^L(E)$ is t-exact, and hence induces an exact functor $H_0(H \circ F_{disc}) : \Mod_{G,\zeta}^{lfin}(\oo)_\mf{B} \to \LMod(E)$.
\end{lemma}

\begin{proof}
When $\mf{B}$ is supersingular or generic principal series, the individual functors are t-exact and the lemma follows. Assume that $\mf{B}$ is of type non-generic I. Then, by our definition of $F_{disc}$, we may write $H\circ F_{disc}$ as a composition
\[
\D(\Mod_{G,\zeta}^{lfin}(\oo)_\mf{B}) \to \D^L(E) \to \IndCoh(\mf{X}_{\mf{B}}) \to \D^L(E)
\]
and the first functor is t-exact, so it suffices to show that the composition $\D^L(E) \to \IndCoh(\mf{X}_{\mf{B}}) \to \D^L(E)$ is t-exact. This composition is given by the formula
\[
M \to \RHom(\mc{V}^\ast,\mc{V}^\ast \otimes^L_E M) \cong \RHom(\mc{V}^\ast,\mc{V}^\ast) \otimes^L_E M.
\]
By Theorem \ref{thm: good invariants}, $\RHom(\mc{V}^\ast,\mc{V}^\ast) = E$ as an $(E,E)$-bimodule (using the involution $E\cong E^{op}$), so we see that the composition is the identity functor, and hence t-exact.

It remains to treat the case when $\mf{B}$ is of type non-generic II (as always, we twist so that $\zeta$ is trivial). In this case $H$ is t-exact, though $F_{disc}$ is not. However, by Proposition \ref{computation of irreducibles non-gen II}, $H(F_{disc}(\pi_\alpha))$ and $H(F_{disc}(\mathrm{Sp}))$ are concentrated in degree $0$, and (by a short computation) $H(F_{disc}(\mbf{1}_G))=0$. Thus all the irreducibles get sent to complexes concentrated in degree $0$, and $H\circ F_{disc}$ commutes with all colimits. By the argument in the proof of Proposition \ref{nongen 1 is abelian!}, $H\circ F_{disc}$ is t-exact, as desired.
\end{proof}

By composing $H_0(H\circ F_{disc})$ with the map $\LMod(E) \to \Mod_{\Gamma_{\Qp}}(\oo)$ coming from $\oo\lb \Gamma_{\Qp} \rb \to E$, we get an exact functor $\mbf{W} : \Mod^{lfin}_{G,\zeta}(\oo)_{\mf{B}} \to \Mod_{\Gamma_{\Qp}}^{disc}(\oo)$, where $\Mod_{\Gamma_{\Qp}}^{disc}(\oo)$ is the category of discrete $\oo$-modules with a continuous $\Gamma_{\Qp}$-action. We may extend the Montr\'eal functor $\V : \Mod^{fin}_{G,\zeta}(\oo)_{\mf{B}} \to \Mod_{\Gamma_{\Qp}}^{fin}(\oo)$ to an exact functor $\V : \Mod^{lfin}_{G,\zeta}(\oo)_{\mf{B}} \to \Mod_{\Gamma_{\Qp}}^{disc}(\oo)$ by taking the $\Ind$-extension. Before proceeding, we note that the equivalence
\begin{equation}\label{eq: discrete equivalence}
\Mod_{G,\zeta}^{lfin}(\oo)_{\mf{B}} \cong \LMod_{disc}(\wt{E}_{\mf{B}})
\end{equation}
is given by the functors $\sigma \mapsto P_{\mf{B}} \otimes_{\oo \lb G \rb} \sigma$ and $M \mapsto \Hom_{\wt{E}_{\mf{B}}}(P_{\mf{B}},M)$. We then have the following comparison theorem.

\begin{theorem}
For any $\sigma \in \Mod_{G,\zeta}^{lfin}(\oo)_{\mf{B}}$, we have $\V(\sigma) \cong \mbf{W}(\sigma)$.  
\end{theorem}

\begin{proof}
We will use the equivalence (\ref{eq: discrete equivalence}) to view $\V$ and $\mbf{W}$ as functors on $\LMod_{disc}(\wt{E}_{\mf{B}})$ whenever convenient (and similarly for $\vc$). We start with the case when $\mf{B}$ is supersingular. Using notation as in the proof of Proposition \ref{rem:vytas question}, the functor $\mbf{W}$ is given by
\[
M \mapsto R\Gamma(\mf{X}_{\mf{B}},\mc{V}\otimes_{\oo_{\mf{X}_{\mf{B}}}} \oo_{\mf{X}_{\mf{B}}}(1) \otimes_R M) = R^2\otimes_R M,
\]
where $R^2$ is the universal deformation of $\rho_{\mf{B}}$. By Proposition \ref{rem:vytas question}, we get $\mbf{W}(M) \cong \vc(P_{\mf{B}}) \otimes_R M$, so it remains to show that $\V(M) \cong \vc(P_{\mf{B}}) \otimes_R M$. By definition and \cite[Lem. 5.53]{paskunas-image}, we have
\[
\V(M) = \vc(M^\vee)^\vee (\epsilon \zeta) \cong (M^\vee \wh{\otimes}_R \vc(P_{\mf{B}}))^\vee (\epsilon \zeta) \cong \Hom_R(\vc(P_{\mf{B}}),M) (\epsilon \zeta) \cong \vc(P_{\mf{B}})^\ast (\epsilon \zeta) \otimes_R M,
\]
and the result then follows since $\vc(P_{\mf{B}})^\ast (\epsilon \zeta) \cong  \vc(P_{\mf{B}})$. 

The proofs of the remaining cases are similar. Assume first that $\mf{B}$ is a generic principal series or non-generic I block, and identify $\wt{E}_{\mf{B}}$ and $E$. In both cases, arguing as in the case of non-generic I in the proof of Lemma \ref{exactness of Whittaker coefficient} and using Corollary \ref{reinterpretation}, we have
\[
\mbf{W}(M) \cong \RHom(\mc{V}^\ast,\mc{V}^\ast) \otimes_{E} M \cong \vc(P_{\mf{B}})^\ast (\epsilon \zeta) \otimes_E M.
\]
As in the supersingular case, one then computes that $\V(M) \cong \vc(P_{\mf{B}})^\ast (\epsilon \zeta) \otimes_E M$ to conclude. 

This leaves the non-generic II case. The projective object $P_{\mf{B}}^{inf} = P_{\mathrm{Sp}^\vee} \oplus P_{\pi^\vee_\alpha}$ corresponds to $N := \Hom(P_\mf{B},P_\mf{B}^{inf}) \in \RMod_{cpt}(\wt{E}_{\mf{B}})$, which carries a left action of $E$. Consider the Serre subcategory of $\Mod_{G,\zeta}^{fin}(\oo)_{\mf{B}}$ consisting of finite-dimensional representations, and take its closure $\mc{S}$ under filtered colimits in $\Mod_{G,\zeta}^{lfin}(\oo)_{\mf{B}}$. Under the equivalence $\Mod_{G,\zeta}^{lfin}(\oo)_{\mf{B}} \cong \LMod_{disc}(\wt{E}_{\mf{B}})$, the quotient category $\Mod_{G,\zeta}^{lfin}(\oo)_{\mf{B}}/\mc{S}$ corresponds to $\LMod_{disc}(E)$ under the functor
\[
M \mapsto N \otimes_{\wt{E}_{\mf{B}}} M
\]
by the dual of \cite[Lem. 10.84, Cor. 10.85]{paskunas-image}. Since $\V$ and $\mbf{W}$ both kill $\mbf{1}_G$ (see the proof of Lemma \ref{exactness of Whittaker coefficient} for $\mbf{W}$), they factor through $\LMod_{disc}(E)$. The proof that $\V \cong \mbf{W}$ then follows the same pattern as above: since $\V$ and $\mbf{W}$ factor through $\LMod_{disc}(E)$, we may treat them as functors on $\LMod_{disc}(E)$ and conflate $\LMod_{disc}(E)$ with its image in $\LMod_{disc}(\wt{E}_{\mf{B}})$ under the right adjoint
\[
M \mapsto \Hom^{cts}_E(N,M).
\]
Because $N$ is a finitely generated $E$-module, this functor commutes with filtered colimits. Then one computes that, for $M\in \LMod_{fin}(E)$,
\[
\mbf{W}(M) \cong \Hom(\mc{V}^\ast,X_{\mf{B}}^\ast \otimes_{\wt{E}_\mf{B}} \Hom_E(N,M)) \cong \Hom(\mc{V}^\ast,X_{\mf{B}}^\ast \otimes_{\wt{E}_\mf{B}} \Hom_E(N,E) \otimes_E M) \cong
\]
\[
\Hom(\mc{V}^\ast,\mc{V}^\ast \otimes_E M) \cong \Hom(\mc{V}^\ast,\mc{V}^\ast) \otimes_E M,
\]
and the formula extends to all $M\in \LMod_{disc}(E)$ since both sides commute with filtered colimits. As before, one computes that $\V(M) \cong \vc(P_{\mf{B}})^\ast (\epsilon \zeta) \otimes_E M$ for $\LMod_{disc}(E)$, and then Corollary \ref{reinterpretation} finishes the proof as before.
\end{proof}

\subsection{Recollections on $p$-arithmetic homology}\label{subsec: p-arithmetic hom} In this subsection we will recall $p$-arithmetic (co)homology in the adelic setting and its comparison with arithmetic homology from \cite{tarrach2022sarithmetic}, and prove a formula computing completed homology as a $p$-arithmetic homology group.

Let $\mbf{G}$ be a connected reductive group over $\Q$. In this subsection only, we set $G= \mbf{G}(\Qp)$ and we let $X_p$ be the Bruhat--Tits building of $G$ over $\Qp$. We recall a few facts about $X_p$ that we will need. First, $X_p$ carries a left action of $G$ and a $G$-invariant metric $d$; also, $X_p$ is contractible and any two points in $X_p$ are connected by a unique geodesic\footnote{We recall that if $(X,d)$ is a metric space and $x,y\in X$, then a geodesic from $x$ to $y$ is an isometric embedding $f : [0,d(x,y)] \to X$ satisfying $f(0)=x$ and $f(d(x,y))=y$.} \cite[\S 2.5]{bruhat-tits}. In particular, for $a,b\in X_p$, we may consider the renormalized geodesic $j_{a,b} : [0,1] \to X_p$ from $a$ to $b$. Finally, given a compact subgroup $K_p \sub G$, there is a point $\alpha \in X_p$ which is fixed by all elements of $K_p$.

We will also need some considerations at $\infty$. Let $\mbf{G}(\R)^+$ denote the identity component of $\mbf{G}(\R)$ and set $\mbf{G}(\Q)^+ = \mbf{G}(\Q) \cap \mbf{G}(\R)^+$. We let $K_\infty \sub \mbf{G}(\R)^+$ be a maximal compact subgroup and let $\mbf{A}$ be the maximal $\Q$-split torus in the center of $\mbf{G}$. Set $X_\infty = \mbf{G}(\R)^+/(\mbf{A}(\R)^+ K_\infty)$, and let $\ol{X}_\infty$ be the Borel--Serre bordification of $X_\infty$ \cite{borel-serre}, which carries a left action by $\mbf{G}(\Q)^+$. Given a compact open subgroup $K^p \sub \mbf{G}(\A^{p,\infty})$, we define
\[
\mc{X} := \mbf{G}(\Q)^+ \backslash X_\infty \times \mbf{G}(\A^\infty) / K^p, \,\,\,\,\, \ol{\mc{X}} := \mbf{G}(\Q)^+ \backslash \ol{X}_\infty \times \mbf{G}(\A^\infty) / K^p
\]
and
\[
\mc{X}_p := \mbf{G}(\Q)^+ \backslash X_\infty \times X_p \times \mbf{G}(\A^\infty) / K^p, \,\,\,\,\, \ol{\mc{X}}_p := \mbf{G}(\Q)^+ \backslash \ol{X}_\infty \times X_p \times \mbf{G}(\A^\infty) / K^p.
\]
Here we equip $\mbf{G}(\A^\infty)$ with the discrete topology rather than its locally profinite topology, so that the maps $X_\infty \times \mbf{G}(\A^\infty) \to \mc{X}$, etc., are all covering maps. The action of $\mbf{G}(\Q)^+$ is always diagonal (from the left) and $K^p$ acts by right translation on $\mbf{G}(\A^\infty)$ and trivially on the other components. We remark that $\mc{X}$, $\ol{\mc{X}}$, $\mc{X}_p$ and $\ol{\mc{X}}_p$ all carry right actions of $G$, induced by right translation on $\mbf{G}(\A^\infty)$. If $Y$ is any topological space, we let $C_\bu (Y)$ denote the complex of singular chains of $Y$. Since $\ol{X}_\infty \setminus X_\infty$ is the boundary of the topological manifold with boundary $\ol{X}_\infty$, the inclusion $X_\infty \to \ol{X}_\infty$ is a homotopy equivalence. It follows that $C_\bu (\mc{X}) \to C_\bu (\ol{\mc{X}})$ and $C_\bu (\mc{X}_p) \to C_\bu (\ol{\mc{X}}_p)$ are $G$-chain homotopy equivalences. Moreover, they are also equivariant for the action of Hecke operators away from $p$. Let us indicate this (standard) construction on $C_\bu(\mc{X})$; the actions on the other complexes are similar. We may think of $\mc{X}$ as the quotient of $\mc{X}^\prime := \mbf{G}(\Q)^+ \backslash X_\infty \times \mbf{G}(\A^\infty)$ by the free action of $K^p$. The natural map
\[
C_\bu(\mc{X}^\prime) \otimes_{\Z[K^p]} \Z \to C_\bu(\mc{X})
\] 
is then an isomorphism\footnote{In general, if $X$ is a topological space with a free right action of a discrete group $K$, then $C_\bu(X/K) = C_\bu(X) \otimes_{\Z[K]}\Z$. We will use this without further comment.}, and since $C_\bu(\mc{X}^\prime)$ carries a right action of $\mbf{G}(\A^\infty)$ we get a (right) Hecke action on $C_\bu(\mc{X})$ by the standard recipe, cf.\ \cite[Lem.~2.6.1]{tarrach2022sarithmetic}.

Let $K_p \sub G$ be a compact open subgroup. We recall the construction of a Hecke- and $K_p$-equivariant chain homotopy equivalence between $C_\bu(\mc{X})$ and $C_\bu(\mc{X}_p)$ from \cite[\S 5.2]{tarrach2022sarithmetic}. First, we have the projection map
\[
f : X_\infty \times X_p \times \mbf{G}(\A^\infty) \to X_\infty \times \mbf{G}(\A^\infty),
\]
which is $\mbf{G}(\Q)^+ \times \mbf{G}(\A^\infty)$-equivariant. Now choose $\alpha \in X_p$ which is fixed by all elements of $K_p$, and consider the map
\[
h_\alpha : X_\infty \times \mbf{G}(\A^\infty) \to X_\infty \times X_p\times \mbf{G}(\A^\infty)
\]
given by $h_\alpha(z,g) = (z, g_p \alpha, g)$, where $g_p$ is the $p$-component of $g$. One checks directly that this is $\mbf{G}(\Q)^+ \times \mbf{G}(\A^{p,\infty}) \times K_p$-equivariant. We see directly that $f \circ h_\alpha$ is the identity. Moreover, the map 
\[
H_\alpha : X_\infty \times X_p \times \mbf{G}(\A^\infty) \times [0,1] \to X_\infty \times X_p \times \mbf{G}(\A^\infty)
\]
given by $H_\alpha(z,q,g,t) = (z,j_{q,\alpha}(t),g)$ is a $\mbf{G}(\Q)^+ \times \mbf{G}(\A^{p,\infty}) \times K_p$-equivariant homotopy from the identity to $h_\alpha \circ f$. It follows that $f$ induces a Hecke- and $K_p$-equivariant chain homotopy equivalence from $C_\bu(\mc{X}_p)$ to $C_\bu(\mc{X})$, with inverse (induced by) $h_\alpha$. We can then define $p$-arithmetic (co)homology.

\begin{definition} Let $M$ be a complex of left $G$-modules, and let $N$ be a complex of right $G$-modules.
\begin{enumerate}
\item We define the \emph{$p$-arithmetic homology} of $M$ to be the homology $H_\ast(K^p,M)$ of the complex $C_\bu(K^p,M) := C_\bu(\mc{X}_p) \otimes^L_{\Z[G]}M$. 

\item We define the \emph{$p$-arithmetic cohomology} of $N$ to be the cohomology $H^\ast(K^p,N)$ of the complex $C^\bu(K^p,N) := \RHom_{\Z[G]}(C_\bu(\mc{X}_p),N)$.
\end{enumerate} 

\end{definition}

For completeness, we also recall the definition of arithmetic (co)homology.

\begin{definition} Let $M$ be a complex of left $K_p$-modules, and let $N$ be a complex of right $K_p$-modules. Set $K=K^p K_p$.
\begin{enumerate}
\item We define the \emph{arithmetic homology} of $M$ to be the homology $H_\ast(K,M)$ of the complex $C_\bu(K,M) := C_\bu(\mc{X}) \otimes^L_{\Z[K_p]}M$. 

\item We define the \emph{arithmetic cohomology} of $N$ to be the cohomology $H^\ast(K,N)$ of the complex $C^\bu(K,N) := \RHom_{\Z[K_p]}(C_\bu(\mc{X}),N)$.
\end{enumerate} 

\end{definition}

We make no assumption on the action of $G$ on $\mc{X}_p$, or $K_p$ on $\mc{X}$, being free. If $G$ acts freely on $\mc{X}_p$, then $C_\bu(\mc{X}_p)$ is a (bounded above) complex of free $\Z[G]$-modules (this will be true for $K^p$ sufficiently small). Similarly, if $K_p$ acts freely on $\mc{X}$, then $C_\bu(\mc{X})$ is a (bounded above) complex of free $\Z[K_p]$-modules. When the actions are free, we will use $C_\bu(K^p,M)$ to denote the actual complex $C_\bu(\mc{X}_p) \otimes_{\Z[G]}M$, and similarly for the other notations. Continue to set $K = K^p K_p$. We have the following comparison, which is a special case of \cite[Prop.~5.2.2]{tarrach2022sarithmetic}.

\begin{proposition}\label{shapiro}
Let $M$ be a complex of left $K_p$-modules and let $N$ be a complex of right $K_p$-modules. Then we have canonical Hecke-equivariant isomorphisms $C_\bu(K,M) \cong C_\bu(K^p, \Z[G]\otimes_{\Z[K_p]}M)$ and $C^\bu(K,N) \cong C^\bu(K^p, \Hom_{\Z[K_p]}(\Z[G],N))$ in the derived category of abelian groups (note that $\Z[G]$ is free over $\Z[K_p]$).
\end{proposition}

\begin{proof}
These follow from the definitions, the chain homotopy equivalence $C_\bu(\mc{X}_p) \to C_\bu(\mc{X})$, and standard manipulations/adjunctions.
\end{proof}

In particular, all arithmetic (co)homology groups occur naturally as $p$-arithmetic (co)homology groups. Before discussing completed homology, we will discuss finiteness properties of arithmetic (co)homology\footnote{$p$-arithmetic (co)homology satisfies similar finiteness properties by the main result of \cite{borel-serre-II}.}. Choose $K_p$ small enough that the action on $\mc{X}$ is free. Then $\ol{\mc{X}}$ is a compact topological manifold with boundary, and hence may be triangulated. We fix such a triangulation. Refining it if necessary, we pull it back to $\mc{X}$ to obtain a $K_p$-equivariant triangulation of $\mc{X}$. The corresponding complex $C_\bu^{BS}(\mc{X})$ of simplicial chains is a bounded complex whose terms are finite free $\Z[K_p]$-modules, and it is $K_p$-equivariantly chain homotopic to $C_\bu(\mc{X})$. We fix a $K_p$-equivariant chain homotopy equivalence $C_\bu(\mc{X}) \to C^{BS}_\bu(\mc{X})$. Given a left $K_p$-module, we write $C_\bu^{BS}(K,M) := C_\bu^{BS}(\mc{X}) \otimes_{\Z[K_p]}M$. The formation of $C_\bu^{BS}(K,M)$ is obviously functorial in $M$. We record the following lemma:

\begin{lemma}\label{homology and inverse limits}
Assume that $K_p$ acts freely on $\mc{X}$. Let $(M_i)_{i\in I}$ be an inverse system of left $K_p$-modules with inverse limit $M$. Then the canonical map $C_\bu(K,M) \to \varprojlim_i C_\bu(K,M_i)$ is a chain homotopy equivalence. Moreover, if the $M_i$ are finite (as sets), then the induced map $H_\ast(K,M) \to \varprojlim_i H_\ast(K,M_i)$ is an isomorphism.
\end{lemma}

\begin{proof}
Using the fixed chain homotopy equivalence $C_\bu(\mc{X}) \to C^{BS}_\bu(\mc{X})$ we have a commutative square 
\[
\xymatrix{
C_\bu(K,M) \ar[r] \ar[d] &  \varprojlim_i C_\bu(K,M_i) \ar[d] \\
C_\bu^{BS}(K,M) \ar[r] & \varprojlim_i C^{BS}_\bu(K,M_i).
}
\]
The vertical maps are chain homotopy equivalences. The lower horizontal map is an isomorphism of complexes, since the terms in $C^{BS}_\bu(\mc{X})$ are finite free $\Z[K_p]$-modules. It follows that the upper horizontal map is a chain homotopy equivalence, as desired. To prove the last part, note that we have $H_\ast( \varprojlim_i C^{BS}_\bu(K,M_i) ) = \varprojlim_i H_\ast(C^{BS}_\bu(K,M_i))$ since the terms in the complexes $C^{BS}_\bu(K,M_i)$ are finite (as sets).
\end{proof}

Let us now discuss completed homology. By definition, completed homology for $\mbf{G}$ with tame level $K^p$ (and $\Zp$-coefficients) is 
\[
\wt{H}_\ast(K^p) := \varprojlim_{K^\prime_p} H_\ast(K^p K_p^\prime,\Zp),
\]
where $K_p^\prime$ runs over all compact open subgroups of $G$. It is a right $\Zp \lb G \rb $-module. In fact, our goal here is to prove that $\wt{H}_\ast(K^p) \cong H_\ast(K^p,\Zp \lb G \rb )$ as right $\Zp \lb G \rb $-modules, with the right $\Zp \lb G \rb $-module structure on $H_\ast(K^p,\Zp \lb G \rb )$ induced from the right $\Zp \lb G \rb $-module structure on $\Zp \lb G \rb $ itself. This is a $p$-arithmetic version of a theorem of Hill \cite{hill}, and is due to one of us (C.J.) and Guillem Tarrach.

From now on, fix $K_p \sub G$ acting freely on $\mc{X}$. We will only consider compact open normal subgroups $K_p^\prime \sub K_p$; these are cofinal, so it suffices to consider only these. Write $K^\prime = K^p K_p^\prime$. The $G$-action on $\wt{H}_\ast(K^p)$ may be described as follows. Let $g\in G$. To simplify notation, if $H\sub \mbf{G}(\A^\infty)$, we will set $^g H := g^{-1}Hg$. The action of $g$ on $\mc{X}$ induces isomorphisms
\begin{equation}\label{eq: G-action 1}
C_\bu(K^\prime, \Zp) \to C_\bu(^g K^\prime,\Zp)
\end{equation}
given by the formula $\sigma \otimes \lambda \mapsto \sigma g \otimes \lambda$. Taking the inverse limit at the level of homology, we get the $G$-action on $\wt{H}_\ast(K^p)$. We may rewrite the left hand side of (\ref{eq: G-action 1}) as
\[
C_\bu(K^\prime, \Zp) = C_\bu(\mc{X}) \otimes_{\Z[K^\prime_p]} \Zp \cong C_\bu(\mc{X}) \otimes_{\Z[K_p]} \Zp[K_p/K_p^\prime] = C_\bu (K,\Zp[K_p/K_p^\prime])
\]
and similarly for the right hand side. The action of $g$ from (\ref{eq: G-action 1}) then becomes an isomorphism 
\[
C_\bu (K,\Zp[K_p/K_p^\prime]) \to C_\bu (^g K,\Zp[^g K_p/^g K_p^\prime])
\]
given by $\sigma \otimes k \mapsto \sigma g \otimes g^{-1}kg$. Now consider the isomorphism
\begin{equation}\label{eq: G-action 2}
C_\bu (K,\Zp \lb K_p \rb ) \to C_\bu (^g K,\Zp \lb ^g K_p \rb )
\end{equation}
given by $\sigma \otimes \mu \mapsto \sigma g \otimes g^{-1}\mu g$, for $\mu \in \Zp \lb K_p \rb $. Note that if $g \in K_p$, then this is equal to the action of $K_p$ induced from the right $\Zp \lb K_p \rb$-module structure on $\Zp \lb K_p \rb$. We have a commutative square
\[
\xymatrix{
C_\bu (K,\Zp \lb K_p \rb ) \ar[r] \ar[d] &  \varprojlim C_\bu (K,\Zp[K_p/K_p^\prime]) \ar[d] \\
C_\bu (^g K,\Zp \lb ^g K_p \rb ) \ar[r] & \varprojlim C_\bu (^g K,\Zp[^g K_p/^g K_p^\prime])
}
\]
where the horizontal maps are chain homotopy equivalences by Lemma \ref{homology and inverse limits}; the Lemma also gives us that $H_\ast(K,\Zp \lb K_p \rb ) \cong \wt{H}_\ast(K^p)$. By Proposition \ref{shapiro}, we have chain homotopy equivalences $C_\ast(K^p,\Zp \lb G \rb) \cong C_\bu (K,\Zp \lb K_p \rb )$ and $C_\ast(K^p,\Zp \lb G \rb) \cong C_\bu (^g K,\Zp \lb ^g K_p \rb )$. Tracing through the definitions, it is tedious but straightforward to show that the `action' of $g$ from (\ref{eq: G-action 2}) is the natural right action of $g$ on $C_\ast(K^p,\Zp \lb G \rb)$ (up to chain homotopy equivalence). We state our conclusion in the following result.

\begin{proposition}
The complex $C_\bu (K^p,\Zp \lb G \rb)$ with its natural right $\Zp \lb G \rb $-module structure (and Hecke action) computes $\wt{H}_\ast(K^p)$ with its right $ \Zp \lb G \rb $-module structure (and Hecke action).
\end{proposition}

\begin{remark}\label{derived tensor = underived tensor}
We note that $C_\bu(\mc{X}_p) \otimes_{\Z[G]} \Zp \lb G \rb = C_\bu(\mc{X}_p) \otimes^L_{\Z[G]} \Zp \lb G \rb$, regardless of whether $G$ acts freely on $\mc{X}_p$ or not, so it makes sense to talk of $C_\ast(K^p,\Zp \lb G \rb)$ as a specific complex and not `just' an object in a derived category. Indeed, for $K_p$ as above, we see that
\[
C_\bu(\mc{X}_p) \otimes^L_{\Z[G]} \Zp \lb G \rb \cong C_\bu(\mc{X}_p) \otimes^L_{\Z[G]} \Z[G] \otimes^L_{\Z[K_p]} \Zp \lb K_p \rb \cong C_\bu(\mc{X}_p) \otimes^L_{\Z[K_p]} \Zp \lb K_p \rb.
\]
The right hand side is equal to $C_\bu(\mc{X}_p) \otimes_{\Z[K_p]} \Zp \lb K_p \rb$ since $K_p$ acts freely on $\mc{X}_p$, and this is just $C_\bu(\mc{X}_p) \otimes_{\Z[G]} \Zp \lb G \rb$.
\end{remark}

Let us now work over $\oo$. In light of the remark above, we may set
\[
\wt{C}_\bu := C_\bu(\mc{X}_p) \otimes_{\Z[G]} \oo \lb G \rb;
\]
this computes completed homology $\wt{H}_\ast(K^p,\oo) = \wt{H}_\ast(K^p) \otimes_{\Zp}\oo$ with coefficients in $\oo$. By the construction in \cite[\S 2.1.10]{gee-newton}, the unramified Hecke action on $\wt{C}_\bu$, viewed as endomorphisms in the derived category, factors through the action of a `big' Hecke algebra $\mb{T} = \mb{T}(K^p)$. The following result shows that completed homology is universal for $p$-arithmetic (co)homology of $\oo \lb G \rb$-modules.

\begin{proposition}\label{completed homology is universal}
Let $M$ be a complex of left $\oo \lb G \rb$-modules, and let $N$ be a complex of right $\oo \lb G \rb$-modules.
\begin{enumerate}
\item We have $C_\bu(K^p,M) \cong \wt{C}_\bu \otimes^L_{\oo \lb G \rb }M$. Moreover, the unramified Hecke action factors through a homomorphism $\mb{T} \to \End_{D(\Mod(\oo))}(C_\bu(K^p,M))$. 

\item We have $C^\bu(K^p,N) \cong \RHom_{\oo \lb G \rb }(\wt{C}_\bu ,N)$. Moreover, the unramified Hecke action factors through a homomorphism $\mb{T} \to \End_{D(\Mod(\oo))}(C^\bu(K^p,N))$.
\end{enumerate}
\end{proposition}

\begin{proof}
We prove the first part; the second is similar. The formula for $C_\bu(K^p,M)$ follows from the computation
\[
C_\bu(K^p,M) = C_\bu(\mc{X}_p) \otimes_{\Z[G]}^L M \cong (C_\bu(\mc{X}_p) \otimes_{\Z[G]} \oo \lb G \rb ) \otimes^L_{\oo \lb G \rb }M
\]
(which relies on Remark \ref{derived tensor = underived tensor}) and the statement about the Hecke action follows directly from the formula.
\end{proof}

\subsection{The local-global formula}

We will now prove a formula for $p$-arithmetic homology of modular curves as the global sections of a sheaf on the moduli stack of global Galois representations. Let $r : \Gamma_{\Q} \to \GL_2(\Fpbar)$ be a continuous representation. If $\ell$ is any prime (including $p$), we write $r_\ell$ for $r |_{\Gamma_{\Q_\ell}}$. We assume that $r$ satisfies the following hypotheses: 
\begin{enumerate}

\item $\det r = \omega$;

\item $r_p$ is indecomposable, and not a twist of an extension of the form $0 \to \omega  \to r_p' \to \mbf{1} \to 0$;

\smallskip 

\item if $r_\ell$ is ramified for some $\ell \neq p$, then $\ell$ is not a vexing prime in the sense of \cite{diamond-ext};

\smallskip

\item $r |_{\Gamma_{\Q(\zeta_p)}}$ has adequate image (in particular, $r$ is irreducible), in the sense of \cite[Defn.~2.3]{thorne-adequate}.
\end{enumerate}

In particular, $r$ is odd and hence modular \cite{khare-wintenberger-1,khare-wintenberger-2,kisin-2-adic}, and we are in the setting of \cite[\S 7]{ceggps2} (except that we have a fixed determinant) and \cite[\S 5]{gee-newton} (except that we allow twists of  extensions of $\omega$ by $\mbf{1}$). The reason for our local assumptions is so that $r_p$ admits a universal deformation ring and we can work over a formally smooth quotient of the universal lifting ring for $r_l$ at ramified primes $l \ne p$.

Let $N$ be the prime-to-$p$ Artin conductor of $r$. We let $\wt{R}_p$ denote the deformation ring of $r_p$ with determinant $\varepsilon$. We let $\wt{R}_{\Q,N}$ denote the deformation ring of deformations of $r$ with determinant $\varepsilon$ which are minimally ramified at all primes $\ell \neq p$. We remark that, by \cite[Thm.~1]{allen-calegari}, the natural map $R_p \to R_{\Q,N}$ is finite.

We will consider arithmetic and $p$-arithmetic (co)homology for $\mbf{G}^{ad}=\PGL_{2/\Q}$ as recalled in \S \ref{subsec: p-arithmetic hom}, with tame level $K^p_1(N) \subseteq \PGL_2(\Z^p)$ (consisting of matrices whose bottom row is congruent to $( \begin{smallmatrix} 0 & 1 \end{smallmatrix})$ modulo $N$ and modulo center). When setting out our conventions and simplifications of notation we will only explicitly mention homology, but the analogous conventions will be in place for cohomology as well. We write $G^{ad}$ for $\PGL_2(\Qp)$. We will only consider $p$-arithmetic homology of left $\oo \lb G^{ad} \rb$-modules $\sigma$ (or complexes of such), and to simplify the notation we will write $H_\ast(N,\sigma)$ for $H_\ast(K^p_1(N),\sigma)$. Similarly, we write $\wt{H}_\ast(N,\oo)$ for completed homology of tame level $K^p_1(N)$ and $\oo$-coefficients. Consider the big Hecke algebra $\mb{T}$ as in \cite[\S 2.1.10]{gee-newton}. The representation $r$ defines a maximal ideal of $\mb{T}$, which we will denote by $\m$, and we have a surjection $R_{\Q,N} \to \T_{\m}$. The localized completed homology $\wt{H}_{\ast}(N,\oo)_{\m}$ is concentrated in degree $1$, and is a faithful $\T_{\m}$-module \cite[Lem.~3.4.20]{gee-newton}. Since the homology is isomorphic to \'etale homology, we also have an action of $\Gamma_{\Q}$ on $\wt{H}_{\ast}(N,\oo)$ and $\wt{H}_{\ast}(N,\oo)_{\m}$. Let $r^{univ} : \Gamma_{\Q} \to \GL_2(R_{\Q,N})$ denote the universal deformation. As in \S \ref{subsec: local considerations}, we let $\pi$ be the admissible $G^{ad}$-representation corresponding to $r_p$, and we let $P$ be the projective envelope of $\pi^\vee$. We then have the following description of completed homology.

\begin{theorem}\label{emerton-style local-global}
We have an isomorphism $\wt{H}_1(N,\oo)_{\m} \cong P \otimes_{R_p} r^{univ}$ of $R_{\Q,N}[G^{ad}\times \Gamma_{\Q}]$-modules. 
\end{theorem}

\begin{proof}
This is essentially \cite[Thm.~7.4]{ceggps2}, but with fixed central character and the added observation that $R_p \to R_{\Q,N}$ is finite, so we do not need a completed tensor product. The difference is that, in the setting of \emph{loc.\ cit.}\ (but with our notation), $\det r = \omega^{-1}$. Thus, our deformation problem is obtained from theirs by tensoring with $\varep$. If we denote their universal deformation by $\rho^{univ}$, then this means that 
\[
r^{univ} =  \rho^{univ}\otimes \varep = (\rho^{univ})^\ast,
\]
where $(-)^\ast$ denotes the $R_{\Q,N}$-linear dual, since $\det \rho^{univ} = \varep^{-1}$ (and we are dealing with two-dimensional representations). This explains why the dual occurs in \emph{loc.\ cit.}\ but not in our formulation.
\end{proof}

To go further, we also need the following result, which appears to be new when $r_p$ is a twist of an extension of $\omega$  by $\mbf{1}$.

\begin{proposition}\label{big R=T}
The map $R_{\Q,N} \to \mb{T}_{\m}$ is an isomorphism of complete intersection rings, and both rings have Krull dimension $3$.
\end{proposition}

\begin{proof}
When $r_p$ is not a twist of an extension of $\omega$ by $\mbf{1}$, this follows from \cite[Prop.~5.1.4]{gee-newton}. We give a different proof that works uniformly for all cases. For this, we need the output of the patching construction from \cite[\S 7]{ceggps2}, so we recall this briefly. At the end of the patching procedure we have 
\begin{itemize}
\item Rings $\oo_\infty = \oo \lb y_1,\dots, y_g \rb$ and $R_\infty = R_p \lb x_1,\dots, x_d \rb$ and a local ring map $ \oo_\infty \to R_\infty$;

\smallskip 

\item A surjection $R_\infty / \mf{a}R_\infty \to R_{\Q,N}$, where $\mf{a} = (y_1,\dots,y_g) \sub \oo_\infty$;

\smallskip 

\item An $R_\infty[G^{ad}]$-module $M_\infty$ which lies in $\mf{C}(\oo)_{\mf{B}}$;

\smallskip

\item The action of $R_\infty/\mf{a}R_\infty$ on $M_\infty/\mf{a}M_\infty$ factors through $\mb{T}_{\m}$, and we have an isomorphism \[(M_\infty / \mf{a} M_\infty) \otimes_{R_{\Q,N}}(\rho^{univ})^\ast \cong \wt{H}_1(N,\oo)_{\m}\] of $R_{\Q,N}[G^{ad} \times \Gamma_{\Q}]$-modules.
\end{itemize}
Moreover, by the proof of \cite[Thm.~7.4]{ceggps2}, $M_\infty \cong P \wh{\otimes}_{R_p}R_\infty$ as $R_\infty[G^{ad}]$-modules. Also, if $K_p$ is a sufficiently small compact open subgroup of $G^{ad}$ then $M_\infty$ is a finitely generated free $\oo_\infty \lb K_p \rb$-module (this is essentially \cite[Prop.~2.10]{ceggps1}), and hence a flat $\oo_\infty$-module. We now prove that (the images of) $y_1,\dots,y_g$ form a regular sequence in $R_\infty$. For this, we need to check that the augmented Koszul complex $K^{aug}_\bu(y_1,\dots,y_g,R_\infty)$ is acyclic. Consider the (non-augmented) Koszul complex $K_\bu(y_1,\dots,y_g,M_\infty)$. Since $y_1,\dots,y_g$ form a regular sequence in $\oo_\infty$, $K_\bu(y_1,\dots,y_g,M_\infty)$ computes $(\oo_\infty/\mf{a}) \otimes_{\oo_\infty}^L M_\infty$. Since $M_\infty$ is $\oo_\infty$-flat, we conclude that the augmented Koszul complex $K^{aug}_\bu(y_1,\dots,y_g,M_\infty)$ is acyclic. Now apply $\Hom_{\mf{C}(\oo)}(P,-)$ to $K^{aug}_\bu(y_1,\dots,y_g,M_\infty)$; this gives us the augmented Koszul complex
\[
K^{aug}_\bu(y_1,\dots,y_g,\Hom_{\mf{C}(\oo)}(P,M_\infty))
\]
of $\Hom_{\mf{C}(\oo)}(P,M_\infty)$. Since $\Hom_{\mf{C}(\oo)}(P,-)$ is exact, this complex is acyclic. Moreover, we have 
\[
\Hom_{\mf{C}(\oo)}(P,M_\infty) \cong \Hom_{\mf{C}(\oo)}(P,P\wh{\otimes}_{R_p}R_\infty) \cong R_\infty
\]
as $R_\infty$-modules, using that $\End_{\mf{C}(\oo)}(P)=R_p$. So this is the augmented Koszul complex for $R_\infty$, and hence $y_1,\dots,y_g$ form a regular sequence in $R_\infty$ as desired.  Since $\Hom_{\mf{C}(\oo)}(P,M_\infty)\cong R_\infty$, we have $\Hom_{\mf{C}(\oo)}(P,M_\infty) \otimes_{R_\infty} R_\infty / \mf{a}R_\infty \cong R_\infty /\mf{a}R_\infty$. But 
\[
\Hom_{\mf{C}(\oo)}(P,M_\infty) \otimes_{R_\infty} R_\infty / \mf{a}R_\infty \cong \Hom_{\mf{C}(\oo)}(P,M_\infty / \mf{a}M_\infty)
\]
by exactness of $\Hom_{\mf{C}(\oo)}(P,-)$, and the action of $R_\infty /\mf{a}R_\infty$ on the right hand side factors through $\mb{T}_{\mf{m}}$. It follows that $R_\infty / \mf{a} R_\infty = R_{\Q,N} = \mb{T}_{\m}$, proving that $R_{\Q,N} = \mb{T}_{\m}$ and that both are complete intersections. Finally, the statement about the dimension of $R_{\Q,N}$ then follows as usual from the known values of $d$, $g$ and the dimension of $R_p$.
\end{proof}

We will now rewrite the isomorphism of Theorem \ref{emerton-style local-global} using the material from \S \ref{subsec: local considerations}. Let $\mf{X}_r$ denote the algebraization of the moduli stack of continuous $\Gamma_{\Q}$-representations with (semisimplified) reduction $r$. Explicitly, we just have $\mf{X}_r = [ \Spec R_{\Q,N} /\mu_2]$. We let $\mf{B}$ be the block corresponding to $r_p^{ss}$; we will then freely use the objects and notation established in \S \ref{subsec: local considerations}. Restriction to $\Gamma_{\Qp}$ gives us a morphism
\[
f : \mf{X}_r \to \mf{X}_{\mf{B}}
\] 
which factors through the algebraic stack $\mf{X}_{r_p} = [\Spec R_p /\mu_2]$. Our goal now is to show that 
\begin{equation}\label{eq: pullback of kernel sheaf}
f^!(X^\ast_{\mf{B}} \otimes_{\wt{E}_{\mf{B}}}P_{\mf{B}}) \cong (R_{\Q,N}\otimes_{R_p}P(1))[1],
\end{equation}
where, as usual, we view quasicoherent sheaves on $\mf{X}_r$ as $\Z/2$-graded $R_{\Q,N}$-modules. Let $g : \mf{X}_r \to \mf{X}_{r_p}$ be the restriction map and let $j : \mf{X}_{r_p} \to \mf{X}_{\mf{B}}$ denote the open immersion. Then $f^! = g^! \circ j^!$ and $j^! = j^\ast$. By Propositions \ref{pullback of universal module generic ps} and \ref{pullback of universal module non-gen II}, we have $j^\ast(X_{\mf{B}}^\ast \otimes_{\wt{E}_{\mf{B}}}P_{\mf{B}}) = P(1)$, so it remains to show that $g^!(P(1)) = (R_{\Q,N}\otimes_{R_p}P(1))[1]$. The map $g$ is the descent (modulo $\mu_2$) of the finite map $R_p \to R_{\Q,N}$, and the exceptional pullback $\D(R_p) \to \D(R_{\Q,N})$ is given by 
\[
C \mapsto \RHom_{R_p}(R_{\Q,N},C),
\]
as it is the right adjoint to pushforward. Note that $R_{\Q,N}$ is a perfect complex of $R_p$-modules since $R_{\Q,N}$ is a relative complete intersection over $R_p$\footnote{Indeed, it is the quotient of $R_\infty$ by a regular sequence. In particular, the corresponding Koszul complex is a finite resolution of $R_{\Q,N}$ by finite free $R_\infty$-modules, and hence by flat $R_p$-modules. It follows that the flat dimension of $R_{\Q,N}$ as an $R_p$-module is finite, and since $R_p$ is Noetherian this implies that $R_{\Q,N}$ has a finite resolution by finitely generated projective $R_p$-modules.}. Thus, the natural map
\[
C \otimes^L_{R_p} \RHom_{R_p}(R_{\Q,N},R_p) \to \RHom_{R_p}(R_{\Q,N},C)
\]
is an isomorphism (since both sides are exact functors that commute with filtered colimits, it is enough to check this for $C=R_p$). Since $R_p$ is a complete intersection of Krull dimension $4$, $R_p$ is a dualizing complex for $R_p$ and $R_p[-4]$ is a normalized dualizing complex. It then follows that $\RHom_{R_p}(R_{\Q,N},R_p[-4])$ is a normalized dualizing complex for $R_{\Q,N}$ \cite[\href{https://stacks.math.columbia.edu/tag/0AX1}{Tag 0AX1}]{stacks-project}. Since $R_{\Q,N}$ is a complete intersection of dimension $3$ we deduce that $\RHom_{R_p}(R_{\Q,N},R_p[-4])\cong R_{\Q,N}[-3]$ by uniqueness of normalized dualizing complexes (over Noetherian local rings), and hence that $\RHom_{R_p}(R_{\Q,N},R_p) \cong R_{\Q,N}[1]$. Thus the exceptional pullback along $R_p \to R_{\Q,N}$ is given by
\[
C \mapsto C \otimes^L_{R_p} R_{\Q,N}[1].
\]
Descending, it follows that $g^!(P(1)) = (R_{\Q,N}\otimes^L_{R_p}P(1))[1]$.

\begin{proposition}
We have $R_{\Q,N}\otimes^L_{R_p}P \cong R_{\Q,N}\otimes_{R_p}P$, i.e.\ $\Tor_i^{R_p}(R_{\Q,N},P)=0$ for $i\geq 1$. In particular, equation (\ref{eq: pullback of kernel sheaf}) holds.
\end{proposition}

\begin{proof}
If $r_p$ is not a twist of an extension of $\omega$ by $\mbf{1}$, then $P$ is a flat $R_p$-module; this follows from \cite[Cor.~3.12]{paskunas-image} (the formalism of \cite[\S 3]{paskunas-image} applies by \cite[Prop.~6.1 and 8.3]{paskunas-image}). In general, one may argue as follows. Recall that the completed tensor product on the category of compact $R_p$-modules has derived functors, which we will denote by $\mc{T}or_i^{R_p}(-,-)$ \cite{brumer}. We also denote the corresponding total derived functor by $-\wh{\otimes}_{R_p}^L-$, and we will use similar notation for other rings. Note that both $R_{\Q,N}$ and $P$ are compact $R_p$-modules. Since $R_{\Q,N}$ is finite over $R_p$ and $R_p$ is Noetherian, it follows that $\Tor_i^{R_p}(R_{\Q,N},P)= \mc{T}or_i^{R_p}(R_{\Q,N},P)$ for all $i$. In particular, it suffices to prove that $\mc{T}or_i^{R_p}(R_{\Q,N},P)=0$ for all $i\geq 1$.

To do this, we will use the notation and facts established in the proof of Proposition \ref{big R=T} freely. Since $R_\infty/\mf{a}R_\infty = R_{\Q,N}$ and $y_1,\dots,y_g$ is a regular sequence in $R_\infty$, the Koszul complex $K_\bu(y_1,\dots,y_g,R_\infty)$ is a resolution of $R_{\Q,N}$ by finite free $R_\infty$-modules, hence by pro-free $R_p$-modules, so $K_\bu(y_1,\dots,y_g,R_\infty \wh{\otimes}_{R_p} P)$ computes $R_{\Q,N}\wh{\otimes}^L_{R_p}P$. But $K_\bu(y_1,\dots,y_g,R_\infty \wh{\otimes}_{R_p} P)$ also computes $(\oo_\infty/\mf{a})\wh{\otimes}_{\oo_\infty}^L (R_\infty \wh{\otimes}_{R_p}P)$, since $y_1,\dots,y_g$ is a regular sequence in $\oo_\infty$. We have $R_\infty \wh{\otimes}_{R_p}P \cong M_\infty$ and we know that $M_\infty$ is a finite free $\oo_\infty \lb K_p \rb$-module for small $K_p$, hence a pro-free $\oo_\infty$-module. It follows that $\mc{T}or_i^{\oo_\infty}(\oo_\infty/\mf{a},M_\infty) = 0$ for $i\geq 1$, which finishes the proof.
\end{proof}

We can now prove the following local-global formula, which is modelled on the statement of \cite[Conj. 4.7.9]{zhu-coherent} and \cite[Exp.~Thm.~9.4.2]{egh}. Unsurprisingly, our proof is also similar to the proof sketched in \cite{egh}. Recall the functor $F_{ext}$ from \S \ref{subsec: local considerations}.

\begin{theorem}\label{local-global formula}
Let $\sigma$ be a complex of left $\oo \lb G^{ad} \rb$-modules.
\begin{enumerate}
\item We have $f^!(F_{ext}(\sigma)) \cong (R_{\Q,N} \otimes_{R_p} P \otimes_{\oo \lb G^{ad} \rb}^L \sigma)(1)[1]$ in $D_{qcoh}(\mf{X}_r)$, functorially in $\sigma$.

\smallskip

\item We have $C_\bu(N,\sigma)_{\mf{m}} \cong R\Gamma(\mf{X}_r,r^{univ}(1) \otimes_{R_{\Q,N}} f^!(F_{ext}(\sigma))[-2])$ in $D(R_{\Q,N})$, functorially in $\sigma$.

\smallskip

\item If $\sigma \in D(\Mod_{G^{ad}}^{lfin}(\oo)_{\mf{B}})$, then we have $C_\bu(N,\sigma)_{\mf{m}} \cong R\Gamma(\mf{X}_r,r^{univ}(1) \otimes_{R_{\Q,N}} f^!(F_{disc}(\sigma))[-2])$ in $D(R_{\Q,N})$, functorially in $\sigma$.
\end{enumerate}
\end{theorem}

Here we use the notation $C_\bu(N,\sigma)$ for $C_\bu(K^p_1(N),\sigma)$, and we clarify that $-(1)$ always denotes a grading shift (and never a Tate twist). We also remark that $r^{univ}(1)$ is the universal representation on $\mf{X}_r$.

\begin{proof}
We start with part (1). By definition, we have $F_{ext}(\sigma) = (X_{\mf{B}}^\ast \otimes_{\wt{E}_{\mf{B}}} P_{\mf{B}})\otimes_{\oo \lb G^{ad} \rb }^L \sigma$ and by our calculations in this subsection, we have $f^!(\mc{F}) = f^\ast(\mc{F})[1]$ for $\mc{F} \in D_{qcoh}(\mf{X}_{\mf{B}})$. It follows that 
\[
f^!(F_{ext}(\sigma)) \cong f^!(X_{\mf{B}}^\ast \otimes_{\wt{E}_{\mf{B}}} P_{\mf{B}}) \otimes_{\oo \lb G^{ad} \rb }^L \sigma   \cong (R_{\Q,N} \otimes_{R_p} P \otimes_{\oo \lb G^{ad} \rb}^L \sigma)(1)[1]
\]
as desired, using equation (\ref{eq: pullback of kernel sheaf}). For part (2), we have 
\[
C_\bu(N,\sigma)_{\mf{m}} \cong \wt{H}_1(N,\oo)_{\mf{m}}[-1] \otimes^L_{\oo \lb G^{ad} \rb} \sigma \cong r^{univ}(1) \otimes_{R_{\Q,N}} R_{\Q,N} \otimes_{R_p} P(1) [-1] \otimes^L_{\oo \lb G^{ad} \rb} \sigma
\]
by Proposition \ref{completed homology is universal}(1) and Theorem \ref{emerton-style local-global}. Part (2) then follows from part (1) (note that global sections of a quasicoherent sheaf on $\mf{X}_r$, i.e.\ a $\Z/2$-graded $R_{\Q,N}$-module, is just the grade $0$ part). Finally, part (3) follows from part (2) and Proposition \ref{LJ is an extension}.
\end{proof}

From Proposition \ref{shapiro}, we also get a formula for arithmetic homology. Using Poincar\'e duality, we can get a formula for arithmetic cohomology. It would be more canonical to formulate it using compactly supported cohomology, but since compactly supported cohomology agrees with usual cohomology after localization at $\mf{m}$ we can phrase it in terms of usual cohomology to avoid introducing extra notation.

\begin{corollary}\label{local-global cohomology}
Let $K_p \sub G^{ad}$ be a compact open subgroup and let $\tau$ be a left $\oo \lb K_p \rb$-module. Then we have 
\[
H^i(K_1^p(N)K_p,\tau)_{\mf{m}} \cong H_{-i}(\mf{X}_r,r^{univ}(1) \otimes_{R_{\Q,N}} f^!(F_{ext}(\oo \lb G^{ad} \rb \otimes_{\oo \lb K_p \rb} \tau))
\]
as $R_{\Q,N}$-modules for all $i$, and both sides vanish if $i\neq 1$.
\end{corollary}

\begin{proof}
Vanishing on the left hand side when $i\neq 1$ is well known, and vanishing on the right hand side follows from $\oo\lb K_p \rb$-flatness of $P$. The isomorphism then follows from Theorem \ref{local-global formula}(2) and the general form of Poincar\'e duality for local systems, see e.g.\ \cite[Thm.~III.3.11]{bellaiche-eigenbook}.
\end{proof}

\begin{remark}
Let $\sigma$ be a complex of left $\oo \lb G^{ad} \rb$-modules. Proposition \ref{completed homology is universal}(1) equips $C_\bu(N,\sigma)_{\m}$ (and hence $H_\ast(N,\sigma)_{\m}$) with a $\Gamma_{\Q}$-action, functorial in $\sigma$, via the action on $\wt{H}_1(N,\oo)$ (even though these are not even the homology of a manifold in general). With this $\Gamma_{\Q}$-action, the isomorphism in Theorem \ref{local-global formula}(2) is $\Gamma_\Q$-equivariant when the right hand side is given the $\Gamma_{\Q}$-action coming from $r^{univ}$. When $\sigma = \oo \lb G^{ad} \rb \otimes_{\oo \lb K_p \rb} \tau$ for some profinite $\oo \lb K_p \rb$-module, then this action agrees with the usual one defined via the Artin comparison isomorphism with \'etale homology (since the action on completed homology is defined via Artin comparison).
\end{remark}

\begin{remark}
We have elected to use $f^!$ instead of $f^\ast$ in our formulas to get the shifts to match up in Corollary \ref{local-global cohomology}, and because this is used in \cite[Conj.\ 9.3.2, Exp.~Thm.~9.4.2]{egh} and in \cite{zhu-coherent} (see e.g.\ Example 4.7.14 of \emph{loc. cit}).
\end{remark}

Theorem \ref{local-global formula} and Corollary \ref{local-global cohomology} have many interesting special cases, including $\sigma = \oo \lb G^{ad} \rb \otimes_{\oo \lb K_p \rb} (\Sym^{k-2}A^2)(\det)^{(2-k)/2}$ (or $\tau = (\Sym^{k-2}A^2)(\det)^{(2-k)/2}$), for $k\geq 2$ even\footnote{Since we have restricted ourselves to $\PGL_{2/\Q}$, we need to have $k$ even.} and $A$ any $\oo$-algebra. Other interesting cases involve taking $\sigma$ to be a representation corresponding to a two-dimensional mod $p$ or $p$-adic representation of $G_{\Qp}$ via the mod $p$ or $p$-adic local Langlands correspondence. We refer to \cite{tarrach2023parithmetic} for a direct approach in the mod $p$ situation, which does not use local-global compatibility for completed homology.

Finally, a different set of interesting coefficient systems are those appearing in the theory of eigenvarieties. We spell this out for the eigenvarieties constructed in \cite{hansen-eigenvarieties} (using locally analytic functions instead of distributions) and \cite{tarrach2022sarithmetic}. Consider the upper triangular Borel subgroup $B^{ad} \subseteq G^{ad}$ and its diagonal torus $T^{ad}$. We may then look at the moduli space $\ms{X}_{T^{ad}}$ of continuous characters of $T^{ad}$ over $L$, as in e.g.\ \cite[Lem.~6.11]{tarrach2022sarithmetic}. For every open affinoid $U \sub  \ms{X}_{T^{ad}}$, let $\ka_U : T^{ad} \to \oo(U)^\times$ denote the corresponding character and let $\Ind_{B^{ad}}^{G^{ad}}(\ka_U)^{la}$ be the locally analytic induction to $G^{ad}$. Tarrach shows that the assignment
\[
U \mapsto H_\ast(N, \Ind_{B^{ad}}^{G^{ad}}(\ka_U)^{la})
\]
defines a (graded) coherent sheaf $\mathcal{H}_\ast$ on $T^{ad}$, which agrees with the coherent sheaf on the eigencurve (implicitly) constructed using modules of locally analytic functions in \cite{hansen-eigenvarieties} (cf. \cite[\S 6.3-6.4]{tarrach2022sarithmetic}). We can then obtain the following:

\begin{corollary}\label{local-global eigenvarieties}
With notation as above, we have 
\[
H_i(N, \Ind_{B^{ad}}^{G^{ad}}(\ka_U)^{la})_{\mf{m}} \cong H_i(\mf{X}_r,r^{univ}(1) \otimes_{R_{\Q,N}} f^!(F_{ext}(\Ind_{B^{ad}}^{G^{ad}}(\ka_U)^{la}))[-2])
\]
for all $i$ and all affinoid $U$ (with both sides being $0$ unless $i=1$) as $R_{\Q,N}\otimes_{\oo}\oo(U)$-modules. Setting $\ms{M} = \mathrm{R}\varprojlim_U \Ind_{B^{ad}}^{G^{ad}}(\ka_U)^{la}$, the global sections of $\mathcal{H}_\ast$ are
\[
\varprojlim_U H_\ast(N, \Ind_{B^{ad}}^{G^{ad}}(\ka_U)^{la})_{\mf{m}} \cong H_\ast(N, \ms{M})_{\m} \cong H_\ast(\mf{X}_r,r^{univ}(1) \otimes_{R_{\Q,N}} f^!(F_{ext}(\ms{M}))[-2]),
\]
viewed as an $R_{\Q,N}\otimes_{\oo}\oo(\ms{X}_{T^{ad}})$-module.
\end{corollary}

\begin{proof}
We need to prove the isomorphism $\varprojlim_U H_\ast(N, \Ind_{B^{ad}}^{G^{ad}}(\ka_U)^{la})_{\mf{m}} \cong H_\ast(N, \ms{M})_{\m}$; the rest follows from Theorem \ref{local-global formula}. It suffices to prove this before localizing at $\mf{m}$. For simplicity, set $\ms{M}_U = \Ind_{B^{ad}}^{G^{ad}}(\ka_U)^{la}$.   The complex $C_\bu(\mc{X}_p)\otimes_{\oo}L$ is a perfect complex of $L[G^{ad}]$-modules by \cite[6.2 Thm]{borel-serre-II} (this shows that $C_\bu(\mc{X}_p)$ is a perfect complex of $\oo[G^{ad}]$-modules for sufficiently small $(K^p)' \triangleleft K^p$; taking $K^p/(K^p)'$-coinvariants gives the statement we need). So $C_\bu(N,\ms{M}) = \mathrm{R}\varprojlim_U C_\bu(N,\ms{M}_U)$. Since $H_\ast(N,\ms{M}_U)$ form a coherent sheaf on $\ms{X}_{T^{ad}}$, we have $\mathrm{R}^i\varprojlim_U H_\ast(N,\ms{M}_U) =0$ for all $i \geq 1$, and the hypercohomology spectral sequence then gives the desired isomorphism $\varprojlim_U H_\ast(N, \ms{M}_U) \cong H_\ast(N, \ms{M})$. 
\end{proof}

\begin{remark}
This can be viewed as a version of \cite[Conj. 9.6.27]{egh} for $G^{ad}$ (and after localizing at $\mf{m}$). With a few extra arguments (using \cite{pan2022note}), we expect that one can upgrade this to an isomorphism of $\oo(\Spf(R_{\Q,N})^{rig})\wh{\otimes}_{L}\oo(\ms{X}_{T^{ad}})$-modules. Moreover, we expect that the representation $\ms{M}$ can be computed more explicitly in terms of the universal character of $\ms{X}_{T^{ad}}$.
\end{remark}

\begin{remark}
We have used \cite[Thm.~7.4]{ceggps1} as the basis for our results here, but one could also use the local-global compatibility results of \cite{emerton-lg} instead, and we expect that similar arguments to the above would prove a different version of Theorem \ref{local-global formula} that allows for non-minimal ramification at places dividing $N$. The reason that we do not carry this out here is that the main extra work, compared to what we have done here, would be at the places $\ell \mid N$, which is orthogonal to the main subject of this paper. In that case, as remarked in \cite[Rem.~7.2]{ceggps2}, one should work with infinite level at places dividing $N$ as well, and consider the universal deformation ring $R_{\Q,S}^{univ}$ for all continuous $G_{\Q,S}$-deformations of $r$ ($S$ is the set of places dividing $Np$). On the automorphic side, this means looking at $S$-arithmetic homology, and using coefficient systems that are (external) tensor products of $\oo \lb G^{ad} \rb$-modules with modules for the Hecke algebras $\mc{H}_{\ell}^\infty$ of compactly supported locally constant functions on $\PGL_2(\Q_\ell)$, for all $\ell \mid N$. The universal $S$-arithmetic homology group for these coefficient systems, in the sense of generalizing Proposition \ref{completed homology is universal}, is completed homology with infinite level at primes $\ell \mid N$ as well.
\end{remark}

\bibliographystyle{alpha}
\bibliography{SL2bib}

\end{document}